\documentclass[10pt]{amsart}

\usepackage{amssymb}
\usepackage{verbatim} 
\usepackage[colorlinks=true, citecolor=black, filecolor=black, linkcolor=black, urlcolor=black]{hyperref}
\usepackage{graphicx}
\usepackage{enumerate}
\usepackage[table,xcdraw]{xcolor}
\usepackage{caption}
\usepackage{todonotes}
\usepackage{etoolbox}
\include{toc}

\def\bfs{\boldsymbol}

\newcommand{\C}{\mathbb{C}}
\newcommand{\Cx}{\widehat{\C}}
\newcommand{\R}{\mathbb{R}}

\renewcommand{\H}{\mathbb{H}}

\renewcommand{\Re}{\operatorname{Re}}
\renewcommand{\Im}{\operatorname{Im}}

\newcommand{\Rd}{\R_d}

\newcommand{\metg}{\textsl{g}}

\newcommand{\mb}[1]{\mathbf{#1}}
\newcommand{\mc}[1]{\mathcal{#1}}

\newcommand{\dell}{\partial}

\newcommand{\supp}{\operatorname{supp}}

\newcommand{\hcap}{\operatorname{hcap}}
\newcommand{\pre}{\textrm{pre}}

\newcommand{\LieL}{\mathcal{L}}

\newcommand{\CRR}{\operatorname{CRR}}
\newcommand{\PSL}{\operatorname{PSL}}

\newcommand{\NV}{\mathrm{NV}}
\newcommand{\Ham}{\mc{H}}

\newcommand{\1}[1]{\mathbf{1} \left \{ #1 \right \}}
\newcommand{\diag}{\operatorname{diag}}

\theoremstyle{plain}
\numberwithin{equation}{section}
\newtheorem{thm}{Theorem}[section]
\newtheorem{lem}[thm]{Lemma}
\newtheorem{cor}[thm]{Corollary}
\newtheorem{prop}[thm]{Proposition}

\theoremstyle{definition}
\newtheorem{eg}{Example}[section]

\newtheorem*{eg*}{Example}
\newtheorem*{egs*}{Examples}
\newtheorem*{def*}{Definition}

\theoremstyle{remark}
\newtheorem*{rmk*}{Remark}
\newtheorem*{rmks*}{Remarks}
\newtheorem*{q*}{Question}

\begin{document}

\title{Pole Dynamics and an Integral of Motion for Multiple SLE(0)}

\author{Tom Alberts}
\address{Department of Mathematics, University of Utah, Salt Lake City, UT 84112, USA} 
\email{alberts@math.utah.edu}
\thanks{Tom Alberts was supported by NSF grants DMS-1811087 and DMS-1715680.}

\author{Sung-Soo Byun}
\address{School of Mathematics, Korea Institute for Advanced Study, Seoul, 02455, Republic of Korea} 
\email{sungsoobyun@kias.re.kr}

\author{Nam-Gyu Kang} 
\address{School of Mathematics, Korea Institute for Advanced Study, Seoul, 02455, Republic of Korea} 
\email{namgyu@kias.re.kr} 
\thanks{Nam-Gyu Kang was partially supported by a KIAS Individual Grant (MG058103) at Korea Institute for Advanced Study. 
Sung-Soo Byun and Nam-Gyu Kang were partially supported by Samsung Science and Technology Foundation (SSTF-BA1401-51), and National Research Foundation of Korea under grant number NRF-2019R1A5A1028324.}

\author{Nikolai~G.~Makarov}
\address{Department of Mathematics, California Institute of Technology, Pasadena, CA 91125, USA} 
\email{makarov@caltech.edu}
\thanks{Nikolai~G.~Makarov was supported by NSF grant no. 1500821.}


\begin{abstract}
We describe the Loewner chains of the real locus of a class of real rational functions whose critical points are on the real line. Our main result is that the poles of the rational function lead to explicit formulas for the dynamical system that governs the driving functions. Our formulas give a simple method for mapping the class of rational functions into solutions to a non-trivial system of quadratic equations, and for directly showing that the curves in the real locus satisfy \textit{geometric commutation} and have the \textit{geodesic multichord} property. These results are entirely self-contained and have no reliance on probabilistic objects, but make use of an integral of motion for the Loewner chain that is motivated by ideas from conformal field theory. We also show that the dynamics of the driving functions are a special case of the Calogero-Moser integrable system, restricted to a particular submanifold of phase space carved out by the Lax matrix. Our approach complements a recent result of Peltola and Wang, who showed that the real locus is the deterministic $\kappa \to 0$ limit of the multiple SLE$(\kappa)$ curves.
\end{abstract}

\vspace*{-1.3cm}
\maketitle
\vspace*{-1cm}
\tableofcontents

\section{Introduction \label{sec: intro}}

Recently, Peltola and Wang \cite{PW20} introduced and made a detailed study of \textbf{multiple SLE}$(0)$: a certain ensemble of $n$ simple curves in a simply connected subset of the Riemann sphere that connect $2n$ specified boundary points according to a non-crossing link pattern. Although the embedding of the curves is purely deterministic and fully determined by the domain, the boundary points, and the chosen link pattern, the name multiple SLE$(0)$ comes from the fact (which Peltola and Wang prove) that the ensemble of curves is the deterministic limit of the \textit{random} \textbf{multiple SLE}$(\kappa)$ process \cite{BBK05, Dubedat07, LK:configurational_measure} as $\kappa \to 0$. As a result the multiple SLE$(0)$ curves have many properties that are natural analogues of their multiple SLE$(\kappa)$ counterparts: \textbf{conformal invariance}, a description via \textbf{Loewner flow}, and an important geometric characterization called the \textbf{geodesic multichord property}. However, Peltola and Wang also show that the multiple SLE$(0)$ curves have a property which has no analogue in the random world: the curves are exactly the non-trivial part of the \textbf{real locus} of a \textbf{real rational function} whose critical points are all on the line. In particular this implies that the SLE$(0)$ curves are solutions to \textit{algebraic} equations, which is an impossible feat for their highly fractal SLE$(\kappa)$ cousins. 

In this paper we exploit the connection with rational functions to give two new descriptions of the Loewner chains that grow multiple SLE$(0)$ curves. We give explicit, algebraic formulas for the dynamical system that governs the \textbf{driving functions} of the Loewner chain. We express the dynamical system in two ways: as a first order evolution for the \textbf{poles} and \textbf{critical points} of the rational function, and as a second order evolution that only involves the critical points. The first order evolution is a natural extension of the well-known \textbf{SLE}$(\kappa, \bfs \rho)$ dynamics to the case $\kappa = 0$, while the second order evolution is a special case of the famous \textbf{Calogero-Moser} system. SLE$(\kappa, \bfs \rho)$ processes were introduced in \cite{LSW:restriction} to study boundaries of a particular family of random two-dimensional sets and are now ubiquitous in the theory of conformally invariant random systems, while Calogero-Moser \cite{Calogero:quantum_CM_integrable, Moser:CM_integrable} is a famous Hamiltonian system that describes the motion of particles interacting on the line via an inverse square pairwise potential. It is well known for being \textbf{integrable} with an explicit \textbf{Lax pair}, and to the best of our knowledge this paper is the first to prove a connection between Calogero-Moser dynamics and the Loewner evolution of SLE processes. 

We commonly describe the curves in this paper as multiple SLE$(0)$s, but we emphasize that our results are entirely self-contained within the theory of rational functions and Loewner evolution. We do not rely on properties of multiple SLE$(\kappa)$ or other tools coming from the theory of two-dimensional conformally invariant random systems. This is in contrast to the Peltola and Wang approach that uses the \textbf{Brownian loop measure} to describe the Loewner flow of the multiple SLE$(0)$ curves. While beautiful, computations involving Brownian loop measure can be involved or impractical. Our explicit formulas for the driving function's dynamical system allow us to prove several properties of multiple SLE$(0)$ curves in a direct manner. We show that each individual curve in the ensemble is an SLE$(0, \bfs \rho)$, where the $\bfs \rho$ term encodes the locations of the poles and critical points of the rational function and how they influence the evolution of the Loewner chain's driving function. See Theorem \ref{thm: real_locus} and equation \eqref{eq: U} for the explicit formula. Moreover we show that this family of Loewner chains (one for the growth of each curve in the ensemble) satisfies a \textbf{commutation} property. Loosely speaking, this means that we can grow the curves by starting and stopping their individual Loewner chains in any desired order and yet the hull that is produced is always a subset of a single fixed set. We also show, in just a few short lines, that our formulas for the first order driving function dynamics are solutions to the system of quadratic equations that we call the \textbf{null vector equations}. This system is also the $\kappa \to 0$ limit of the \textbf{Belavin-Polyakov-Zamolodchikov (BPZ)} equations that appear in conformal field theory. We also directly show that the real locus of a real rational function with real critical points satisfies the geodesic multichord property. In a sense our argument is a converse to the approach of Peltola and Wang, who use a Schwarz reflection argument to show that an ensemble of curves having the geodesic multichord property (and satisfying mild additional assumptions) must be the real locus of a certain type of real rational function with real critical points. A crucial component underlying all of these results is a new \textbf{field of integral of motion} for SLE$(0, \bfs \rho)$ processes, meaning that we have a preserved quantity for each $z \in \H$. 

The structure of this paper is as follows. A detailed description of our results is given in Section \ref{sec: results}. We also introduce the SLE$(0, \bfs \rho)$ processes and explain how they are a natural extension of the SLE$(\kappa, \bfs \rho)$ processes for $\kappa > 0$. In Section \ref{sec:IoM} we prove the field integral of motion property for SLE$(0, \bfs \rho)$ processes, and explain the condition on $\bfs \rho$ that gives conformal invariance of the family of processes. Section \ref{sec: RRF_Loewner} establishes an important algebraic relationship between poles and critical points that characterizes the class of rational functions under consideration. We refer to this relationship as the \textbf{stationary relation}. In the dynamical context of Loewner evolution and Calogero-Moser the stationary relation translates into a requirement of very specific initial conditions for the particles. It also suggests a new approach for enumerating equivalence classes (under post-composition by M\"{o}bius transforms) of real rational functions with prescribed critical points on the real line. Section \ref{sec: RRF_Loewner} also uses the field integral of motion from Section \ref{sec:IoM} to prove that our Loewner chains generate the real locus of the real rational function associated to a solution of the stationary relation. In Section \ref{sec: RRF_Loewner} we also prove the geodesic multichord property for the real locus, and that the formulas for our driving function dynamics are solutions to the quadratic null vector equations. In Section \ref{sec:locus_geodesic} we briefly recall existing theory that associates a vector field to each rational function and describes the real locus as the flow lines of this vector field. Although this theory is well known, we include it as a contrast to our new results on Loewner chains for the real locus.

Section \ref{sec: commutation} is essentially separate from the rest of the paper. It gives a self-contained explanation for the appearance of the null vector equations in this theory that does not involve $\kappa \to 0$ limits of the BPZ equations. The main result is that the null vector equations are a necessary algebraic relationship imposed on the dynamics of each individual curve's driving function in order for the entire system of curves to satisfy the commutation property and be conformally invariant. In this sense Section \ref{sec: commutation} is an intrinsic geometric characterization of the null vector equations. Section \ref{sec: CM} establishes the connection between our Loewner chains and the Calogero-Moser system. We show that the Calogero-Moser dynamics have to be restricted to a particular submanifold of phase space in order to match the first order dynamics of the Loewner driving function. This submanifold is carved out by a formula that expresses the null vector equations in terms of the Lax pair of the Calogero-Moser system. We also use the Calogero-Moser description to derive a partial differential equation for the evolution of a rational function under our Loewner evolution.

In Section \ref{sec: examples} we give examples of real rational functions and the formulas for our driving function dynamics in the cases $n=1,$ $n=2$, and $n=3$, i.e. two, four,  and six marked points on the real line, respectively. In the $n=1$ case this is simple since there is only one way of connecting the two boundary points. For $n=2$, in which there are two ways to connect the four boundary points, we use our results to find explicit formulas for the solutions to the null vector equations and for the curves that make up the two multiple SLE$(0)$ ensembles. We also show that our formulas agree with the $\kappa \to 0$ limits of the so-called \textit{pure} solutions to the BPZ equations, which are explicitly known for $n=2$.  For $n=3$, there are five link patterns, of which we discuss three symmetric patterns.

Finally, in Section \ref{sec: conclusion} we explain how our results are motivated by heuristics coming from conformal field theory. Slight modifications of the same heuristic explain both the new formulas for the Loewner driving function dynamics that are the cornerstone of this paper, and the classical description of the real locus as the flow lines of a vector field. We also use this heuristic to explain (non-rigorously) how the classical flow line description is a $\kappa \to 0$ limit of the Miller-Sheffield \textbf{imaginary geometry}  \cite{MS16:imaginary1, MS16:imaginary2, MS16:imaginary3, MS16:imaginary4}.

\vspace{.1in}
\noindent
\textbf{Acknowledgements: } We thank Eveliina Peltola and Yilin Wang for several helpful discussions about their paper \cite{PW20}. We also thank Alexander Abanov, Ilya Gruzberg, and Pavel Wiegmann for pointing out the connection to \cite{ABW09, AGK11}. This material is based upon work supported by the National Science Foundation under Grant No. DMS-1928930, while the authors participated in a program hosted by the Mathematical Sciences Research Institute in Berkeley, California, during the Spring 2022 semester.

\section{Formulation of Main Results \label{sec: results}}

We begin this section by briefly reviewing the basic properties of multiple SLE$(0)$ proved in \cite{PW20} and the relation with multiple SLE$(\kappa)$. We then state our main results, first by describing the Loewner chains for the real locus of critically real rational functions. The results are stated solely in the upper half-plane $\H$ equipped with the standard coordinate charts, and later we show that all results hold equally well in the coordinate charts on $\H$ induced by the real M\"{o}bius transformations of $\H$ to itself. Throughout we consider $n \geq 1$ and points $\bfs x = \{x_1, \ldots, x_{2n} \}$ that are real and distinct. A link pattern for $\bfs x$ is a pairing
\[
\alpha = \{ \{ a_1, b_1 \}, \ldots, \{ a_n, b_n \} \}, 
\]
such that each $a_i, b_i$ is an element of $\bfs x$, every element of $\bfs x$ appears exactly once, and there exists $n$ continuous, non-intersecting curves in $\H$ connecting the endpoints of the pairs. More precisely, the latter means that there exists continuous $\eta_i : [0,1] \to \overline{\H}$ with $\eta_i[0,1] \cap \eta_j[0,1] = \emptyset$, $\eta_i(0,1) \subset \H$, and $\eta_i(0) = a_i, \eta_i(1) = b_i$. For each $\bfs x$ the number of link patterns is given by the Catalan number
\[
C_n = \frac{1}{n+1} \dbinom{2n}{n}.
\]
We conclude the section by describing our algebraic solutions to the null vector equations and the connection with the Calogero-Moser system.

\subsection{\texorpdfstring{Review: Multiple SLE$(0)$ from Multiple SLE$(\kappa)$}{Review: Multiple SLE(0) from Multiple SLE(kappa)}}
Each multiple SLE$(0)$ in $\H$ is uniquely determined by $2n$ distinct real points $\bfs x$ and one of the $C_n$ link patterns $\alpha$ that specifies how the points in $\bfs x$ are connected. We use multiple SLE$(0; \bfs x; \alpha)$ to denote the ensemble of $n$ curves $\bfs \eta = (\eta_1, \ldots, \eta_n)$ in $\H$ corresponding to the data $(\bfs x; \alpha)$. We also use multiple SLE$(\kappa; \bfs x; \alpha)$ to denote the law of the $n$ random SLE$(\kappa)$ curves that almost surely connect the boundary points $\bfs x$ according to the link pattern $\alpha$. For $0 < \kappa \leq 4$ there are several existence arguments for these laws \cite{LK:configurational_measure, Lawler:part_functions_SLE, Lawler:PC, PelWu:multiple_SLEs} (see also \cite{Wu:hypergeometric} for $4 < \kappa \leq 6$), and in \cite{BPW:uniqueness} it has been shown that the law is uniquely characterized by a certain conditional law/resampling property. Loewner chains for multiple SLE$(\kappa)$ curves were studied in \cite{BBK05, Dubedat07, Graham:multiple_SLE}, the approach being that the driving function dynamics are determined by a collection of so-called partition functions that encode the geometry of the domain one is growing in and the location of the growth points. There is an associated literature on solutions to the system of PDEs (the BPZ equations) satisfied by the partition functions \cite{Dubedat06, FK1, FK2, FK3, FK4, KyPel:pure_partitions, JL:smoothness}. We also study multiple SLE$(\kappa)$ in the forthcoming \cite{AKM}, through the lens of a Gaussian conformal field theory introduced in \cite{KM13,KM}. The work in \cite{AKM} motivates the approach of the present work and we explain more of our intuition in the concluding remarks. For $\kappa = 0$, Peltola and Wang show that each multiple SLE$(0; \bfs x; \alpha)$ ensemble has the following properties:
\begin{itemize}
 \item it is the deterministic limit of multiple SLE$(\kappa; \bfs x; \alpha)$ as $\kappa \to 0$,
 \item it is the unique minimizer of a certain functional (called the \textbf{multichordal Loewner energy}) on ensembles of $n$ curves that connect the points in $\bfs x$ according to the pattern $\alpha$,
 \item it has the \textbf{geodesic multichord} property, meaning that each $\eta_j$ is the hyperbolic geodesic in the connected component of $\H \backslash \bigcup_{k \neq j} \eta_{k}$ whose boundary contains the endpoints of $\eta_j$, 
 \item each curve $\eta_j$ can be individually generated by a Loewner chain, where the driving term for $\eta_j$ is determined by solutions that Peltola and Wang construct (using the Brownian loop measure of the curves) for the classical limit of the so-called \textbf{Belavin-Polyakov-Zamolodchikov (BPZ)} equations as $\kappa \to 0$, and finally,
 \item when regarded as a subset of the Riemann sphere $\widehat{\C}$, $\bfs \eta \cup \overline{\bfs \eta} \cup \R$ is the real locus of a real rational function of degree $n+1$, whose critical points are precisely at $x_1, \ldots, x_{2n}$.
\end{itemize}
The first four properties, although difficult to prove rigorously, are quite natural from heuristic considerations coming from the multiple SLE$(\kappa)$ theory. Indeed, the first and second properties come from a small $\kappa$ large deviations principle which, prior to \cite{PW20}, was known to hold for a \textit{single} chordal SLE$(\kappa)$ curve \cite{Wang:LE}. In the single curve case the limiting chordal SLE$(0)$ curve is the hyperbolic geodesic connecting two boundary points in the half-plane. The geodesic multichord property is a manifestation of this single curve geodesic property and the resampling property of multiple SLE$(\kappa)$ ensembles: that each random curve, conditionally on all of the others, is an SLE in the connected component that it lives in after all of the other curves are removed. And for the fourth property, one would certainly expect that the Loewner chain description of the multiple SLE$(\kappa)$ curves should certainly carry over to the $\kappa = 0$ case, and likely even be less technically cumbersome owing to the smoothness of the curves when compared to their multiple SLE$(\kappa)$ counterparts.

In contrast, what does \textit{not} appear so naturally in the multiple SLE$(\kappa)$ theory is a connection with rational functions. As mentioned earlier this property implies that the points on the SLE$(0)$ curves are solutions to algebraic equations, which is certainly impossible for their fractal multiple SLE$(\kappa)$ counterparts. In the final section we explain how, at least heuristically, one may suspect the appearance of rational functions in a $\kappa \to 0$ limit of ``rational type'' functions that appear in conformal field theory. These heuristics are an underlying motivation for the following main results.

\subsection{\texorpdfstring{Real Rational Functions and SLE$(0, \bfs \rho)$}{Real Rational Functions and SLE(0, rho)}}\label{subsec:RRF_SLE}
Now we begin our description of the Loewner chains for the real locus of a class of real rational functions. We emphasize that our approach is fully self-contained and does not rely on multiple SLE$(\kappa)$ (for $\kappa = 0$ or $\kappa > 0$) nor on \cite{PW20}. To describe our results we first recall some basic facts about real rational functions and their real locus. Recall that any real rational function $R : \widehat{\C} \to \widehat{\C}$ of degree $n+1$ can be written as $R = P/Q$, where $P$ and $Q$ are monic polynomials with real coefficients, no common factors, and $\max \{ \deg P, \deg Q \} = n+1$. Any such $R$ is a branched covering of the Riemann sphere by itself, and the degree $n+1$ is the number of pre-images of any regular value. The branch points of the covering are precisely the critical points of $R$, i.e. the zero set of
\[
R' = \frac{P'Q - Q'P}{Q^2},
\]
considered as a multiset in the case of multiplicities (which we do not study in this paper). We use the following notation for the set of real rational functions with prescribed critical points on the line. 

\begin{def*}\label{def: CRR}
For $\bfs x = \{x_1, \ldots, x_{2n}\}$ with $x_j$ distinct, real, and finite, let $\CRR_{n+1}(\bfs x)$ be the set of real rational functions of degree $n+1$ whose critical points are precisely $\bfs x$. For any $R \in \CRR_{n+1}(\bfs x)$ we let $\Gamma(R)$ denote its real locus, i.e. the set
\[
\Gamma(R) = \left \{ z \in \widehat{\C} : R(z) \in \widehat{\R} \right \},
\]
where $\widehat{\R}$ is the one-point compactification of $\R$.
\end{def*}

$\CRR$ stands for \textit{critically real, real rational} functions. By this we mean that the rational function has real coefficients and that additionally all of its critical points are real. Natural operations on such functions are pre- and post-composition by elements of $\PSL(2,\R)$, the group of conformal automorphisms (M\"{o}bius transformations) of $\H$ to itself. It is straightforward to verify that the pre-composition of an element $R \in \CRR_{n+1}(\bfs x)$ by $\phi \in \PSL(2, \R)$ satisfies $R \circ \phi^{-1} \in \CRR_{n+1}(\phi(\bfs x))$ and $\Gamma(R \circ \phi^{-1}) = \phi(\Gamma(R))$. Similarly, one can check that the post-composition of $\phi$ by $R$ satisfies
\[
\phi \circ R \in \CRR_{n+1}(\bfs x), \quad \Gamma(\phi \circ R) = \Gamma(R). 
\]
Consequently, post-composition induces a natural equivalence relation on $\CRR_{n+1}(\bfs x)$. Goldberg \cite{Goldberg91} proved that for every $\bfs x$ there is at least one equivalence class in $\CRR_{n+1}(\bfs x)$ but no more than $C_n$ distinct classes. Later, \cite{EG02, MTV:Shapiro, EG11, PW20} all independently prove the matching lower bound of at least $C_n$ distinct classes. Therefore as $R$ varies over $\CRR_{n+1}(\bfs x)$ there are exactly $C_n$ distinct realizations of $\Gamma(R)$. The next lemma describes the structure of the real locus $\Gamma(R)$ for any $R \in \CRR_{n+1}(\bfs x)$, although we note that it is elementary and only relies on properties of $\CRR_{n+1}(\bfs x)$ (in particular it does not use any of \cite{Goldberg91, EG02, MTV:Shapiro, EG11, PW20}). 

\begin{lem}\label{lem: RRF_locus}
Let $\bfs x = \{ x_1, \ldots, x_{2n} \}$ be real and distinct and assume $R \in \CRR_{n+1}(\bfs x)$. Then $\Gamma(R)$ consists of the real line, $n$ non-crossing simple curves in $\H$ that connect the points in $\bfs x$ according to some non-crossing link pattern, and the complex conjugates of these curves.
\end{lem}

Here is a sketch: first appeal to the Riemann-Hurwitz formula, which implies that a rational function of degree $n+1$ has $2n$ critical points if and only if the critical points are all of index two. Thus the rational function is locally a two-to-one branched cover around each critical point, i.e.
\[
R(z) = R(x_i) + c_i(z - x_i)^2 + o((z-x_i)^3), \quad c_i = \frac12 R''(x_i) \in \R, \, c_i \neq 0.
\]
Therefore around each critical point the real locus locally consists only of the real line and a curve that passes through the critical point at a right angle. Since the $x_i$ are the only branch points the curves emanating from them cannot cross each other, and the real locus being path-connected forces that the $2n$ curves must pair up and connect according to some non-crossing link pattern. 

It is often useful to regard a set with the structure of $\Gamma(R)$ as a graph embedded in $\widehat{\C}$. The critical points $x_i$ are the vertices, the arcs connecting them are the edges (including the intervals on the real line), and connected components of the complement of $\Gamma(R)$ are the faces. The geometric structure of $\Gamma(R)$ suggests that one should be able to grow the $n$ curves via a Loewner chain, with the growth occurring either separately or in tandem. In the latter case it is better to think of the ensemble consisting of $2n$ curves, one growing from each of the critical points, but paired according to some link pattern. To describe the dynamics of the driving function for a single curve it turns out to be convenient to adopt the language of SLE$(\kappa, \bfs \rho)$ processes, specialized to the case $\kappa = 0$. 

\begin{def*}
Let $x \in \R$ and $\bfs \rho = \sum_i \rho_i \delta_{z_i}$ be a finite atomic measure on $\Cx$ that is symmetric under conjugation, i.e. $\bfs \rho(z) = \bfs \rho(\overline{z})$ for all $z$. Define the chordal SLE$(0, x, \bfs \rho)$ Loewner chain by
\begin{align}\label{eq: g}
\partial_t g_t(z) = \frac{2}{g_t(z) - x_t}, \quad g_0(z) = z,
\end{align}
where the driving function $x_t$ evolves as
\begin{align}\label{eq: zero_rho_driving}
\dot{x}_t = \int_{\C} \frac{d \bfs \rho(w)}{x_t - g_t(w)}, \quad x_0 = x.
\end{align}
The flow map $g_t$ is well-defined up until the first time $\tau$ at which $x_t = g_t(w)$ for some $w$ in the support of $\bfs \rho$, while for each $z \in \C$ the process $t \mapsto g_t(z)$ is well-defined up until $\tau_z \wedge \tau$, where $\tau_z$ is the first time at which $g_t(z) = x_t$. We let $K_t = \{ z \in \overline{\H} : \tau_z \leq t \}$ be the hull associated to this Loewner chain. We use the notation SLE$(0, x, \bfs \rho)$ when we want to emphasize the initial position $x$ of the driving function, but when the initial position is less important we generically refer to these systems as SLE$(0, \bfs \rho)$.
\end{def*}

We explain the connection with SLE$(\kappa, \bfs \rho)$ processes in Section \ref{sec:IoM}, although essentially it only involves adding a stochastic term $\sqrt{\kappa} \, dB_t$ to the evolution equation \eqref{eq: zero_rho_driving}. For our main theorem it is beneficial to also consider time-varying superpositions of the SLE$(0, \bfs \rho)$ dynamics. The purpose of these superpositions is to allow for the corresponding hull to grow from several different locations at once, which is useful for analyzing commutation properties of an ensemble of curves. 

\begin{def*}
Let $x_1, \ldots, x_N$ be real and distinct, and $\bfs \rho_1, \ldots, \bfs \rho_N$ be finite atomic measures on $\Cx$ that are symmetric under conjugation. Let $\bfs \nu = (\nu_1, \ldots, \nu_N)$, where each $\nu_i : [0, \infty) \to [0, \infty)$ is assumed to be measurable. Then define the $\bfs \nu$-superposition of SLE$(0, x_j, \bfs \rho_j)$ to be the Loewner chain $(g_t)_{t \geq 0}$ obtained by superimposing the weighted dynamics of each individual SLE$(0, x_j, \bfs \rho_j)$ Loewner chain, with the weight of each chain at time $t$ being $\nu_j(t)$. More precisely, both the flow on the Riemann sphere and the flows for the driving functions are superimposed, so that \eqref{eq: g} is replaced by
\begin{align}\label{eq: g2}
\partial_t g_t(z) = \sum_{j=1}^N \frac{2 \nu_j(t)}{g_t(z) - x_j(t)}, \quad g_0(z) = z,
\end{align}
and the driving functions $x_j(t)$, $j=1,\ldots,N$, evolve as
\begin{align}\label{eq: x2}
\dot{x}_j = \nu_j(t) \int_{\C} \frac{d \bfs \rho_{j,t}(w)}{x_j - g_t(w)} + \sum_{k \neq j} \frac{2 \nu_k(t)}{x_j - x_k}, \quad x_j(0) = x_j.
\end{align}
The measure $\bfs \rho_{j,t}$ is defined by
\[
\bfs \rho_{j,t} = \bfs \rho_j |_{\C \backslash \{ \bfs x \}}  +  \sum_{k \neq j} \bfs \rho_j(x_k(0)) \delta_{\gamma_k(t)},
\]
where $\bfs x = \{ x_1, \ldots, x_N \}$ and $\gamma_1(t), \ldots, \gamma_N(t)$ are the tips of the $N$ individual curves at time $t$ (these are well-defined since the Loewner chain is generated by a curve, see Proposition \ref{prop: SLE_zero_rho_curve}). For each $z$ the process $t \mapsto g_t(z)$ is well-defined up to a time $\tau_z \wedge \tau$, where $\tau$ is the first time at which any of the individual SLE$(0, x_j, \bfs \rho_j)$ stops being well-defined or any two $x_j$ collide. We continue to let $K_t = \{z \in \overline{\H} : \tau_z \leq t \}$ be the hull generated by this Loewner chain. 
\end{def*}

Equation \eqref{eq: g2} captures two distinct sources for the evolution of the driving functions: the integral term being the contribution from the SLE$(0, x_j, \bfs \rho_j)$ process, and the summation term being the contribution of the Loewner chain for $x_k$ acting on $x_j$, over all $k \neq j$. We now state our main result on the Loewner chains of $\Gamma(R)$ for elements of $\CRR_{n+1}(\bfs x)$.

\begin{thm}\label{thm: real_locus}
Let $\bfs x = \{x_1, \ldots, x_{2n} \}$ be distinct real points. Assume $\bfs \zeta = \{ \zeta_1, \ldots, \zeta_{n+1} \} \subset \widehat{\C}$ is closed under conjugation and solves the \textbf{stationary relation}
\begin{align}\label{eq: stationary}
\sum_{j=1}^{2n} \frac{1}{\zeta_k - x_j} = \sum_{l \neq k} \frac{2}{\zeta_k - \zeta_l}, \quad k=1,\ldots,n+1.
\end{align}
Then there exists an $R \in \CRR_{n+1}(\bfs x)$ with pole set $\bfs \zeta$. Furthermore, for
\begin{align}\label{eq: rhoj}
\bfs \rho_j = \sum_{k \neq j} 2 \delta_{x_k} - \sum_{k=1}^{n+1} 4 \delta_{\zeta_k}, \quad j=1,\ldots,2n,
\end{align}
the hulls $K_t$ generated by \textit{any} $\bfs \nu$-superposition of the SLE$(0, x_j, \bfs \rho_j)$ Loewner flows are a subset of $\Gamma(R)$, up to any time $t$ before the collisions of any poles or critical points. Up to any such time
\[
R \circ g_t^{-1} \in \CRR_{n+1}(\bfs x(t)),
\]
where $\bfs x(t)$ is the location of the critical points at time $t$ under the $\bfs \nu$-superposition.
\end{thm}

The last statement shows that the poles and critical points of the rational function generate Loewner chains along which the rationality of $R$ is preserved, and this fact is what ultimately shows that these particular chains generate the real locus. Moreover these results hold under \textit{any} $\bfs \nu$-superposition of the SLE$(0, x_j, \bfs \rho_j)$ processes. For any choice of $\bfs \nu$ the generated hull is always a subset of the same \textit{fixed} set $\Gamma(R)$. This is a strong manifestation of the \textbf{geometric commutation} property: that the same curves are generated regardless of the order in which the driving function dynamics are applied. Dub\'{e}dat studies geometric commutation for multiple SLE$(\kappa)$ in \cite{Dubedat06} (see also \cite{Graham:multiple_SLE}), although in the stochastic setting the commutation is only required to hold in law. In Section \ref{sec: commutation} we study geometric commutation in the deterministic setting and how it naturally leads to the system of null vector equations for the driving function dynamics. 

Under any $\bfs \nu$-superposition the processes $x_1(t), \ldots, x_{2n}(t)$ and $\zeta_1(t), \ldots, \zeta_{n+1}(t)$ form an autonomous dynamical system. The dynamics in the standard coordinate chart of $\H$ are expressed in \eqref{eq:x_zeta_dynamics} of Section \ref{sec:IoM}. Note that the particular choice $\nu_j \equiv 1$ and $\nu_k \equiv 0$ for $k \neq j$ corresponds to growing a single curve anchored at $x_j$. Consequently Theorem \ref{thm: real_locus} gives the following simple description of each curve in the multiple SLE$(0)$ ensemble.

\begin{cor}
Each individual curve in a multiple SLE$(0)$ ensemble is an SLE$(0, \bfs \rho)$ with $\int \! \bfs \rho = -6$. 
\end{cor}

The calculation that $\int \! \bfs \rho_j = -6$ is straightforward, but of course relies on the charges at the critical points being of size $2$ and the charges at the poles being of size $-4$. The reason for the specific choices of $2$ and $-4$ is more subtle, and we only explain it heuristically in our concluding remarks. An important consequence of $\bfs \rho_j$ having total mass $-6$ is the following.

\begin{cor}
Each individual curve in a multiple SLE$(0)$ ensemble is M\"{o}bius invariant.
\end{cor}

This is a property of SLE$(0, \bfs \rho)$ processes with $\int \! \bfs \rho = -6$. We state the precise result in Section \ref{sec:IoM} and prove it in Section \ref{sec: commutation}. It is the analogue of the well known fact that SLE$(\kappa, \bfs \rho)$ processes are conformally invariant (in law) iff $\int \! \bfs \rho = \kappa-6$. M\"{o}bius invariance means that if $\gamma$ is an SLE$(0, x, \bfs \rho)$ curve and $\phi \in \PSL(2, \R)$ then $t \mapsto \phi(\gamma(t))$ is an SLE$(0, \phi(x), \phi_{\#} \bfs \rho)$ curve, at least up to a time change. Here $\phi_{\#} \bfs \rho$ is the pushforward of $\bfs \rho$ by $\phi$. Note that because $\phi$ can swap the point at infinity with another point it is important that $\bfs \rho$ account for any charges at infinity.

The other mysterious aspect of Theorem \ref{thm: real_locus} is the stationary relation \eqref{eq: stationary}. Although SLE$(0, \bfs \rho)$ processes are defined for general $\bfs \rho$, Theorem \ref{thm: real_locus} requires that the points $\zeta_1, \ldots, \zeta_{n+1}$ in $\supp \bfs \rho$ satisfy \eqref{eq: stationary} in order to make the connection with real rational functions. The first assertion of the theorem, that a solution to the stationary relation implies the existence of an element of $R \in \CRR_{n+1}(\bfs x)$, is intimately connected to the poles of the rational function. By writing $R = P/Q$ we see that the finite poles of $R$ are just the zeros of $Q$, and since $Q$ is a polynomial with real coefficients these zeros must always appear in conjugate pairs. Furthermore, since $R$ is of degree $n+1$ with $2n$ critical points the number of zeros of $Q$ must always be $n$ or $n+1$. This follows from the fact that the numerator of $R'$ is $P'Q - Q'P$ and hence the $2n$ distinct points of $\bfs x$ can be critical points of $R$ iff $\deg P'Q - Q'P = 2n$. Since $\deg P, \deg Q \leq n+1$ this is possible iff either $\deg P = \deg Q = n+1$, or $\deg Q = n+1$ and $\deg P = n$, or $\deg P = n+1$ and $\deg Q = n$. The last case is precisely when $R$ has a pole at infinity, since then $R(z) = z + o(1)$ as $z \to \infty$. Note that including a point at infinity in $\bfs \zeta$ does not necessarily preclude $\bfs \zeta$ from being a solution to the stationary relation, since in that case both sides of \eqref{eq: stationary} degenerate to zero.

The connection between the stationary relation and elements of $\CRR_{n+1}(\bfs x)$ is made more precise in the next theorem. It states the stationary relation is not only a sufficient condition for the existence of real rational functions but is also (essentially) necessary. 

\begin{thm}\label{thm: stationary}
Let $\bfs x = \{ x_1, \ldots, x_{2n} \}$ be real and distinct and $\bfs \zeta = \{ \zeta_1, \ldots, \zeta_{n+1} \} \subset \widehat{\C}$ be closed under conjugation with the $\zeta_i$ distinct. Then there exists an $R \in \CRR_{n+1}(\bfs x)$ with pole set $\bfs \zeta$ iff $\bfs \zeta \cap \bfs x = \emptyset$ and $\bfs \zeta$ satisfies the stationary relation \eqref{eq: stationary}.
\end{thm}

Theorem \ref{thm: stationary} is a straightforward exercise in complex analysis based on the partial fraction expansion of $R'$. We highlight the statement as it is used heavily in our analysis of the Loewner chains and in making the connection with Calogero-Moser. Given $\bfs x$ it is typically difficult to explicitly determine solutions to \eqref{eq: stationary} for arbitrary $\bfs x$, other than for small $n$ (see Section \ref{sec: examples} for $n=1,$ $n=2$, and $n=3$). Since the theorem holds in both directions the existence results of \cite{EG02, MTV:Shapiro, EG11, PW20} for real rational functions already implies existence of solutions to the stationary relation. In fact there must be infinitely many solutions since post-composition of elements in $\CRR_{n+1}(\bfs x)$ does not preserve the poles, even though it does preserve the real locus and critical points. By \cite{Goldberg91, EG02,MTV:Shapiro,EG11,PW20} the solution space to \eqref{eq: stationary} must partition into exactly $C_n$ distinct equivalence classes under post-composition. In fact, if one could independently prove results on the structure of the solution space to \eqref{eq: stationary}, using algebraic methods or any other means, then Theorem \ref{thm: stationary} provides a new avenue for re-proving either Goldberg's upper bound or the matching lower bounds of \cite{EG02, MTV:Shapiro, EG11, PW20}. 

Using the stationary relation we also give a natural, self-contained proof of the following.

\begin{thm}\label{thm:RRF_are_geos}
Let $\bfs x = \{x_1, \ldots, x_{2n} \}$ be distinct real points. If $R \in \CRR_{n+1}(\bfs x)$ then $\Gamma(R)$ has the geodesic multichord property.
\end{thm}

Our proof of Theorem \ref{thm:RRF_are_geos} uses that the stationary relation is preserved under the SLE$(0, x_j, \bfs \rho_j)$ Loewner chains, which is a consequence of the proof of Theorem \ref{thm: real_locus}. We carefully analyze the limiting behavior of the stationary relation as a curve in the ensemble approaches its endpoint. We show that in this limit the map $R \circ g_t^{-1}$ degenerates to a rational function of one smaller degree and with two fewer critical points. Repeated iteration of this argument brings us to a single curve which we can directly show is a hyperbolic geodesic. From this we recover the geodesic multichord property by inverting the intermediate conformal maps and using that hyperbolic geodesics are conformally invariant. In essence Theorem \ref{thm:RRF_are_geos} is a converse to the result of Peltola and Wang, who show that geodesic multichords must be the real locus of a real rational function (this is for geodesic multichords of the type considered in this paper - more general types are considered in \cite{BE:canonical} and the recent \cite{MRW:piecewise}). Then by appealing to Goldberg's result \cite{Goldberg91} Peltola and Wang show that $\Gamma(R)$ satisfies the geodesic multichord property for every $R \in \CRR_{n+1}(\bfs x)$. Our argument is more direct and has no reliance on Goldberg's result. 

Our proof of Theorem \ref{thm: real_locus} ultimately rests on a field integral of motion for SLE$(0, \bfs \rho)$ Loewner chains, which we describe in the next theorem. We note that the structure of the $\bfs \rho$ measure is the same as in Theorem \ref{thm: real_locus}, but the result does \textit{not} require the stationary relation to hold.

\begin{thm} \label{primitive}
Let $x_1, \ldots, x_{2n}$ be real and distinct and $\bfs \zeta = \{\zeta_1, \ldots, \zeta_{n+1} \} \subset \widehat{\C}$ be distinct and closed under conjugation. Define $\bfs \rho_j$, $j=1,\ldots,2n$, as in \eqref{eq: rhoj}. Then for any $\bfs \nu$-superposition of the SLE$(0, x_j, \bfs \rho_j)$ processes the quantity
\begin{equation}\label{eq: Nt}
N_t(z) := g_t'(z) \frac{\prod_{j=1}^{2n} (g_t(z) - x_j(t))}{\prod_{\zeta_k \neq \infty} (g_t(z) - g_t(\zeta_{k}))^2}
\end{equation}
is an integral of motion on the interval $[0, \tau_z \wedge \tau)$, for each $z \in \overline{\H}$.
\end{thm}

As we explain in the concluding remarks, we discovered this field integral of motion by taking a heuristic $\kappa \to 0$ limit of a \textbf{martingale observable} for a system of multiple SLE$(\kappa)$ curves. Martingale observables can be thought of as the stochastic counterparts of integrals of motion, and in \cite{AKM} we will explain how conformal field theory produces an infinite family of martingale observables for systems of multiple SLE$(\kappa)$. These observables are correlation functions derived from formal algebraic operations on the random Gaussian free field, although multiple SLE requires a version of the GFF that is shifted by an appropriate deterministic function; see \cite{SS09, Dubedat09, SS13} (among others) for more on this idea. The quantity \eqref{eq: Nt} is closely related to this shift.

\subsection{Solutions to the Null Vector Equations}

In this paper the term null vector equations refers, for a given set $\bfs x = \{x_1, \ldots, x_{2n} \}\in \Rd^{2n}=\{x_j \textrm{ are distinct real points}\}$, to the system of $2n$ quadratic equations
\begin{align}\label{eq: NV0}
\frac{1}{2} U_j^2 + \sum_{k \neq j} \frac{2}{x_k - x_j} U_k - \sum_{k \neq j} \frac{6}{(x_k - x_j)^2} = 0, \quad j=1,\ldots,2n.
\end{align}
We are interested in real-valued solutions to these equations, in which the variables $U_1, \ldots, U_{2n}$ are regarded as unknowns. Their relevance to multiple SLE$(0)$ initially came from the expectation that, given solutions to \eqref{eq: NV0}, trajectories of the differential equation
\begin{align}\label{eq:Uj_intro}
\dot x_j(t) = U_j(\boldsymbol x(t))\nu_j(t) + \sum_{k\ne j} \frac{2\nu_k(t)}{x_j(t)-x_k(t)}, \quad j=1,\ldots,2n
\end{align}
should be the driving functions for the individual curves in a commuting multiple SLE$(0)$ ensemble. This heuristic was first explained in the SLE literature in \cite{BBK05}, who arrived at it by taking a formal $\kappa \to 0$ limit of the \textbf{BPZ equations}, the system of $2n$ linear partial differential equations
\begin{align}\label{eq: BPZ}
\left( \frac{\kappa}{2} \partial_{j}^2 + \sum_{k \neq j} \left( \frac{2}{x_k - x_j} \partial_{k} - \frac{(6-\kappa)/\kappa}{(x_k - x_j)^2} \right) \right) \mathcal{Z} = 0, \quad j=1,\ldots,2n. 
\end{align} 
The BPZ equations are satisfied by partition functions $\mathcal{Z} : \R^{2n}_d \to \R_+$ that describe multiple SLE$(\kappa)$ curves; more precisely $\kappa \dell_j \log \mathcal{Z}$ is the drift term in the stochastic differential equation for the driving process $x_j(t)$ that generates the $j$th curve in the ensemble. Bauer, Bernard, and Kyt\"{o}lla derived the BPZ equations using ideas from conformal field theory (which we also do in the forthcoming \cite{AKM}), and it is due to the CFT connection that we refer to \eqref{eq: NV0} as the null vector equations. Later Dub\'{e}dat \cite{Dubedat07} (see also \cite{Graham:multiple_SLE}) showed that \eqref{eq: BPZ} is a necessary condition for generating random curves that satisfy a geometric commutation property (in law). In Section \ref{sec: commutation} we use similar ideas to show that the null vector equations \eqref{eq: NV0} are a necessary condition on any $U_j : \Rd^{2n} \to \R$, if one wants that the curves generated by solutions to \eqref{eq:Uj_intro} satisfy geometric commutation and a conformal invariance property. In the deterministic setting the argument is quite different from \cite{Dubedat07}, owing to the quadratic nature of the null vector equations compared to the linear structure of the BPZ equations. We also note that our argument in Section \ref{sec: commutation} frees the null vector equations \eqref{eq: NV0} from their connection to multiple SLE$(\kappa)$, which is the only way in which they have appeared in the SLE literature to date. Indeed the derivation of \eqref{eq: NV0} from \eqref{eq: BPZ} in \cite{BBK05} is by taking a limit of $\kappa \partial_j \log \mathcal{Z}$ as $\kappa \to 0$, although their argument was only a heuristic and not a precise proof. Bauer, Bernard, and Kyt\"{o}lla also make no claim about existence of solutions to \eqref{eq: NV0}, and until \cite{PW20} no solutions were known in the SLE literature. It is possible that solutions to \eqref{eq: NV0} had been discussed in the algebraic geometry literature but we are unaware of them.

Our next result gives a new expression for solutions to the null vector equations \eqref{eq: NV0}. Our solutions are precisely the right hand side of \eqref{eq: zero_rho_driving}, for the $\bfs \rho_j$ encoded by the SLE$(0, x_j, \bfs \rho_j)$ processes of Theorem \ref{thm: real_locus}. The exact formula is given in equation \eqref{eq: U}, where we note the appearance of the terms $\zeta_{\alpha, k}(\bfs x)$ that we call \textbf{pole functions}. The purpose of these objects is so that we can write our solutions purely in terms of the $\bfs x$ variables, which are the only parameters that appear in \eqref{eq: NV0}, and without the auxiliary $\bfs \zeta$ variables. For convenience we define these pole functions in terms of a distinguished representative of each equivalence class $\CRR_{n+1}(\bfs x; \alpha)$, where the classes are induced by post-composition and indexed by link patterns. In Section \ref{sec: RRF_Loewner} we show that any other choice of representative would lead to the same solutions $U_{\alpha,j}$, i.e. the solutions $U_{\alpha,j}$ are constant on each equivalence class $\CRR_{n+1}(\bfs x; \alpha)$.

\begin{def*}
Let $\bfs x = \{ x_1, \ldots, x_{2n} \}$ be real and distinct and $\alpha$ be a non-crossing topological link pattern connecting the points in $\bfs x$. Define the canonical element of $\CRR_{n+1}(\bfs x; \alpha)$ to be the rational function $R$ that satisfies the \textit{hydrodynamic normalization} $R(z) = z + o(1)$ at infinity and whose derivative factors as
\[
R'(z) = \frac{\prod_{j=1}^{2n} (z - x_j)}{\prod_{k=1}^n (z - \zeta_k)^2}.
\]
Here $\zeta_1, \ldots, \zeta_n$ are the finite poles of $R$ and by the hydrodynamic normalization there is another simple pole at infinity. We define $\zeta_{\alpha, 1}(\bfs x), \ldots, \zeta_{\alpha,n}(\bfs x)$ to be the poles of $R$, and we call $\zeta_{\alpha,k}$ the \textbf{pole functions}. We use no specific labeling scheme for the indices $k=1,\ldots,n$, as the choice of the labeling is not important so long as it is applied consistently.
\end{def*}

By Theorem \ref{thm: stationary} the pole functions must satisfy the stationary relation \eqref{eq: stationary} but it is difficult to find explicit expressions for them in terms of the $\bfs x$. In this sense the solutions \eqref{eq: U} are semi-explicit. Nonetheless the pole functions exist as a consequence of the non-emptiness of the equivalence classes $\CRR_{n+1}(\bfs x; \alpha)$ as proved by (any one of) \cite{EG02, MTV:Shapiro, EG11, PW20}. These existence results provide the existence of the solutions \eqref{eq: U} in the next theorem.

\begin{thm}\label{thm: Z}
For distinct boundary points $\bfs x$ and a link pattern $\alpha$ connecting them define
\begin{align} \label{eq: Z}
Z_{\alpha}(\bfs x) := \!\! \prod_{1 \leq j < k \leq 2n} \!\! (x_j - x_k)^2 \!\! \prod_{1 \leq l < m \leq n} \!\! (\zeta_{\alpha,l}(\bfs x) - \zeta_{\alpha,m}(\bfs x))^8 \prod_{k=1}^{2n} \prod_{l=1}^n (x_k - \zeta_{\alpha,l}(\bfs x))^{-4}.
\end{align}
Then $Z_{\alpha}$ is strictly positive, and is a $(3,0)$-differential at each coordinate of $\bfs x$ with respect to M\"{o}bius transformations of $\H$ to itself. Moreover 
\begin{align} \label{eq: U}
U_{\alpha,j}(\bfs x) = \dell_{x_j} \log Z_{\alpha}(\bfs x) = \sum_{k \neq j} \frac{2}{x_j - x_k} - \sum_{k=1}^n \frac{4}{x_j - \zeta_{\alpha,k}(\bfs x) }, \quad j=1,\ldots,2n,
\end{align}
are real-valued, translation invariant, homogeneous of degree $-1$, and satisfy the system of null vector equations \eqref{eq: NV0}.
\end{thm}

The simplicity of formula \eqref{eq: U} is one of the strengths of the result, as it also allows us to verify that \eqref{eq: U} is a solution to the null vector equations \eqref{eq: NV0} using nothing more than the stationary relation and basic algebra. That \eqref{eq: U} is the logarithmic derivative of a function $Z_{\alpha}$ (the ``partition function'') is analogous to the situation for multiple SLE$(\kappa)$, but the partition function itself is somewhat more transparent in this case. Peltola and Wang also express their solutions $U_{\alpha,j}(\bfs x)$ as the derivative of a function $\mathcal{U}_{\alpha}(\bfs x)$, defined as the value at the minimum for a certain functional on curves that connect the points in $\bfs x$ according to the link pattern $\alpha$ (the multichordal Loewner \textit{potential}). The functional itself is explicit in terms of the Brownian loop measure, and using convexity properties they are able to show that a minimizer exists and is unique. However the Brownian loop measure term makes it difficult to evaluate the functional explicitly for any fixed ensemble of curves, including at the minimizer, and so \cite{PW20} has no simple closed form expression for $\mathcal{U}_{\alpha}(\bfs x)$. Nonetheless they are still able to show that the values $-\partial_{x_j} \mathcal{U}_{\alpha}(\bfs x)$ solve the null vector equations \eqref{eq: NV0}, and that when used as Loewner vector fields these curves also generate the individual curves in $\Gamma(R)$. This indirectly shows that our version of $U_{\alpha,j}$ agrees with theirs (since they generate the same ensemble of curves), which in turn implies that $\log Z_{\alpha}$ and $-\mathcal{U}_{\alpha}$ should be the same (up to an additive constant). It would be interesting to see a direct reason for this argument, as it clearly hints at an interesting relation between the Brownian loop measure and rational functions.

\subsection{Connection to the Calogero-Moser System}

Finally we come to the connection with the Calogero-Moser integrable system. Calogero-Moser describes a one-dimensional many-body problem in which the particles interact via a pairwise inverse quadratic potential. This means that each particle applies a force to every other equal to the cube of the inverse distance between them. The system is of Hamiltonian type and is most famous for being \textbf{completely integrable}. Informally this means that one can find as many integrals of motion as there are particles, with all of the integrals being \textit{in involution}. These integrals of motion are determined by finding a \textbf{Lax pair} \cite{Lax:pairs} for the dynamical system, and their existence implies that the integrals of motion can be used (at least abstractly) to re-parameterize the dynamical system via the so-called \textbf{action-angle coordinates} in which the new system is solvable. Moser \cite{Moser:CM_integrable} discovered the Lax pair $(L,M)$ shortly after Calogero \cite{Calogero:quantum_CM_integrable} solved the quantum version of the system. Moser's work showed that the transformation into action-angle coordinates is explicit (see also \cite{Ruij:action_angle}), leading to several concrete methods for generating solutions to the equation. In our case the Calogero-Moser system arises by a particular superposition of the SLE$(0, x_j, \bfs \rho_j)$ processes of Theorem \ref{thm: real_locus}. We call it the $1/4$-superposition.

\begin{thm}\label{thm: CM-main}
Let $\bfs x = \{x_1, \ldots, x_{2n} \}$ be distinct real points and $\bfs \zeta = \{ \zeta_1, \ldots, \zeta_{n+1} \} \subset \widehat{\C}$ be closed under conjugation and solve the stationary relation \eqref{eq: stationary}. Under the $1/4$-superposition of the SLE$(0, x_j, \bfs \rho_j)$ processes of \eqref{eq: rhoj}, i.e. $\nu_j \equiv 1/4$ for all $j$, the critical points $\bfs x(t)$ evolve as
\begin{align}\label{eq: CM-main}
\ddot x_j = - \sum_{k \neq j} \frac{2}{(x_j - x_k)^3}.
\end{align}
\end{thm}

The choice of $\nu_j \equiv 1/4$ is mostly for convenience; any other constant value would be a linear time change of the system. Under these second order dynamics the particles $\bfs x(t)$ stay on the real line but due to the sign they begin to collide with each other in finite time. Hence the forces between particles are \textit{attractive} in our situation, whereas Calogero-Moser is most typically studied in the \textit{repulsive} case (see \cite{Wilson:collisions} for a notable exception). The Hamiltonian corresponding to \eqref{eq: CM-main} is
\begin{align}\label{eq: Ham}
\mathcal{H}(\bfs x, \bfs p) = \frac{1}{2} \sum_j p_j^2 - \sum_{j < k} \frac{1}{(x_j - x_k)^2},
\end{align}
and the negative sign therein is somewhat atypical as a result. It is, however, equivalent to the typical Hamiltonian with a positive sign but the particles $x_i$ confined to the imaginary axis instead of the real line, and so we may freely use the known results on Calogero-Moser in the study of our system. The transition from the first-order SLE$(0, x_j, \bfs \rho_j)$ processes to the second order Calogero-Moser system is ultimately a consequence of the stationary relation. Under the $1/4$-superposition the poles and critical points follow a coupled system of first order differential equations, and then owing to the stationary relation the second order evolution of the poles and critical points decouple into two autonomous systems. Remarkably the poles also end up following the same Calogero-Moser system (see Corollary \ref{cor: CM}). The same coupled first order system appears in \cite{ABW09, AGK11} although without the stationary relation involved. Many non-linear evolution equations  (Korteweg-de Vries, Kadomtsev-Petviashvili, Benjamin-Ono) turn out to be integrable by studying the evolution of their poles under finite-dimensional dynamical systems. See \cite{AC:solitons_book} for a broad survey, and \cite{AMM:KdV} in particular for a situation similar to ours in which the authors study a collection of flows induced by the Calogero-Moser system that preserve the rationality of certain initial conditions to the KdV equation. We derive a result of this type by applying Theorem \ref{thm: CM-main} to show that the time evolution of the real rational functions $R \circ g_t^{-1}$ of Theorem \ref{thm: real_locus} follows a parabolic partial differential equation. 

\begin{thm}\label{parabolic PDE}
Let $\bfs x = \{x_1, \ldots, x_{2n} \}$ be distinct real points and $\bfs \zeta = \{ \zeta_1, \ldots, \zeta_{n+1} \} \subset \widehat{\C}$ be closed under conjugation and solve the stationary relation \eqref{eq: stationary}. Under the 1/4-superposition of the SLE$(0, x_j, \bfs \rho_j)$ processes of \eqref{eq: rhoj} the function
\[
Q_t(z) = \prod_{\zeta_k \neq \infty} (z - g_t(\zeta_k))
\]
solves the backward heat equation 
\begin{align}\label{eq: backward heat}
\dot Q_t = -\frac12 Q_t''.
\end{align}
Moreover, if $R$ is an element of $\CRR_{n+1}(\bfs x)$ with pole set $\bfs \zeta$ and $g_t$ is the Loewner chain induced by the $1/4$-superposition, the rational functions $R_t := R \circ g_t^{-1}$ follow the parabolic PDE
\[
\dot R_t(z) = -\frac12 R_t''(z) - \frac{Q_t'(z)}{Q_t(z)} R_t'(z),
\]
at least until the first collision time between any particles in $\bfs x \cup \bfs \zeta$.
\end{thm}

In Section \ref{sec: examples} we directly verify this evolution in specific examples, by explicitly computing the time evolution of the poles and critical points. As stated, in both Theorems \ref{thm: CM-main} and \ref{parabolic PDE} the connection with real rational functions comes about through the stationary relation and Theorem \ref{thm: stationary}. Note that no explicit dependence on the link pattern for the real locus is mentioned; instead it lurks in the choice of the initial momenta $\dot{\bfs x}(0)$. The next result shows that for every equivalence class in $\CRR_{n+1}(\bfs x)$ there is a choice of initial momenta $\dot{\bfs x}(0)$ such that the corresponding solutions $t \mapsto \bfs x(t)$ of the Calogero-Moser system grow the real locus corresponding to that equivalence class, when $\bfs x(t)$ is used as the driving function in a Loewner chain. The initial momenta $\dot{\bfs x}(0)$ is intimately related to the null vector equations \eqref{eq: U}. 

\begin{thm} \label{thm: CM2SLEzero}
Let $\bfs x(t) = (x_1(t), \ldots, x_{2n}(t))$ evolve according to the Calogero-Moser dynamics \eqref{eq: CM-main} with $\bfs x(0)$ having all coordinates distinct. For each choice of link pattern $\alpha$ there is a choice of canonical momenta $\dot{\bfs x}(0)$ that solves the system of quadratic equations
\[
\dot x_j^2 - \sum_{k\ne j}\frac{\dot x_j + \dot x_k}{x_j - x_k}- \sum_{k\ne j}\frac{1}{(x_j-x_k)^2} + \frac{1}{2} \sum_{k \neq j} \sum_{l \neq k} \frac{1}{(x_j-x_k)(x_j-x_l)} = 0, \quad j=1,\ldots,2n,
\]
such that the Loewner chain
\[
\dot g_t(z) = \sum_{j=1}^{2n} \frac{1/2}{g_t(z) - x_j(t)}, \quad g_0(z) = z
\]
generates the multiple SLE$(0; \bfs x; \alpha)$ curves, at least up until the first collision time of any two $x_j$.
\end{thm}

In Section \ref{sec: CM} we show how this result can naturally be expressed in terms of the Lax pair for the Calogero-Moser system. 

\section{\texorpdfstring{SLE$(0, \bfs \rho)$ and the Field Integral of Motion}{SLE(0, rho) and the Field Integral of Motion}} \label{sec:IoM}

We begin by deriving some basic but important properties of SLE$(0, \bfs \rho)$ processes. Note that the dynamics are defined on all of $\Cx$, but that symmetry of $\bfs \rho$ ensures that the driving function stays on the real line. Consequently the Loewner chain maps real points to real points and the hull $K_t$ is symmetric under conjugation. Although the SLE$(0, \bfs \rho)$ flow is purely deterministic the name comes from the stochastic SLE$(\kappa, \bfs \rho)$ processes, which first appeared in \cite{LSW:restriction} for describing the boundaries of random conformally invariant sets in scaling limits of discrete statistical mechanics models. These processes are also of central importance in the study of the so-called \textit{imaginary geometry} \cite{MS16:imaginary1}, and in many other contexts which are too numerous to list here. We use the definition of SLE$(\kappa, \bfs \rho)$ processes coming from \cite{SW05}, which is equivalent to all other definitions in the literature but is best suited to our needs. For $\kappa > 0$ and $x, \bfs \rho$ as in Section \ref{subsec:RRF_SLE} an SLE$(\kappa, x, \bfs \rho)$ corresponds to the same Loewner chain \eqref{eq: g} but with a driving function that evolves as
\[
d x_t = \sqrt{\kappa} dB_t + \int_{\C} \frac{d \bfs \rho(z)}{x_t - g_t(z)}, \quad x_0 = x, \,
\]
with $B_t$ a standard Brownian motion. Setting $\kappa = 0$ matches the definition of SLE$(0, x, \bfs \rho)$ above. Our next result shows that SLE$(0, \bfs \rho)$ processes match their SLE$(\kappa, \bfs \rho)$ counterparts in two important ways: the hulls $K_t$ are simple curves in $\H$ (which is also the case for $\kappa$ positive but small), and these curves are invariant under conformal automorphisms of $\H$ iff $\bfs \rho(\Cx) = \int \! \bfs \rho = -6$ (for $\kappa > 0$ the required condition is $\int \! \bfs \rho = \kappa-6$). Applying this result to Theorem \ref{thm: real_locus} show that the arcs of real loci of real rational functions can be generated by particular choices of $\bfs \rho$ with $\int \! \bfs \rho = -6$, thereby allowing us to understand the conformal invariance property of the real loci in a new way.

\begin{prop}\label{prop: SLE_zero_rho_curve}
Let $x \in \R$ and $\bfs \rho$ be a finite atomic measure on $\Cx$ that is symmetric under conjugation. If $\bfs \rho(x) = 0$ the hull $K_t$ is a simple curve in $\H$ for any $t < \tau$. This curve is invariant under conformal automorphisms of $\H$ to itself iff $\bfs \rho(\Cx) = \int \! \bfs \rho = -6$.
\end{prop}

\begin{proof}
That the hull $K_t$ is a simple curve (i.e. does not self-touch) in $\H$ follows from the smoothness of the vector field that the driving function follows. Indeed, since $\bfs \rho(x) = 0$ the force points are initially separated from the driving function and the vector field on the right hand side of the driving function dynamics \eqref{eq: zero_rho_driving} is locally Lipshitz for all $t < \tau$. This is a sufficient condition for the hull to be a simple curve up until time $\tau$.
 
The proof of conformal invariance for $\int \! \bfs \rho = -6$ is deferred until Section \ref{sec:Loewner_flows_charts}, where it fits in naturally to the discussion on Loewner flows in a different global coordinate chart.
\end{proof}

Invariance of the curve under conformal automorphisms means that
\[
\phi \left( \operatorname{SLE}(0, x, \bfs \rho) \right) = \operatorname{SLE}(0, \phi(x), \phi_{\#} \bfs \rho),
\]
where on both sides ``SLE'' refers to the hull generated by the Loewner chain, considered as an unparameterized subset of $\H$. As a parameterized curve equality only holds after an appropriate time change of one of the curves. The conformal automorphisms $\phi$ are the M\"{o}bius transformations of $\H$ to itself, i.e. the subgroup $\PSL(2, \R)$ of fractional linear transformations. One way to think of it is the following: if $\gamma(t)$ is the curve generated by the SLE$(0, x, \bfs \rho)$ dynamics then it is fully determined by $x$ and $\bfs \rho$. Conformal invariance is the statement that the curve $\phi(\gamma(t))$ is the same as the curve determined by $\phi(x)$ and $\phi_{\#} \bfs \rho$, at least up to a time change.

Our next result derives a field integral of motion a single for SLE$(0, \bfs \rho)$ process. This integral of motion is a special case of Theorem \ref{primitive}, in which the given quantity is an integral of motion for the superposition of many such processes. One consequence of considering only a single SLE$(0, \bfs \rho)$ process is that the next result does not assume that $\bfs \rho$ has the particular charges of $2$ and $-4$ that are present in Theorem \ref{primitive}.

\begin{prop}\label{prop: SLE_zero_rho_IoM}
Let $x \in \R$ and $\bfs \rho$ be a finite atomic measure on $\Cx$ that is symmetric under conjugation. Under the SLE$(0, x, \bfs \rho)$ dynamics process the quantity
\begin{align}\label{eq:SLE0_rho_IoM}
g_t'(z) (g_t(z) - x_t) \exp \left \{ \frac{1}{2} \int_{\C} \log (g_t(z) - g_t(w)) \, d \bfs \rho|_{\C}(w) \right \}
\end{align}
is, for each $z \in \overline{\H}$, an integral of motion on $[0, \tau_z \wedge \tau)$. 
\end{prop}

Here $\bfs \rho|_{\C}$ refers to the measure restricted to $\C$, i.e. ignoring any charge at infinity. In practice we will only consider $\bfs \rho$ such that $\bfs \rho(z)$ is an even integer at any $z$. Therefore the exponential term will be well-defined with respect to branch cuts of the logarithm. The algebra in the proof does not require this property of $\bfs \rho$, as we now see.

\begin{proof}[Proof of Proposition \ref{prop: SLE_zero_rho_IoM}]
Without loss of generality assume that $\bfs \rho = \bfs \rho|_{\C}$, i.e. that $\bfs \rho$ has no charge at infinity. The time derivative of the logarithm of \eqref{eq:SLE0_rho_IoM} is
\[
\frac{\dot{g}_t'(z)}{g_t'(z)} + \frac{\dot{g}_t(z) - \dot{x}_t}{g_t(z) - x_t} + \frac{1}{2} \int_{\C} \frac{\dot{g}_t(z) - \dot{g}_t(w)}{g_t(z) - g_t(w)} d \bfs \rho(w).
\]
By using \eqref{eq: g} and its derivative with respect to $z$ we obtain
\[
\frac{\dot{g}_t'(z)}{g_t'(z)} + \frac{\dot{g}_t(z)}{g_t(z) - x_t} = - \frac{2}{(g_t(z) - x_t)^2} + \frac{2}{(g_t(z) - x_t)^2} = 0.
\]
Similarly from \eqref{eq: g} and the evolution equation \eqref{eq: zero_rho_driving} for $x_t$ we have
\begin{align*}
- \frac{\dot{x}_t}{g_t(z) - x_t} + \frac{1}{2} \int_{\C} \frac{\dot{g}_t(z)}{g_t(z) - g_t(w)} \, d \bfs \rho(w) &= -\frac{1}{g_t(z) - x_t} \int_{\C} \left( \frac{1}{x_t - g_t(w)} - \frac{1}{g_t(z) - g_t(w)} \right) \, d \bfs \rho(w) \\
&= -\int_{\C} \frac{d \bfs \rho(w)}{(x_t - g_t(w))(g_t(z) - g_t(w))}. 
\end{align*}
Using \eqref{eq: g} again we obtain
\[
-\frac{1}{2} \int_{\C} \frac{\dot{g}_t(w)}{g_t(z) - g_t(w)} \, d \bfs \rho(w) = \int_{\C} \frac{d \bfs \rho(w)}{(x_t - g_t(w))(g_t(z) - g_t(w))},
\]
and the result follows.
\end{proof}

Now we turn to Theorem \ref{primitive} and the integral of motion for superimposed SLE$(0, x_j, \bfs \rho_j)$ processes. Recall that the purpose of the superpositions is to allow for the curve to grow from several locations at once. If more than one $\nu_i$ is non-zero at some time then the hull $K_t$ grows from at least two distinct points, since there are multiple singularities on the right hand side of \eqref{eq: g2}. By changing the value of $\nu_j$ with time we can alter the speed of growth at each of the growth locations. For example, setting $\nu_j \equiv 1$ for all $j$ causes $K_t$ to grow from $N$ locations simultaneously. This choice also leads to the connection with the Calogero-Moser integrable system, as we will see in Section \ref{sec: CM}. On the other hand, setting $\nu_j \equiv 1$ and $\nu_k \equiv 0$ for $k \neq j$ corresponds to the single SLE$(0, x_j, \bfs \rho_j)$ system in which $K_t$ only grows from a single point. It will always be clear from the context which particular growth type we are referring to. Also note that the dynamics are initially only defined until the first collision time between the points within or between any two SLE$(0, x_j, \bfs \rho_j)$ processes, but sometimes they can be extended beyond these times. However this requires a choice of how precisely the dynamics should be extended and the structure of the hull after this time may depend on this choice. We discuss a natural extension in Section \ref{sec: RRF_Loewner}.

\begin{proof}[Proof of Theorem \ref{primitive}]
For each $j=1,\ldots,2n$, the definition of $\bfs \rho_j$ implies that $N_t(z)$ has the form
\[
N_t(z) = g_t'(z) (g_t(z) - x_j(t)) \exp \left \{ \frac{1}{2} \int_{\C} \log(g_t(z) - g_t(w)) \, d \bfs \rho_j|_{\C}(w) \right \}.
\]
Note that the expression is the same for each value of $j=1,\ldots,2n$. This is a special property of the $\bfs \rho_j$ that we work with, in particular that $\bfs \rho_j(x_i) = 2$ for all $j \neq i$. For these $\bfs \rho_j$ we also have that the driving functions $x_j(t)$, $j=1,\ldots,2n$ and the poles $\zeta_k(t)$, $k=1,\ldots,n+1$, evolve as the autonomous system
\begin{align}\label{eq:x_zeta_dynamics}
\dot x_j = \nu_j \left( \sum_{k \neq j} \frac{2}{x_j - x_k} - \sum_{k=1}^{n+1} \frac{4}{x_j - \zeta_k} \right) + \sum_{k \neq j} \frac{2 \nu_k}{x_j - x_k}, \quad \dot \zeta_k = \sum_{j=1}^{2n} \frac{2 \nu_j}{\zeta_k - x_j},
\end{align}
where we have dropped the $t$ dependence for convenience. As in the proof of Proposition \ref{prop: SLE_zero_rho_IoM} we have that the time derivative of the logarithm of $N_t(z)$ is
\[
\frac{d}{dt} \log N_t(z) = \frac{\dot g_t'(z)}{g_t'(z)} + \sum_{j=1}^{2n}\frac{\dot g_t(z)-\dot x_j(t)}{g_t(z)-x_j(t)} -2 \sum_{k=1}^{n+1} \frac{\dot g_t(z)-\dot \zeta_k(t)}{g_t(z)-\zeta_k(t)}.
\]
Also, as in the proof of Proposition \ref{prop: SLE_zero_rho_IoM} we have (dropping the $z$ dependence in $g_t(z)$)
\[
\frac{\dot{g}_t'}{g_t'} = -\sum_{j=1}^{2n} \frac{2 \nu_j}{(g_t - x_j)^2}.
\]
Insert this into the expression for $\partial_t \log N_t(z)$ and expand out $\dot x_j$ and $\dot \zeta_k$ and group the terms as $\dot N_t(z)/ N_t(z) = \mathrm{I} + \mathrm{II}$, where
\[
\mathrm{I} =\sum_{j\ne k} \frac{2\nu_k}{(g_t-x_j)(g_t-x_k)}-\sum_{j\ne k} \frac{2\nu_j}{(g_t-x_j)(x_j-x_k)}-\sum_{j\ne k} \frac{2\nu_k}{(g_t-x_j)(x_j-x_k)}
\]
is the sum of the terms not involving $\zeta_k$ and
\[
\mathrm{II} =\sum_{j,k} 4\nu_j \Big(\frac{1}{(g_t-x_j)(x_j-\zeta_k)}- \frac{1}{(g_t-x_j)(g_t-\zeta_k)}- \frac{1}{(g_t-\zeta_k)(x_j-\zeta_k)}\Big)
\]
is the sum of the terms that do involve $\zeta_k$. 
It is easy to see that each summand of $\mathrm{II}$ vanishes. It follows that  $\mathrm{II} = 0$. 
Rearranging terms or changing the summation index,
\[
\mathrm{I} =\sum_{j\ne k} 2\nu_k \Big(\frac{1}{(g_t-x_j)(g_t-x_k)}+\sum_{j\ne k} \frac{1}{(g_t-x_k)(x_j-x_k)}-\sum_{j\ne k} \frac{1}{(g_t-x_j)(x_j-x_k)} \Big) = 0.
\]
This completes the proof.
\end{proof}

\section{Real Rational Functions and Loewner Flow \label{sec: RRF_Loewner}}

In the first subsection we prove Theorem \ref{thm: stationary}, that the stationary relation between poles and critical points is both a necessary and sufficient condition for elements of $\CRR_{n+1}(\bfs x)$, and Theorem \ref{thm: real_locus}, that any superposition of the SLE$(0, x_j, \bfs \rho_j)$ processes generates the real locus of $R$. In the second subsection we prove Theorem \ref{thm: Z} on our solutions to the stationary relation, and in the third subsection we prove Theorem \ref{thm:RRF_are_geos} that $\Gamma(R)$ has the geodesic multichord property. The final subsection reviews the classical description of the real locus as flow lines and geodesics.

\subsection{Stationary Relation and Loewner Flow\label{subsec:Loewner}}
The next lemma collects some basic facts about meromorphic functions and their partial fraction expansions.

\begin{lem}\label{lem:rational_basic}
Fix $n \geq 1$ and let $\bfs x = \{ x_1, \ldots, x_{2n} \}$ be real and distinct. The following are true.
\begin{enumerate}[(i)]
\item For $0 \leq d \leq n+1$ consider the meromorphic function
\begin{align}\label{eq:meromorphic_product}
f(z) = \frac{\prod_{j=1}^{2n} (z - x_j)}{\prod_{k=1}^d (z - \zeta_k)^2}.
\end{align}
If $\zeta_1, \ldots, \zeta_d \in \C$ are all distinct then $f(z)$ has a partial fraction expansion
\begin{align}\label{eq:meromorphic_partial_fraction}
f(z) = p(z) + \sum_{k=1}^d \frac{A_k}{(z - \zeta_k)^2} + \sum_{k=1}^d \frac{B_k}{z - \zeta_k},
\end{align}
where $p(z)$ is a polynomial of degree $\max \{ 2n - 2d, 0 \}$, and the constants $A_k$ and $B_k$ are
\begin{align}\label{eq:Ak_Bk}
A_k = \frac{\prod_{j=1}^{2n} (\zeta_k - x_j)}{\prod_{l \neq k} (\zeta_k - \zeta_l)^2}, \quad B_k = \left( \sum_{j=1}^{2n} \frac{1}{\zeta_k - x_j} - \sum_{l \neq k} \frac{2}{\zeta_k - \zeta_l} \right) A_k, \quad k=1,\ldots, d.
\end{align}
\item If $\zeta_1, \ldots, \zeta_d \in \C$ are all distinct then the meromorphic function \eqref{eq:meromorphic_partial_fraction} has a rational primitive iff all $B_k = 0$.
\item Up to a real multiplicative constant, the derivative of any $R \in \CRR_{n+1}(\bfs x)$ factorizes as
\begin{align}\label{eq:Rprime_product}
R'(z) = \frac{\prod_{j=1}^{2n} (z - x_j)}{\prod_{\zeta_k \neq \infty} (z - \zeta_k)^2}, 
\end{align}
where $\{ \zeta_1, \ldots, \zeta_{n+1} \}$ are the poles of $R$, considered as a multi-set in the case of multiplicities.
\item If $R \in \CRR_{n+1}(\bfs x)$ and no element of $\bfs x$ is a pole of $R$ then all poles are simple.
\item If $R \in \CRR_{n+1}(\bfs x)$ and no element of $\bfs x$ is a pole of $R$ then, up to a real multiplicative constant, $R'$ has the partial fraction expansion 
\begin{align}\label{eq:Rprime_partial_fraction}
R'(z) = \1{d=n} + \sum_{k=1}^d \frac{A_k}{(z - \zeta_k)^2},
\end{align}
where $\{ \zeta_1, \ldots, \zeta_d \}$ are the finite poles of $R$, $d \in \{n, n+1 \}$, and $A_k$ is given by \eqref{eq:Ak_Bk}. 
\item If $R \in \CRR_{n+1}(\bfs x)$ has all simple poles then they are distinct from the points in $\bfs x$.
\end{enumerate}
\end{lem}

Parts (iv)-(vi) are stated so as to deal with a rare but possible phenomenon: the existence of $R \in \CRR_{n+1}(\bfs x)$ that have a double pole at some critical point $x_i$. Example \ref{eg: rainbow} in Section \ref{sec: examples} is one such example for $n=2$. The existence of such rational functions complicates the statements of many of our results but not the essence of the underlying arguments. We will also see that these exceptional $R$ are in a sense ``isolated singularities'' that can be removed in a very natural way, so rather than writing lengthy additional statements to deal with them we simply work with $R \in \CRR_{n+1}(\bfs x)$ that do not exhibit this phenomenon. In this case statement (ii) can be replaced by the more general statement that a meromorphic function has a rational primitive iff it has zero residue at each of its poles.

\begin{proof}
In part (i) it is clear that \eqref{eq:meromorphic_partial_fraction} follows from \eqref{eq:meromorphic_product} and each pole being of order at most two, the latter of which follows from the assumption that the $\zeta_k$ are distinct. The formula for $A_k$ follows from \eqref{eq:meromorphic_product} and \eqref{eq:meromorphic_partial_fraction} by
\[
A_k = \lim_{z \to \zeta_k} f(z) (z - \zeta_k)^2 = \frac{\prod_{j=1}^{2n} (\zeta_k - x_j)}{\prod_{l \neq k} (\zeta_k - \zeta_l)^2}.
\]
The formula for $B_k$ follows from
\[
B_k = \lim_{z \to \zeta_k} (f(z) (z - \zeta_k)^2)' = \lim_{z \to \zeta_k} \left( \sum_{j=1}^{2n} \frac{1}{z - x_j} - \sum_{l \neq k} \frac{2}{z - \zeta_l} \right) f(z) (z - \zeta_k)^2.
\]

Part (ii) is obvious. In part (iii) the factorization of $R'$ into the form \eqref{eq:Rprime_product} follows from a combination of $R' = (P'Q - QP')/Q^2$, the critical points of $R$ being precisely the points in $\bfs x$, and the finite poles of $R$ being the zeros of $Q$. The real-valuedness of the multiplicative constant follows from $P$ and $Q$ having real coefficients. Part (iv) follows by contradiction: if $R = P/Q$ has a double pole (or higher) at a zero $\zeta_k$ of $Q$ then $Q(\zeta_k) = Q'(\zeta_k) = 0$ and hence the Wronskian of $P$ and $Q$ vanishes at $\zeta_k$, implying that $\zeta_k$ is also a critical point. 

For part (v) the assumption that no $\bfs x$ is a pole of $R$ is combined with part (iv) to guarantee that all poles are simple. Hence by parts (iii) and (i) the partial fraction expansion \eqref{eq:meromorphic_partial_fraction} applies to $R'$, with the coefficients $A_k$, $B_k$ given by \eqref{eq:Ak_Bk}. However it must hold that $B_k = 0$ for all $k$, which follows from part (ii) and the fact that $R'$ has a rational primitive ($R$ itself). That the $p(z)$ of \eqref{eq:meromorphic_partial_fraction} becomes $p(z) = \1{d=n}$ in the expansion of $R'$ follows from the behavior of \eqref{eq:Rprime_product} at infinity.

Part (vi) also follows by contradiction. Assume that $\zeta_l = x_m$ for some $l,m$. Define $q(z) = R'(z)(z - \zeta_l)^2$, so that by part (iii) we have 
\[
q(z) = \frac{\prod_{j=1}^{2n} (z - x_j)}{\prod_{k \neq l} (z - \zeta_k)^2}.
\]
Since the poles are all distinct the denominator is well-defined at $\zeta_l$, and since $\zeta_l = x_m$ by assumption we have $q(\zeta_l) = 0$. Further, since the $x_j$ are all distinct the factorization of $q(z)$ has only one $(z - x_m) = (z - \zeta_l)$ term, hence $q$ has the expansion 
\[
q(z) = (z - \zeta_l) + O((z - \zeta_l)^2),
\]
in a neighborhood of $\zeta_l$. In particular $q'(\zeta_l) \neq 0$. However this contradicts the conclusion $q'(\zeta_l) = B_l = 0$ that follows from the partial fraction expansion \eqref{eq:meromorphic_partial_fraction} and part (ii).
\end{proof}

Given these results we now prove Theorem \ref{thm: stationary} on the stationary relation.

\begin{proof}[Proof of Theorem \ref{thm: stationary}]
First recall that by assumption the $n+1$ elements of $\bfs \zeta$ are all distinct. Assume that $R \in \CRR_{n+1}(\bfs x)$ has pole set $\bfs \zeta$. Since the elements of $\bfs \zeta$ are distinct then by Lemma \ref{lem:rational_basic} part (vi) we have $\bfs \zeta \cap \bfs x = \emptyset$. Then apply Lemma \ref{lem:rational_basic} part (v) to conclude that $R'$ has the partial fraction expansion \eqref{eq:Rprime_partial_fraction}. The coefficients $A_k$ are given by \eqref{eq:Ak_Bk} and hence $A_k \neq 0$ for each $k$, thanks to $\bfs \zeta \cap \bfs x = \emptyset$. At the same time the formula \eqref{eq:Ak_Bk} for $B_k$ also holds, but we also know $B_k = 0$ from the argument in part (v). Since $A_k \neq 0$ this implies the stationary relation \eqref{eq: stationary}.

Conversely, suppose that $\bfs \zeta$ satisfies the stationary relation \eqref{eq: stationary} and $\bfs \zeta \cap \bfs x = \emptyset$. Combine $\bfs x$ and $\bfs \zeta$ to form the meromorphic function \eqref{eq:meromorphic_product}. Since the elements of $\bfs \zeta$ are assumed to be all distinct the partial fraction expansion \eqref{eq:meromorphic_partial_fraction} holds, and since $\bfs \zeta$ satisfies the stationary relation we have $B_k = 0$ for all $k$. Lemma \ref{lem:rational_basic} part (ii) completes the proof.
\end{proof}

\begin{rmk*}
The principle underlying Theorem \ref{thm: stationary} is that if $R \in \CRR_{n+1}(\bfs x)$ then $R'$ has a rational primitive, and therefore its residue at all poles must be zero. When the poles are distinct from the critical points this principle expresses itself algebraically in the form of the stationary relation. When, however, a pole overlaps with a critical point (in which case the pole is necessarily of order two) then the partial fraction expansion of $R'$ becomes more complicated as does the corresponding algebra. However the underlying principle remains the same.
\end{rmk*}

\begin{proof}[Proof of Theorem \ref{thm: real_locus}]
We first prove that $R \circ g_t^{-1}$ is real rational and in $\CRR_{n+1}(\bfs x(t))$. Since the points $\bfs \zeta$ are assumed to be closed under conjugation and solve the stationary relation Theorem \ref{thm: stationary} guarantees the existence of an $R \in \CRR_{n+1}(\bfs x)$ with pole set $\bfs \zeta$. Moreover, by Lemma \ref{lem:rational_basic} the derivative of $R$ factors (up to a real multiplicative constant) as
\[
R'(z) = \frac{\prod_{j=1}^{2n} (z - x_j)}{\prod_{\zeta_k \neq \infty} (z - \zeta_k)^2}.
\]
Using the integral of motion $N_t(z)$ of Theorem \ref{primitive} we have
\[
N_t(z) = g_t'(z) \frac{\prod_{j=1}^{2n} (g_t(z) - x_j(t))}{\prod_{\zeta_k \neq \infty} (g_t(z) - g_t(\zeta_k))^2 } = \frac{\prod_{j=1}^{2n} (z - x_j)}{\prod_{\zeta_k \neq \infty} (z - \zeta_k)^2} = N_0(z) = R'(z).
\]
Let $f_t = g_t^{-1}$, and since the above holds everywhere evaluate it at $f_t(z)$ to obtain
\begin{align}\label{eq:Rt_prime}
\frac{\prod_{j=1}^{2n} (z - x_j(t))}{\prod_{\zeta_k \neq \infty} (z - g_t(\zeta_k))^2 } = f_t'(z) \frac{\prod_{j=1}^{2n} (f_t(z) - x_j)}{\prod_{\zeta_k \neq \infty} (f_t(z) - \zeta_k)^2} = f_t'(z) R(f_t(z)) = (R \circ f_t)'(z).
\end{align}
Apply Lemma \ref{lem:rational_basic} part (i) to the factorization on the left hand side to conclude that it has the partial fraction expansion \eqref{eq:meromorphic_partial_fraction} with
\[
A_k(t) = \frac{\prod_{j=1}^{2n} (\zeta_k(t) - x_j(t))}{\prod_{l \neq k} (\zeta_k(t) - \zeta_l(t))^2}
\]
and the corresponding formula for $B_k = B_k(t)$. Note the application of Lemma \ref{lem:rational_basic} part (i) used that the points $x_j(t)$ and $g_t(\zeta_k)$ are all distinct (over all $j$ and $k$), which follows from $t < \tau$ and the definition of the time $\tau$. This gives $A_k(t) \neq 0$ at all such $t$. At such times it is also true that $B_k(t) = 0$ for all $k$, since otherwise the partial fraction expansion \eqref{eq:meromorphic_partial_fraction} for $(R \circ f_t)'$ would imply that $R \circ f_t$ has logarithmic singularities. This contradicts rationality of $R$ and smoothness of $g_t^{-1}$ on $\H$. From $B_k(t) = 0$ and $A_k(t) \neq 0$ we immediately conclude, again from Lemma \ref{lem:rational_basic} part (i), that the stationary relation is preserved at all times. Furthermore, from $B_k(t) = 0$ and Lemma \ref{lem:rational_basic} part (ii) we conclude that $(R \circ f_t)'$ has a rational primitive, and the factorization \eqref{eq:Rt_prime} and Theorem \ref{thm: stationary} imply that $R \circ f_t \in \CRR_{n+1}(\bfs x(t))$, as claimed.

Finally, to prove that the hull $K_t$ is a subset of the real locus $\Gamma(R)$, first recall that $g_{t}$ is the unique conformal map from $\H \backslash K_t$ onto $\H$ with the hydrodynamic normalization $g_t(z) = z + o(1)$ as $z \to \infty$. Hence $g_t^{-1}$ maps $\R$ onto the subset $K_t$. On the other hand $R \circ g_t^{-1}$ being real rational implies that it maps the real line into the real line. Consequently $R$ must map $K_t$ back to the real line, which implies that $K_t$ is a subset of $\Gamma(R)$. 
\end{proof}

We now consider an additional result on the interplay between poles and critical points and their evolution under Loewner flow. 

\begin{lem}\label{lem: gt_algebraic}
Let $\bfs x = \{ x_1, \ldots, x_{2n} \}$ be real and distinct and $\bfs \zeta = \{ \zeta_1, \ldots, \zeta_{n+1} \} \subset \widehat{\C}$ be symmetric under conjugation and solve the stationary relation \eqref{eq: stationary}. Let $g_t$ be the Loewner chain of any $\bfs \nu$-superposition of the SLE$(0, x_j, \bfs \rho_j)$ flows of Theorem \ref{thm: real_locus}. Then for any $t < \tau$ and any $z$ such that $t < \tau_z$ we have
\begin{align}\label{eq:Rt_partial_fraction}
g_t(z) - \sum_{k=1}^{n+1} \frac{A_k(t)}{g_t(z) - g_t(\zeta_k)} = z - \sum_{k=1}^{n+1} \frac{A_k(0)}{z - \zeta_k},
\end{align}
where $A_k(t)$ is defined by
\[
A_k(t) = \frac{\prod_{j=1}^{2n} (g_t(\zeta_k) - x_j(t))}{\prod_{l \neq k} (g_t(\zeta_k) - g_t(\zeta_l))^2}.
\]
\end{lem}

This lemma is useful for finding $\Gamma(R)$ given its poles and critical points. Indeed, if we insert $z = \gamma_j(t)$ into \eqref{eq:Rt_partial_fraction} for some $j=1,\ldots,2n$ then the left hand side is real-valued. By searching for $z \in \H$ that make the right hand side real-valued we can locate the tips of the curves as the Loewner evolution unfolds. We use this technique frequently in the examples. It would be interesting to see if this can lead to an efficient numerical method for plotting $\Gamma(R)$. The proof of Lemma \ref{lem: gt_algebraic} is simple.

\begin{proof}
Combine $R \circ g_t^{-1} \circ g_t = R$ with equation \eqref{eq:Rt_prime} for $R \circ g_t^{-1}$.
\end{proof}

\subsection{Solutions to the Null Vector Equations via Real Rational Functions\label{sec: NV0_solns}}

Now we prove Theorem \ref{thm: Z}, that our $U_{\alpha, j}$ of \eqref{eq: U} are solutions to the null vector equations \eqref{eq: NV0}. Recall that we represent our $U_{\alpha, j}$ via the canonical element of $\CRR_{n+1}(\bfs x; \alpha)$, the purpose of which is to make the functions $U_{\alpha,j}$ a function of $(\bfs x; \alpha)$ alone. The proof uses that the pole functions $\zeta_{\alpha,k}(\bfs x)$ are smooth functions of the $\bfs x$ variables but does not require any explicit formulas for the derivatives. We also prove that our functions $U_{\alpha,j}$ satisfy a system of three algebraic equations called the \textbf{conformal Ward identities}. These are algebraic expressions of the fact that the $U_j$ transform in a particular way under M\"{o}bius transformations of the half-plane to itself; see Section \ref{sec: commutation} for more.

\begin{proof}[Proof of Theorem \ref{thm: Z}]
We first prove that the $U_{\alpha,j}$ of \eqref{eq: U} satisfy the null vector equations. Write $U_{\alpha,j} = U_j$ throughout. Real-valuedness of $U_j$ follows from $x_k \in \R$ and the $\zeta_k$ occurring in complex conjugate pairs. For the translation invariance, note that the pole functions $\zeta_{\alpha,k}$ clearly satisfy $\zeta_{\alpha,k}(\bfs x + h) = \zeta_{\alpha,k}(\bfs x) + h$ for $h \in \R$, since if $R \in \CRR_{n+1}(\bfs x; \alpha)$ is canonical then $R(z-h) \in \CRR_{n+1}(\bfs x + h; \alpha)$ and is also canonical. Similarly, $r R(r^{-1} z) \in \CRR_{n+1}(r \bfs x; \alpha)$ for $r > 0$ and is canonical, hence $\zeta_{\alpha,k}(r \bfs x) = r \zeta_{\alpha,k}(\bfs x)$. The homogeneity of $U_{\alpha,j}$ follows from these properties of the pole function and the definition \eqref{eq: U}.

For the proof that the $U_{j}$ satisfy the null vector equations it is convenient to introduce 
\begin{equation} \label{eq: uU}
u_j:= U_j - \sum_{k\ne j} \frac2{x_j-x_k}.
\end{equation}
All we need to check is that $u_j$'s satisfy the algebraic equations 
\begin{equation} \label{eq: u}
\frac12\, u_j^2 + 2 \sum_{k\ne j} \frac{u_j-u_k}{x_j-x_k}=0, \quad j=1,\ldots,2n.
\end{equation}
Using the stationary relation~\eqref{eq: stationary}, we have 
\begin{align*}
\frac18\Big(\frac12\, u_j^2 &+ 2 \sum_{k\ne j} \frac{u_j-u_k}{x_j-x_k}\Big)=\Big(\sum_{l=1}^n \frac1{\zeta_l-x_j}\Big)^2+\sum_{k\ne j}\sum_{l=1}^n \frac1{(\zeta_l-x_j)(\zeta_l-x_k)}\\
&=\sum_{l=1}^n \frac1{\zeta_l-x_j} \sum_{k=1}^{2n}\frac1{\zeta_l-x_k} + 2\sum_{k<l}\frac1{(\zeta_k-x_j)(\zeta_l-x_j)}\\
&=\sum_{l=1}^n \frac1{\zeta_l-x_j}\sum_{k\ne l}\frac2{\zeta_l-\zeta_k} + 2\sum_{k<l}\frac1{(\zeta_k-x_j)(\zeta_l-x_j)}=0.
\end{align*}
Here we used the identity 
\[
\sum_{k<l}\frac1{(\zeta_k-x_j)(\zeta_l-x_j)} = \sum_{k<l} -\frac1{(\zeta_k-x_j)(\zeta_k-\zeta_l)} +\frac1{(\zeta_l-x_j)(\zeta_k-\zeta_l)}.
\]

Now we prove that $U_{\alpha,j} = \partial_j \log Z_{\alpha}$ for $j=1,\ldots,2n$. Write $Z = Z_{\alpha}$ throughout. Strict positivity of $Z$ follows from the $x_j$ assumed to be real and distinct and the $\zeta_k$ always appearing in complex conjugate pairs. For the differentiability of $Z$ we observe that smoothness of the pole functions follows from the fact that for each $\bfs x$ they obey the stationary relation \eqref{eq: stationary}. Indeed, the stationary relation \eqref{eq: stationary} together with the implicit function theorem implies that there is a neighborhood of $\bfs x$ on which the poles $\zeta_{\alpha,k} = \zeta_{\alpha,k}(\bfs x)$ are smooth functions of the critical points. We use the existence and continuity of the first and second order partial derivatives to compute $\partial_j \log Z$. Indeed, using expression \eqref{eq: Z} for $Z$ and computing $\dell_{j} \log Z$ directly we have
\[
\dell_j \log Z = \sum_{k \neq j} \frac{2}{x_j - x_k} + \sum_{l=1}^{n} \frac{4}{\zeta_l - x_j} + 4 \sum_{k=1}^{2n} \sum_{l=1}^n \frac{\dell_{x_j} \zeta_l}{x_k - \zeta_l} + 8 \!\!\!\! \sum_{1 \leq l < m \leq n} \!\! \frac{\dell_{x_j} \zeta_l - \dell_{x_j} \zeta_m}{\zeta_l - \zeta_m}.
\]
For the last two terms use the stationary relation \eqref{eq: stationary} to obtain
\[
\sum_{l=1}^n \dell_{x_j} \zeta_l \sum_{k=1}^{2n} \frac{1}{x_k - \zeta_l} = - \sum_{l=1}^n \dell_{x_j} \zeta_l \sum_{m \neq l} \frac{2}{\zeta_l - \zeta_m} = -2 \!\!\!\! \sum_{1 \leq l < m \leq n} \!\! \frac{\dell_{x_j}\zeta_l - \dell_{x_j} \zeta_m}{\zeta_l - \zeta_m}.
\]
Thus the last two terms in the above expression vanish, thereby proving $\dell_j \log Z = U_j$. The $(3,0)$-differential property of $Z_{\alpha}$ simply means that
\[
Z_{\alpha}(\phi(\bfs x)) \prod_{i=1}^{2n} \phi'(x_i)^3 = Z_{\alpha}(\bfs x)
\]
holds for all M\"{o}bius transforms $\phi$ of $\H$ to itself. Here $\phi(\bfs x) = (\phi(x_1), \ldots, \phi(x_{2n}))$. This property is verified directly by checking it for translations, dilations, and inversions.
\end{proof}

The $(3,0)$-differential property of $Z_{\alpha}$ is also necessary to obtain M\"{o}bius invariance of the curves, which we explain in Section \ref{sec: commutation}. The following system of conformal Ward identities is also an expression of conformal invariance but in this section we focus on the algebraic verification of the identities. The proof relies on another simple relationship between the poles and critical points that we call the \textbf{centroid relation}, which we prove next.

\begin{lem}\label{lem:arithmetic_mean}
Let $\bfs x = \{ x_1, \ldots, x_{2n} \} \subset \R$ be distinct and $\alpha$ be a topological link pattern connecting them. Let $R \in \CRR_{n+1}(\bfs x; \alpha)$ be canonical. If the pole function values $\zeta_{\alpha,1}(\bfs x), \ldots, \zeta_{\alpha,n}(\bfs x)$ are all distinct then
\begin{equation} \label{eq: means}
\frac{1}{n} \sum_{k=1}^{n} \zeta_{\alpha, k}(\bfs x) = \frac{1}{2n} \sum_{j=1}^{2n} x_j.
\end{equation}
\end{lem}

\begin{proof}
By the definition of our canonical element we know that $R'$ has the partial fraction expansion \eqref{eq:Rprime_partial_fraction} with $d = n$ and with $\zeta_k$ representing the pole functions. The same partial fraction expansion implies $R'(z) = 1 + O(z^{-2})$ near $z = \infty$, and therefore
\[
\lim_{z \to \infty} z (R'(z) - 1) = 0.
\]
On the other hand, we may use the factorization \eqref{eq:Rprime_partial_fraction} of $R'$ to write $R'(z) - 1$ as
$$
R'(z) - 1 = \frac{(2\sum_{k=1}^n \zeta_k - \sum_{j=1}^{2n} x_j)z^{2n-1} + \cdots}{\prod_{k=1}^{n}(z-\zeta_k)^2} \sim \left( 2 \sum_{k=1}^n \zeta_k - \sum_{j=1}^{2n} x_j \right) z^{-1}.
$$
Combining the last two facts we conclude that $2\sum_{k=1}^n \zeta_k = \sum_{j=1}^{2n} x_j$, as claimed. 
\end{proof}

\begin{thm}\label{CWI}
For distinct real points $\bfs x$ and a link pattern $\alpha$ connecting them the functions $U_{\alpha,j}$ of \eqref{eq: U} satisfy the \textbf{conformal Ward identities}
\begin{align}\label{eq: CWI0}
\sum_{j=1}^{2n} U_{\alpha, j} = 0, \quad \sum_{j=1}^{2n} x_j U_{\alpha, j} = -6n, \quad \sum_{j=1}^{2n} x_j^2 U_{\alpha, j} = - 6 \sum_{j=1}^{2n} x_j.
\end{align}
\end{thm}

\begin{proof}[Proof of Theorem \ref{CWI}] For convenience we drop the dependence on $\alpha$ throughout. The first conformal Ward identity holds by symmetry:
\[
\sum_{j=1}^{2n} U_j = \sum_{j=1}^{2n}\sum_{k\ne j} \frac2{x_j-x_k} + \sum_{j=1}^{2n}\sum_{k=1}^n \frac4{\zeta_k-x_j}= \sum_{k=1}^n \sum_{l\ne k}\frac8{\zeta_k-\zeta_l} = 0.
\]
To find a sufficient and necessary condition for the second identity $\sum_{j=1}^{2n} x_j U_j= -6n$, we compute 
\[
\sum_{j=1}^{2n}\sum_{k\ne j} \frac{x_j}{x_j-x_k} = 2n^2-n, \qquad 
\sum_{j=1}^{2n}\sum_{k=1}^n \frac{x_j}{\zeta_k-x_j } = -2n^2 + \sum_{k=1}^n \sum_{l\ne k}\frac{2\zeta_k}{\zeta_k-\zeta_l} = -2n^2 +n(n-1).
\]
Thus we find $\sum_{j=1}^{2n} x_j U_j= 4n^2 - 2n - 8n^2 +4n(n-1) =-6n$, as claimed. For the last of the conformal Ward identities we compute 
\[
\sum_{j=1}^{2n} \sum_{k\ne j} \frac{x_j^2}{x_j-x_k} = \sum_{j<k}(x_j+x_k)= (2n-1)\sum_{j=1}^n x_j
\]
and therefore 
\[
\sum_{j=1}^{2n}\sum_{k=1}^n \frac{x_j^2}{\zeta_k-x_j} = -n\sum_{j=1}^{2n} x_j -2n \sum_{k=1}^n \zeta_k + \sum_{k=1}^n \sum_{l\ne k} \frac{2\zeta_k^2}{\zeta_k-\zeta_l} = -n\sum_{j=1}^{2n} x_j -2 \sum_{k=1}^n \zeta_k. 
\]
Now the last conformal Ward identity follows from the centroid relation of equation \eqref{eq: means}.
\end{proof}

\begin{rmk*}
Section \ref{sec: commutation} will show that the $U_{\alpha,j}$ satisfying the conformal Ward identities implies that the curves they generate are invariant under M\"{o}bius transformations. However it is already apparent that the curves that make up $\Gamma(R)$ have the correct M\"{o}bius invariance properties by considering \textit{pre}-compositions of $R \in \CRR_{n+1}(\bfs x)$. For $\phi \in \PSL(2, \R)$ it is straightforward to verify that $R \circ \phi^{-1} \in \CRR_{n+1}(\bfs \phi(x))$ and that $\Gamma(R \circ \phi^{-1}) = \phi(\Gamma(R))$. Therefore any formulation of the Loewner flow for real loci of real rational functions should obey this type of M\"{o}bius invariance. We also note that the term ``M\"{o}bius invariance'', although standard, is not necessarily the best description. A better term might be ``M\"{o}bius commutation'', since it indicates that growing $\Gamma(R)$ via Loewner flow and then mapping the curves via a M\"{o}bius transform is equivalent to first pre-composing the rational function and then applying the Loewner chain corresponding to the pre-composition. As usual this will only hold up to a time change.
\end{rmk*}

Using the canonical element for the definition of $U_{\alpha,j}$ is largely a matter of convenience, as the next lemma shows that the $U_{\alpha,j}$ are constant on (most of) the equivalence class $\CRR_{n+1}(\bfs x; \alpha)$.

\begin{lem}\label{lem:U_on_CRR}
Let $\bfs x = \{ x_1, \ldots, x_{2n} \} \subset \R$ be distinct and $\alpha$ be a topological link pattern connecting them. Define $U_{\alpha,j}$ on $\CRR_{n+1}(\bfs x; \alpha)$ by
\begin{align}\label{eq:U_on_CRR}
U_{\alpha,j}(R) = \sum_{k \neq j} \frac{2}{x_j - x_k} - \sum_{\zeta_k \in \bfs \zeta(R)} \frac{4}{x_j - \zeta_k},
\end{align}
where $\bfs \zeta(R)$ is the set of poles of $R$. Apart from a subset of $\CRR_{n+1}(\bfs x; \alpha)$ with $\PSL(2, \R)$ Haar measure zero the map $U_{\alpha,j}$ is well-defined, constant, and agrees with \eqref{eq: U}.
\end{lem}

\begin{proof}
As written the formula \eqref{eq:U_on_CRR} is ill-defined at elements of $\CRR_{n+1}(\bfs x; \alpha)$ for which the real rational function has a pole that overlaps with a critical point. In that case the pole is of order two and $R'$ has a pole of order three at the critical point. However there are only finitely many post-inversions that map any given element of $\CRR_{n+1}(\bfs x; \alpha)$ into such an element, from which it follows that the subset of such elements has zero measure under the Haar measure on $\PSL(2, \R)$.

For canonical $R$ this definition agrees with \eqref{eq: U}. For the rest of the equivalence class it is sufficient to consider post-compositions of $R$ by elements of $\PSL(2, \R)$, and this group is generated by translations, dilations, and inversions. In the first two cases it is easy to see that the pole set is preserved under post-translations and post-dilations of $R$, i.e. $\bfs \zeta(R + h) = \bfs \zeta(R)$ for $h \in \R$ and $\bfs \zeta(rR) = \bfs \zeta(R)$ for $r > 0$. Since the critical points of $R$ are also preserved under these operations it follows that the $U_{\alpha,j}$ are also preserved. Under post-inversion we have $\bfs \zeta(1/R) = \mathfrak{Z}(R)$, where $\mathfrak{Z}(R)$ is the zero set of $R$. The zeros, poles, and critical points satisfy
\[
\sum_{\zeta_k \in \bfs \zeta(R)} \frac{1}{x_j - \zeta_k} = \sum_{\eta_m \in \mathfrak{Z}(R)} \frac{1}{x_j - \eta_m}, \quad j=1,\ldots,2n,
\]
which follow from $R'(x_j) = 0$ and $R'/R$ given by
\[
\frac{R'(z)}{R(z)} = \sum_{\eta_m \in \mathfrak{Z}(R)} \frac{1}{z - \eta_m} - \sum_{\zeta_k \in \bfs \zeta(R)} \frac{1}{z - \zeta_k}.
\]
Consequently $U_{\alpha,j}(1/R) = U_{\alpha,j}(R)$, proving that $U_{\alpha,j}$ is constant. 
\end{proof}

\begin{rmk*}
See Example \ref{eg: rainbow} for a particular canonical rational function for which the pole overlaps with a critical point. Since the set of such singularities has measure zero within the equivalence class it is natural to extend the constant value of $U_{\alpha,j}$ to these elements, and this is the choice we adopt. From the Loewner chain point of view this is certainly the correct choice, since the set $\Gamma(R)$ that is being grown does not depend on the particular representative of the equivalence class. This choice also naturally extends our SLE$(0, x_j, \bfs \rho_j)$ dynamics of Theorem \ref{thm: real_locus} past the times at which a pole collides with a critical point. Indeed if $R$ is the initial rational function of Theorem \ref{thm: real_locus} then the flow $t \mapsto R \circ g_t^{-1}$ may eventually run into one of these singularities in $\CRR_{n+1}(\bfs x(t))$. At such a time simply post-compose by an element of $\PSL(2, \R)$ such that the poles and critical points no longer overlap, thereby allowing one to continue on the SLE$(0, x_j, \bfs \rho_j)$ dynamics. The choice of the particular post-composition does not matter precisely because the $U_{\alpha,j}$ are constant on the equivalence classes.
\end{rmk*}

\subsection{Geodesic Multichord Property\label{sec:geo_multichord}}

For $R \in \CRR_{n+1}(\bfs x)$ Theorem \ref{thm: real_locus} shows that the map $t \mapsto R \circ g_t^{-1}$ remains real rational before any two elements in $\bfs x \cup \bfs \zeta$ collide, where $\bfs \zeta = \bfs \zeta(R)$ is the pole set of $R$. The discussion at the end of the last section shows that the SLE$(0, x_j, \bfs \rho_j)$ dynamics that drive the evolution of $R \circ g_t^{-1}$ can be continued on past collision times between $\bfs x$ and $\bfs \zeta$ elements. In this section we study the limit of $R \circ g_t^{-1}$ when two elements of $\bfs x$ collide, and this analysis leads to the proof of the geodesic multichord property of Theorem \ref{thm:RRF_are_geos}.

Our arguments make use of the structure of $\Gamma(R)$ described in Lemma \ref{lem: RRF_locus}. We use that any $R \in \CRR_{n+1}(\bfs x)$ consists of $n$ curves connecting the points in $\bfs x$ according to a link pattern $\alpha$ and we assume that we know $\alpha$. The simplest situation for merging is when one grows a curve that has no other curve below it. 

\begin{prop}\label{prop:merging}
Let $\bfs x = \{ x_1, \ldots, x_{2n} \} \subset \R$ be distinct, relabeling them such that $x_1 < \ldots < x_{2n}$ if necessary. Let $R \in \CRR_{n+1}(\bfs x)$ and $\alpha$ be the link pattern for $\Gamma(R)$. If $\{x_j, x_{j+1} \} \in \alpha$ then
\[
\lim_{t \uparrow \tau} R \circ g_{t}^{-1} \in \CRR_{n}( g_{\tau}(\bfs{\hat{x}}_j); \hat{\alpha}_j),
\]
where $\hat{\bfs{x}}_j = \bfs x \backslash \{ x_j, x_{j+1} \}$, $\hat{\alpha}_j = \alpha \backslash \{ x_j, x_{j+1} \}$, $g_t$ is the Loewner chain of the SLE$(0, x_j, \bfs \rho_j)$ process, $\bfs \rho_j$ is given by \eqref{eq: rhoj} with $\zeta_k$ the poles of $R$, and $\tau$ the time at which $x_j(t)$ collides with $g_t(x_{j+1})$ under the flow.
\end{prop}

We have stated the result for growth from $x_j$ (i.e. under the SLE$(0, x_j, \bfs \rho_j)$ Loewner chain) but the proof makes it self-evident that the analogous result holds when the growth occurs at $x_{j+1}$ (i.e. under the SLE$(0, x_{j+1}, \bfs \rho_{j+1})$ Loewner chain) and $\tau$ is the collision time between $x_{j+1}(t)$ and $g_t(x_j)$.

\begin{proof}
First observe that existence of at least one $j$ such that $\{ x_j, x_{j+1} \} \in \alpha$ is guaranteed by the non-crossing property of the link pattern $\alpha$. Furthermore, given the existence of a single $R \in \CRR_{n+1}(\bfs x; \alpha)$ there exists a $\phi \in \PSL(2, \R)$ such that $\phi \circ R$ has a pole in $(x_j, x_{j+1})$. To see this let $\eta$ be the SLE$(0, x_j, \bfs \rho_j)$ curve from $x_j$ to $x_{j+1}$ and let the $F$ be the face of $\Gamma(R)$ whose boundary is $\eta \cup [x_j, x_{j+1}]$. This face of $\Gamma(R)$, like every other, must have exactly one pole of $R$ on its boundary. This is because $R$ is a bijection of each face onto $\H$ or $-\H$ that extends continuously to the boundary of the face. The boundary being a subset of the real locus means that $R$ maps it to $\widehat{\R}$, hence there must be a pole on the boundary. For this particular face the pole must lie either on $\eta$ or in $[x_j, x_{j+1}]$. The family of post-compositions 
\[
x \mapsto \frac{-1}{R-x}, \quad x \in \R,
\]
continuously maps the pole around $\Gamma(R)$, so choose an $x$ and $\phi(z) = -1/(z-x)$ such that $\phi \circ R$ has the pole in $(x_j, x_{j+1})$. It is convenient to represent the SLE$(0, x_j, \bfs \rho_j)$ process by $\phi \circ R$, meaning the driving function $x_j(t)$ evolves as
\[
\dot{x}_j = U_{\alpha,j}(\bfs x)
\]
but with $U_{\alpha,j}$ determined by \eqref{eq:U_on_CRR} and the pole set $\bfs \zeta = \{ \zeta_1, \ldots, \zeta_{n+1} \}$ of $\phi \circ R$. By Theorem \ref{thm: real_locus} we know that $\phi \circ R \circ g_t^{-1} \in \CRR_{n+1}(\bfs x(t))$ for any $t < \tau$, and from the proof of that same theorem we know that $\phi \circ R \circ g_t^{-1}$ has pole set $\{ g_t(\zeta_1), \ldots, g_t(\zeta_{n+1}) \}$. The proof also shows that the stationary relation is preserved at time $t$, hence
\begin{align}\label{eq:stationary_at_t}
\sum_{l=1}^{2n} \frac{1}{g_t(\zeta_k) - x_l(t)} = \sum_{l \neq k} \frac{2}{g_t(\zeta_k) - g_t(\zeta_l)}, \quad k=1,\ldots,n+1.
\end{align}
Now if $\zeta_m$ is the pole of $\phi \circ R$ in $(x_j, x_{j+1})$ then for all $t < \tau$ we have $x_j(t) \leq g_t(\zeta_m) \leq x_{j+1}(t)$ since the Loewner chain preserves order on the real line. Consider the $m$th equation in \eqref{eq:stationary_at_t} (i.e. the one belonging to $\zeta_m$). The left hand side of this equation stays finite in the $t \uparrow \tau$ limit iff
\begin{align}\label{eq:pole_limit}
g_t(\zeta_m) \sim \tfrac{1}{2}(x_j(t) + x_{j+1}(t)),
\end{align}
since $x_{j+1}(t) - x_j(t) \to 0$ as $t \uparrow \tau$. But the left hand side must stay finite in this limit because the right hand side does; this follows because the other poles $\zeta_k, k \neq m$, are on the boundaries of the other faces, hence are separated from $\zeta_m$ initially and must remain separated as $t \uparrow \tau$. Thus the $t \uparrow \tau$ limit of the equations \eqref{eq:stationary_at_t} for $1 \leq k \leq n+1$ (but $k \neq m$) is
\[
\frac{2}{g_{\tau}(\zeta_k) - x_j(\tau)} + \sum_{\substack{l=1 \\ l \neq j,j+1}}^{2n} \frac{1}{g_{\tau}(\zeta_k) - x_l(\tau)} = \frac{2}{g_{\tau}(\zeta_k) - x_j(\tau)} + \sum_{l \neq k,j} \frac{2}{g_{\tau}(\zeta_k) - g_{\tau}(\zeta_l)}.
\]
Thus the $n$ points $\{ g_{\tau}(\zeta_k), k \neq m \}$, are a solution to the stationary relation for the $2n-2$ points $\{ x_l(\tau), 1 \leq l \leq 2n, l \neq j,j+1 \}$. By Theorem \ref{thm: stationary} there exists an $R_o \in \CRR_{n}(\hat{\bfs x}_j)$ with pole set $\{g_{\tau}(\zeta_k) : k \neq m \}$. To see that $R_o$ is the limit of $R \circ g_{t}^{-1}$ we may assume that $R$ initially has the factorization \eqref{eq:Rprime_product}, from which Theorem \ref{thm: real_locus} implies that
\[
(R \circ g_t^{-1})'(z) = \frac{\prod_{l = 1}^{2n} (z - x_l(t))}{\prod_{\zeta_k \neq \infty} (z - g_t(\zeta_k))^2} 
\]
for $t < \tau$. As $t \uparrow \tau$ the asymptotic relation \eqref{eq:pole_limit} implies that
\[
\lim_{t \uparrow \tau} \frac{(z - x_j(t))(z - x_{j+1}(t))}{(z - g_t(\zeta_m))^2} = 1,
\]
from which it follows that
\[
\lim_{t \uparrow \tau} (R \circ g_t^{-1})'(z) = \frac{\prod_{l \neq j,j+1} (z - x_l(\tau))}{\prod_{\substack{\zeta_k \neq \infty \\ k \neq m} } (z - g_{\tau}(\zeta_k))^2}.
\]
Therefore the limit is a real rational function with critical points $\{ x_l(\tau)$, $l \neq j,j+1 \}$ and poles $\{ g_{\tau}(\zeta_k)$, $k \neq m \}$, since they are already known to satisfy the stationary relation. But this rational function is precisely $R_o$. Therefore the $t \uparrow \tau$ limit of $R \circ g_t$ is in $\CRR_n(g_{\tau}(\hat{\bfs{x}}_j))$. It only remains to show that $\Gamma(R_0) = \hat{\alpha}_j$, which follows from the simple topological fact that deleting the curve anchored at $x_j$ from $\Gamma(R)$ endows the remaining curves with the $\hat{\alpha}_j$ pattern.
\end{proof}

\begin{rmk*}
The assumption of Proposition \ref{prop:merging} that the curve connects adjacent critical points is not crucial. It simplifies the proof since it only requires analyzing the asymptotic behavior of two critical points and one pole under a merging event. If we apply the Loewner chain for a curve connecting non-adjacent critical points then we simply have to analyze the limiting behavior of more critical points and poles. In that case the number of equations that drop out of the stationary relation is one-half the number of critical points between and including the endpoints of the curve. The degree of the limiting rational function also drops by this amount. When the curve completes the poles and critical points that are not swallowed by the curve become the poles and critical points of the limiting rational function. 
\end{rmk*}

The next result is basic but important for the proof of Theorem \ref{thm:RRF_are_geos}.

\begin{lem}\label{lem:base_geo_chord}
Let $x,y \in \R$ be distinct and suppose $R \in \CRR_{2}(\{x,y\})$. Then $\Gamma(R)$ is the hyperbolic geodesic connecting $x$ and $y$ in $\H$.
\end{lem}

\begin{proof}
The set $\CRR_2(\{x,y\})$ is non-empty since we can write down the canonical rational function:
\[
z \mapsto z + \frac{(x-y)^2/4}{z - (x+y)/2}.
\]
The poles of this rational function solve the stationary relation with $\zeta_1 = \infty$, $\zeta_2 = (x+y)/2$. Example \ref{eg: circle} computes directly that the real locus of this rational function is the real line together with the circle that passes through $x$ and $y$ and intersects the real line at right angles. Thus the arc connecting $x$ and $y$ in $\H$ is a semi-circle, which is precisely the hyperbolic geodesic. 

Now we claim that for $R \in \CRR_2(\{x,y\})$ we can always find a $\phi \in \PSL(2, \R)$ such that $\phi \circ R$ is equal to the rational function above, and this proves that $\CRR_2(\{x,y\})$ has only one equivalence class. Note that 
\[
(\phi \circ R)' = (\phi' \circ R) R' = \frac{P'Q - Q'P}{(cP + dQ)^2} = \frac{(z-x)(z-y)}{(cP + dQ)^2},
\]
where $R = P/Q$ and $\phi(z) = (az + b)/(cz + d) \in \PSL(2,\R)$. Since $\deg P, \deg Q \in \{ 1, 2 \}$ and at least one of the two has degree equal to two it is straightforward to find $c,d \in \R$ such that $cP + dQ = z - (x+y)/2$ (this comes down to solving a $2 \times 2$ linear system). This shows that $\phi \circ R$ has the same derivative as the canonical rational function above, which is sufficient to conclude that $\phi \circ R$ can be post-composed into the canonical rational function itself.
\end{proof}

Now we come to the proof of Theorem \ref{thm:RRF_are_geos}. It is recursive and based on a descent argument, and Lemma \ref{lem:base_geo_chord} acts as the base case. The recursive part uses a repeated application of Proposition \ref{prop:merging} together with the well known conformal invariance property of hyperbolic geodesics.

\begin{figure}
\begin{center}
\includegraphics[scale=.54]{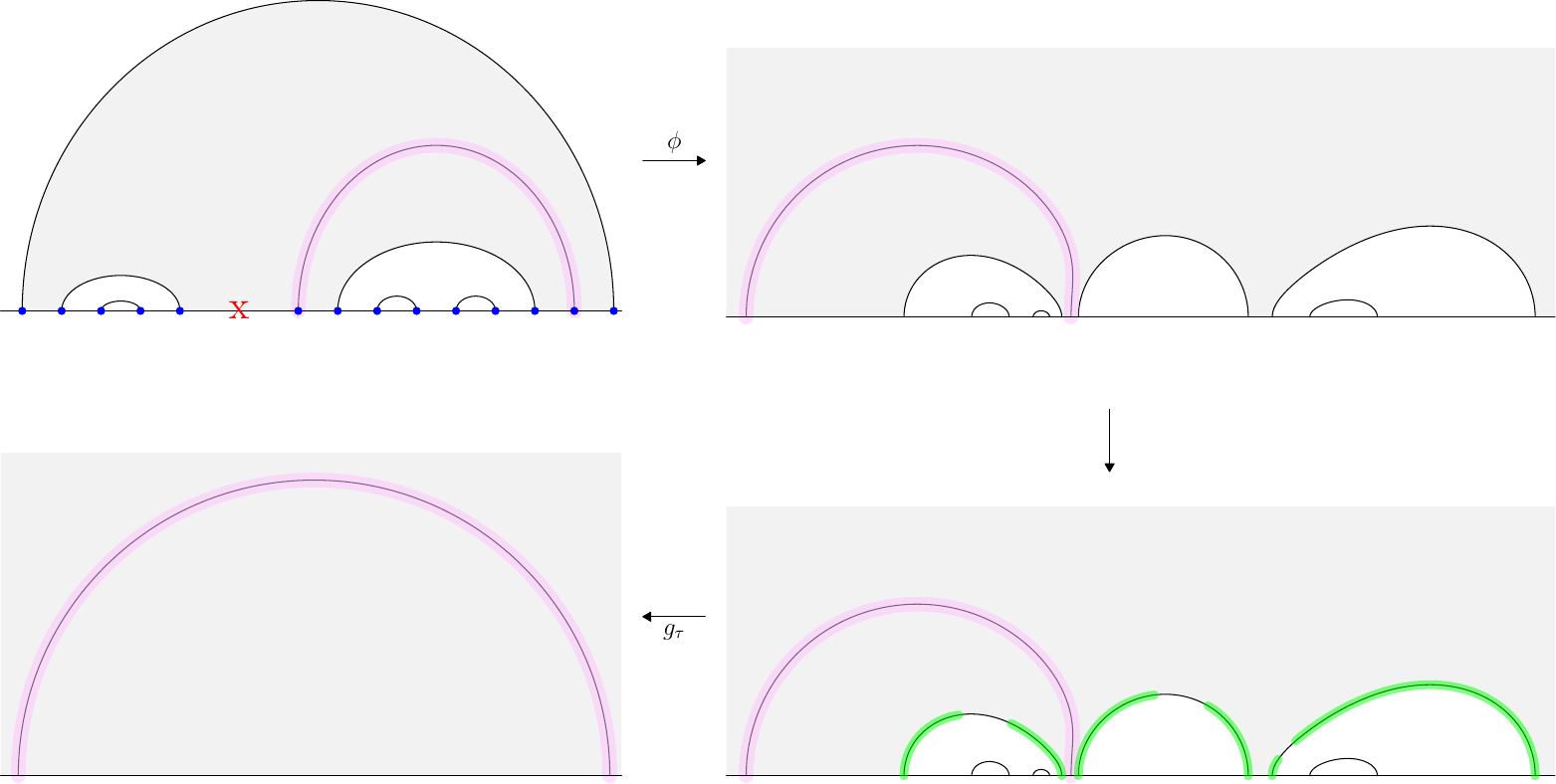}
\captionsetup{width=.95\linewidth}
\caption{A schematic description of the proof of Theorem \ref{thm:RRF_are_geos}. Here $n = 7$ and the critical points are indicated by the blue points in the top left figure. From $R \in \CRR_{n+1}(\bfs x)$ we know that the real locus (in $\overline{\H}$) must connect the blue points according to some non-crossing link pattern; the pattern given in the top left is one such possibility. Note that the shown embedding is only for illustration purposes; it is not the real locus of an actual rational function (such a thing is difficult to compute). The goal is to show that the pink highlighted curve is a geodesic in the grey shaded region. First apply a M\"{o}bius transform of $\H$ to itself such that the shaded region becomes unbounded; this is done so that we can apply the Loewner chain in $\H$. Here the M\"{o}bius transform is an inversion that sends the point marked with an x to infinity. The top right diagram is the image of the curves and the shaded region under this M\"{o}bius transform. The black curves and the pink curve form $\Gamma(R \circ \phi^{-1})$. Now apply Proposition \ref{prop:merging} repeatedly to remove the black curves in the top right, choosing to remove a ``bottom'' curve at each step. At each step the remaining curves are a part of the real locus of a real rational function. In the final step (on the bottom left) everything under the green curves has been removed and only the image of the pink curve remains. We directly compute that it is a hyperbolic geodesic in $\H$. Inverting the intermediate conformal maps and using the conformal invariance of hyperbolic geodesics shows that the original pink curve was a hyperbolic geodesic in the original shaded domain.}
\end{center}
\end{figure}

\begin{proof}[Proof of Theorem \ref{thm:RRF_are_geos}]
Let $\bfs \eta = (\eta_1, \ldots, \eta_n)$ be the ensemble of $n$ curves in $\H$ that connect the points in $\bfs x$ in a non-crossing way. The goal is to show that any $\eta_j$ is the hyperbolic geodesic in the connected component of
\[
\H \backslash \bigcup_{k \neq j} \eta_k
\]
that it belongs to. More precisely, we want to show that it is the hyperbolic geodesic among all curves in that component that connect the two endpoints of $\eta_j$.

First, we will assume without loss of generality that $\eta_j$ is on the boundary of the unbounded connected component of $\H \backslash \Gamma(R)$. If not then there exists a $\phi \in \PSL(2, \R)$ such that $\phi(\eta_j)$ is on the boundary of the unbounded connected component of $\H \backslash \phi(\Gamma(R))$, and since $\phi(\Gamma(R)) = \Gamma(R \circ \phi^{-1})$ and $R \circ \phi^{-1} \in \CRR_{n+1}(\phi(\bfs x))$ one may apply the argument below starting from the real rational function $R \circ \phi^{-1}$ instead. The existence of such a $\phi$ follows by taking $\phi(z) = -1/(z-y)$, where $y \in \R$ is in the interval of $\R \backslash \bfs x$ that is immediately to the left of the left endpoint of $\eta_j$.

Now we begin the descent argument. Topological considerations show that there must exist a curve $\eta_k$ in $\Gamma(R)$ that connects two adjacent elements of $\bfs x$ (adjacent in the usual ordering of $\R$), and this curve can be chosen to be distinct from $\eta_j$. Distinctness from $\eta_j$ follows from the assumption that $\eta_j$ is on the boundary of the unbounded connected component of $\H \backslash \Gamma(R)$. Indeed there is only one link pattern, the nested ``rainbow'' pattern, that has only a single curve connecting two adjacent points of $\bfs x$, and $\eta_j$ cannot be that single curve since we assume that $\eta_j$ borders the unbounded connected component of $\H \backslash \Gamma(R)$. 

Apply the SLE$(0, \bfs \rho)$ Loewner chain of Theorem \ref{thm: real_locus} that generates $\eta_k$ (one can apply it from either endpoint of $\eta_k$). Let $g_{1,t}$ be the Loewner chain, $\tau_1$ be the time at which the curve completes, and denote $h_1 = g_{1, \tau_1}$. By Proposition \ref{prop:merging} we have that $R \circ h_1^{-1}$ is a real rational function of degree $n$ with $2n-2$ critical points. Furthermore we have that
\[
\Gamma(R \circ h_1^{-1}) = h_1(\Gamma(R))
\]
so that $h_1(\eta_j)$ is one of the $n-1$ curves in $\Gamma(R \circ h_1^{-1})$. Since the Loewner chain preserves the point at infinity it follows that $h_1(\eta_j)$ is on the boundary of the unbounded connected component of $\H \backslash \Gamma(R \circ h_1^{-1})$. Therefore we can apply the descent argument again. We apply the Loewner flow for a curve in $\Gamma(R \circ h_1^{-1})$ that connects adjacent points among the $2n-2$ critical points of $R \circ h_1^{-1}$ but that is distinct from $h_1(\eta_j)$. This gives a second conformal map $h_2$ with the properties that $R \circ h_1^{-1} \circ h_2^{-1}$ is a real rational function of degree $n-1$ with $2n-4$ critical points, that
\[
\Gamma(R \circ h_1^{-1} \circ h_2^{-1}) = (h_2 \circ h_1)(\Gamma(R)),
\]
and that $(h_2 \circ h_1)(\eta_j)$ is one of the $n-2$ curves in $\Gamma(R \circ h_1^{-1} \circ h_2^{-1})$ and is on the boundary of the unbounded connected component of $\H \backslash \Gamma(R \circ h_1^{-1} \circ h_2^{-1})$. Continue repeating the descent, at each step choosing a Loewner chain that removes a curve other than the image of $\eta_j$ but that connects two adjacent boundary points. The procedure stops when we have applied $n-1$ Loewner chains and generated a sequence of conformal maps $h_1, \ldots, h_{n-1}$ such that $R_{n-1}$ is a real rational function of degree two with exactly two critical points, where $R_i = R_{i-1} \circ h_i^{-1}$ with $R_0 = R$. Then Lemma \ref{lem:base_geo_chord} implies that $\Gamma(R_{n-1})$ is the hyperbolic geodesic in $\H$ connecting the two critical points of $R_{n-1}$. By construction of the $R_i$ we have
\[
\Gamma(R_{n-1}) = h_{n-1}(\Gamma(R_{n-2})) = \ldots = (h_{n-1} \circ \ldots \circ h_1)(\Gamma(R)).
\]
However the right hand side is exactly $(h_{n-1} \circ \ldots \circ h_1)(\eta_j)$, since at each step we constructed $h_i$ by applying the Loewner chain for a curve other than $\eta_j$. This gives
\[
\eta_j = (h_{n-1} \circ \ldots \circ h_1)^{-1}(\Gamma(R_{n-1})),
\]
and this completes the proof since hyperbolic geodesics are conformally invariant and $\Gamma(R_{n-1})$ has already been shown to be a hyperbolic geodesic in $\H$.
\end{proof}

\subsection{Real Locus as Flow Lines\label{sec:locus_geodesic}}

Thus far we have focused on our new description of Loewner chains for the real locus of $R \in \CRR_{n+1}(\bfs x)$. There is well known and classical theory showing that the curves in the real locus are also flow lines of the vector field $1/R'$. We now briefly recall this theory, and in the concluding remarks of Section \ref{sec: conclusion} we explain how, at least heuristically, it is the classical limit of the \textbf{imaginary geometry} of Miller and Sheffield \cite{MS16:imaginary1, MS16:imaginary2, MS16:imaginary3, MS16:imaginary4}. 

Given a rational function $R \in \CRR_{n+1}(\bfs x)$ associate to it a vector field $v_R$ on $\widehat{\C}$ defined by
\[
v_R(z) = \frac{1}{R'(z)}.
\]
We consider flow lines of the differential equation $\dot{z} = v_R(z)$, and by the factorization of $R'$ given in \eqref{eq:Rprime_product} this equation can be written in the standard coordinate chart of $\C$ as 
\[
\dot{z} = \frac{\prod_{\zeta_k \neq \infty} (z - \zeta_k)^2}{\prod_{i=1}^{2n} (z - x_i)}.
\]
The set $\mathcal{S}_R = \{ \bfs x, \bfs \zeta \}$ are the singularities of this vector field although they are of two distinct types: the poles of $R$ are stationary points of the vector field $v_R$, while $v_R$ blows up at the critical points of $R$. For the sake of simplicity we will assume in this section that the poles and the critical points are fully distinct, i.e. that $\bfs x \cap \bfs \zeta = \emptyset$. The phase portrait of the vector field is largely determined by its \textit{index} at each singular point. The index counts the number of rotations of the vector field as one traverses a small circle around the singular point, i.e. the index $\mathcal{J}(v_R; s_0)$ around critical point $s_0$ is the winding number of the curve $\theta \mapsto v_R(s_0 + r e^{i \theta})$, $0 \leq \theta \leq 2 \pi$. For $R \in \CRR_{n+1}(\bfs x)$ one easily verifies that $\mathcal{J}(v_R; \zeta_k) = 2$ while $\mathcal{J}(v_R; x_i) = -1$. In most modern literature the singular points with index $-1$ are called \textit{saddle points} while those with index $2$ are called \textit{dipoles}. These values agree with the Poincar\'{e}-Hopf index theorem (which states that the total sum of all indices equals the Euler characteristic of the surface) since the vector field $v_R$ lives on the Riemann sphere and
\[
\sum_{s \in \mathcal{S}_R} \mathcal{J}(v_R; s) = 2(n+1) - 2n = 2 - 2 \cdot 0 = \chi(\widehat{\C}).
\]
Flow lines of rational vector fields on the Riemann sphere are well described in several papers, among them \cite{Benzinger, MV95, GGJ:phase_portrait, KR:vector_fields} and the references therein. Since we work in the much simpler special case of rational vector fields that have a rational primitive we can derive a much quicker description of the flow lines. Indeed, for a holomorphic function $f$ it is straightforward to verify that the flow lines of $\dot{z} = 1/f'(z)$ are the level lines of $\Im f$. This fact can be verified by observing that $1/f'(z)$ is orthogonal to the gradient $\overline{\dell} \Im f$, which is in turn orthogonal to the level lines of $\Im f$. Thus $1/f'(z)$ is tangent to the level lines of $\Im f$. Alternatively, one can also verify this fact by noting that it is true for $f(z) = z$ and that vector fields are $(-1,0)$-differentials. Specializing this fact to $f = R$ gives the following: 

\begin{lem}\label{lem:Rprime_flow_lines}
Let $x_1, \ldots, x_{2n}$ be real and distinct and $R \in \CRR_{n+1}(\bfs x)$. The flow lines of $\dot{z} = 1/R'(z)$ are level lines of $\Im R$.
\end{lem}

With this result the phase portrait of the differential equation has an elementary description. Consider the flow lines on each face of the graph generated by $\Gamma(R)$ (recall that the faces are the $2n+2$ open, simply connected regions of the complement of $\Gamma(R)$ in the Riemann sphere). As we mentioned in the proof of Proposition~\ref{prop:merging}, on each face $R$ is a bijection onto either $\H$ or $-\H$. Thus by Lemma \ref{lem:Rprime_flow_lines} the flow lines of $v_R$ within a given face are the pre-images of the horizontal lines in either $\H$ or $-\H$. Since those horizontal lines meet at infinity their pre-images meet at the pole of $R$ that sits on the boundary of the given face. In other words, if $\zeta_j$ is the pole of $R$ that sits the boundary of a face $F$, then every flow line in $F$ moves towards $\zeta_j$ in both forward \textit{and} backward time. See Figures~\ref{figure: circle}, \ref{figure: quadruple1}, \ref{figure: quadruple2}, and \ref{figure: n=3} in Section~\ref{sec: examples} for illustrations of this fact. Another way of stating it is that any face of $\Gamma(R)$ that has $\zeta_j$ on its boundary is a basin of attraction for $\zeta_j$, when it is regarded as a stationary point of the vector field $v_R$. 

Under replacement of $R$ by a post-composition $\phi \circ R$, where $\phi \in \PSL(2, \R)$, the horizontal lines get mapped to circles and the flow lines become the pre-images of these circles by $R$. This is also readily apparent in Figures~\ref{figure: circle}, \ref{figure: quadruple1}, \ref{figure: quadruple2}, and \ref{figure: n=3}. In these figures we also observe that around each critical point the flow lines of $v_R$ that are off of the real locus locally look like hyperbolas that are asymptotic to the coordinate axes.

On the other hand, around each critical point the flow lines on the real locus move toward the critical point in one of the coordinate directions and move away from the critical point in the other direction, i.e. they display the behavior of a saddle point. The real locus of $\Gamma(R)$ is characterized as the flow lines of $v_R$ that move towards the critical points of $R$ in at least one direction. We summarize these results in the next proposition, which is a special case of the more general results in \cite{Benzinger}.

\begin{prop}[\cite{Benzinger}]\label{prop:Rprime_flow_real_locus}
Let $x_1, \ldots, x_{2n}$ be real and distinct and $R \in \CRR_{n+1}(\bfs x)$. For each $z \in \widehat{\C} \backslash \mathcal{S}_R$ the flow line of $1/R'$ passing through $z$ converges in each direction, and the limit is either a critical point or a pole of $R$. Moreover $z \in \Gamma(R)$ iff at least one of the limits of the flow line passing through $z$ is a critical point.
\end{prop}

Another way of summarizing the above results is that the vector field $v_R$ is, on the closure of each face, the pullback of the vector field $1$ by the rational map $R$. This is consistent with vector fields being $(-1,0)$-differentials. The flow lines of the vector field $1$ are also geodesics of the flat metric on the Riemann sphere, and the pullback of these geodesics by $R$ are also geodesics in the pullback metric (this fact is true in great generality). Since intervals on the real line are geodesics in the flat metric on the sphere we immediately obtain the following. 

\begin{prop}\label{prop:locus_as_geodesics}
Any arc within an edge of $\Gamma(R)$ is a geodesic of the pullback metric
\begin{align}\label{eq:Rprime_metric}
ds^2 = |R'(z)|^2 |dz|^2.
\end{align}
\end{prop}

Thus the real locus $\Gamma(R)$ can be characterized as both flow lines of a vector field and as geodesics of a metric. Similar characterizations for random interface curves via random vector fields or random metrics have been of intense interest over the last decade. In the final section we use heuristics from conformal field theory to explain how a certain limit of the Miller-Sheffield imaginary geometry, which uses the Gaussian free field as a vector field to describe random interface curves, naturally leads to the description of $\Gamma(R)$ as the flow lines of $v_R$. We are unsure of how to view the geodesic description of $\Gamma(R)$ as a limit of random geodesics, but we note that in the deterministic setting the geodesic description can be obtained from the flow line description. Indeed, consider vector fields
\[
X_1(z) = \Re \Big( \frac{1}{R'(z)} \partial_z \Big), \quad X_2(z) = -\Im \Big( \frac{1}{R'(z)} \partial_z \Big), 
\]
where $\partial_z = \partial_x - i \partial_y$. The Cauchy-Riemann equations show that the Lie bracket $[X_1, X_2]$ vanishes and so the flows of $X_1$ and $X_2$ locally commute. Classical theory, as explained in \cite{MV95} among other places, implies that there exists a metric tensor $\textsl{g}$ such that $\{X_1,X_2\}$ forms an orthonormal (moving) frame for $\metg$, i.e. $\metg(X_j, X_k) = \delta_{jk}$ for $j,k=1,2$. Under this metric the flow lines of $X_1$ and $X_2$ are geodesics, by construction. In the standard Euclidean chart this metric is expressed as
\[
ds^2 = \metg(\partial_x, \partial_x) \, dx^2 + \metg(\partial_y, \partial_y) \, dy^2 + 2 \metg(\partial_x, \partial_y) dx \, dy = |R'(x+iy)|^2 (dx^2 + dy^2) = |R'(z)|^2 \, |dz|^2,
\]
and therefore agrees with \eqref{eq:Rprime_metric}.
It would be interesting to see if the flow line description of random interface curves could be carried over to a geodesic description in a similar manner, although of course the classical theory would not apply in that case.

\section{Commutation, Conformal Invariance, and the Null Vector Equations \label{sec: commutation}}

We now reconsider the origins of the null vector equations \eqref{eq: NV0} and their somewhat mysterious appearance in this theory. Theorem \ref{thm: Z} constructs a solution set to \eqref{eq: NV0} from the critical points and poles of the canonical rational function in $\CRR_{n+1}(\bfs x; \alpha)$, one such set for each $\alpha$. This is natural if one knows beforehand that solutions to \eqref{eq: NV0} produce driving functions for Loewner chains for the real locus of $R \in \CRR_{n+1}(\bfs x)$, but to date the literature seems only to arrive at the null vector equations \eqref{eq: NV0} by considering $\kappa \to 0$ limits of the BPZ equations for multiple SLE$(\kappa)$. 

In this section we give a purely geometric explanation for the null vector equations that is free of any reference to multiple SLE$(\kappa)$. We show instead that the null vector equations appear naturally when one asks for two desirable properties in a class of Loewner chains: \textbf{geometric commutation} and \textbf{conformal invariance}. The former asks that one can grow the entire ensemble of curves by applying the Loewner chains for each individual curve in any particular order, while the second asks that the Loewner chains generating the curves are invariant under conformal transformations (at least up to a time change). In spirit our derivation of the null vector equations is most similar to the argument of Dub\'{e}dat \cite{Dubedat07} (see also \cite{Graham:multiple_SLE}) who proved that the BPZ equations are a natural consequence of commutation for the multiple SLE$(\kappa)$ curves. We show that together the commutation and conformal invariance properties imply that the dynamics for the Loewner driving function must satisfy the null vector equations \eqref{eq: NV0} and an additional set of algebraic relations known as the \textbf{conformal Ward identities}. This shows that these algebraic relations are a natural consequence of the geometric properties of the curves, and ultimately that multiple SLE$(0)$ curves are natural geometric objects in their own right without relying on any reference to their multiple SLE$(\kappa)$ counterparts, the Brownian loop measure, or even as the real loci of real rational functions. 

We proceed axiomatically. We first specify geometric commutation and conformal invariance as defining properties of an ensemble of curves and then derive the null vector equations and conformal Ward identities for the dynamics of the Loewner driving functions as a consequence of these properties. Loosely speaking, the defining properties we wish to consider are the following.
\begin{enumerate}[(i)]
  \item The $n$ curves in the ensemble are smooth, simple, and disjoint, and each curve connects two distinct boundary point in $\R$. We enumerate these $2n$ boundary points as $\bfs x = (x_1, \ldots, x_{2n})$, and the curves connect them according to one of the $C_n$ possible non-crossing link patterns $\alpha$. 
  \item Each curve can be individually generated by a Loewner chain, and the dynamics of the driving function only depends on the $2n$ marked points $\bfs x$ and the desired link pattern $\alpha$. Moreover each of the $n$ curves can be grown from either one of the boundary points that it connects. To distinguish between the two possibilities we consider the ensemble as $2n$ distinct curves, enumerated as $\gamma_j$, $j=1,\ldots,2n$, where $\gamma_j$ is the curve grown from $x_j$.
  \item The curves \textbf{geometrically commute}, meaning that the same collection of curves can be generated by applying the individual Loewner chains in any chosen order. 
  \item \label{item: Moebius invariance} Each curve $\gamma_j$ is \textbf{M\"{o}bius invariant} in $\H$. This means that if $\gamma_j$ is the curve generated by a Loewner flow and initial data $\bfs x$, then its image $\phi(\gamma_j)$ under a conformal automorphism $\phi$ of $\H$ is, up to a time change, generated by the same flow with initial data $\phi(\bfs x) = (\phi(x_1), \ldots, \phi(x_n))$. 
\end{enumerate}
We now explain these conditions more precisely. Let $\bfs x = \{ x_1, \ldots, x_{2n} \} \subset \R$ be distinct. The $x_i$ will be the collection of $2n$ boundary points for the curves. The first condition implies that the ensemble of $n$ curves is homeomorphic to a multiple SLE$(0; \bfs x; \alpha)$ ensemble (equivalently the real locus of an $R \in \CRR_{n+1}(\bfs x; \alpha)$). We assume without loss of generality that each $\gamma_j$ is parameterized so that $\hcap(\gamma_j[0, t]) = 2t$, although later we will consider other time parameterizations. For assumption (ii) we assume that individually each curve $\gamma_j$ is generated by a Loewner chain of the form
\begin{equation} \label{eq: Loewner1}
\partial_t g_{j,t}(z) = \frac{2}{g_{j,t}(z) - x_j(t)}, \quad g_{j,0}(z) = z, \, x_j(0) = x_j,
\end{equation}
where $t \mapsto x_j(t)$ is the driving function. We assume that the collection of points $\bfs x(t) = (x_1(t), \ldots)$ evolves as an autonomous dynamical system, and we restrict attention to systems of the type
\begin{align}\label{eq: Loewner1_driving}
\dot{x}_j = (\partial_j \mathcal{U})(\bfs x), \quad \dot{x}_k = \frac{2}{x_k - x_j} \textrm{ for } j \neq k.
\end{align}
At this point $\mc{U} : \R^{2n}_d \to \R$ is assumed to be smooth on each connected component of $\R^{2n}_d$, where the latter is the open subset of $\R^{2n}$ in which all coordinates are distinct. Under these dynamics the driving function $x_j(t)$ evolves as specified by $\partial_j \mc{U}$, while the points $x_k(t)$ simply follow the Loewner chain generated by $x_j(t)$. This reflects the fact that $\gamma_j$ is only growing from one point, but the evolution of the points $x_k(t)$ does influence the dynamics of $x_j(t)$ through the function $\mc{U}$. Throughout we use the notation
\begin{align}\label{eq:comm_Uj}
U_j = \partial_j \mc{U}
\end{align}
and observe that smoothness of $\mc{U}$ implies that $\partial_j U_k = \partial_k U_j$ for all $j,k$. We define differential operators corresponding to \eqref{eq: Loewner1_driving} by
\begin{align}\label{eq: Lj_vector_field}
\LieL_j = U_j(\bfs x) \partial_j + \sum_{k \neq j} \frac{2}{x_k - x_j} \partial_k, \quad j=1,\ldots,N.
\end{align}
Note that each $\LieL_j$ and the boundary data $\bfs x = \bfs x(0)$ determines the Loewner chain $t \mapsto g_{j,t}$ and the curve $\gamma_j$. Before stating the theorem relating $U_j$ to the null vector equations we first make a precise definition of the notion of geometric commutation that is expressed in (iii).

\begin{def*}
Let $\gamma_1, \ldots, \gamma_N$ be smooth, simple curves in $\H$ that are generated by Loewner chains of the form \eqref{eq: Loewner1}-\eqref{eq: Loewner1_driving}. The curves are said to \textit{locally} geometric commute if for any pair $1 \leq j \neq k \leq N$, and any $s, t \geq 0$ such that $\gamma_j[0,t]$ and $\gamma_k[0,s]$ do not intersect and individually only intersect the boundary at one point ($x_j$ and $x_k$, respectively), then $g_{j,t}(\gamma_k[0,s])$ is the curve generated by $\LieL_k$ and the initial data $(g_{j,t}(x_1), \ldots, g_{j,t}(x_N))$, at least up to a time change.
\end{def*}

Since the above holds for all $j \neq k$, this condition means that the hull $\gamma_j[0,t] \cup \gamma_k[0,s]$ can be removed by running the two Loewner chains in either order. The only caveat is that the second Loewner chain must be applied with a different time parameterization, since the half-plane capacity of the remaining hull is altered after the application of the first flow. This difference in time parameterizations is the reason for the appearance of the term $(x_k - x_j)^{-2}$ in the following result.

\begin{thm}\label{thm:commutation}
Let $\gamma_1, \ldots, \gamma_{2n}$ be smooth, simple curves in $\H$ that are generated by Loewner chains of the form \eqref{eq: Loewner1}-\eqref{eq: Loewner1_driving}. If the curves locally geometrically commute then $\LieL_j$ satisfy the \textbf{commutation relations}
\begin{align}\label{eq: Lie_CR}
[\LieL_j, \LieL_k] = \frac{4}{(x_k - x_j)^2} ( \LieL_k - \LieL_j ), \quad \textrm{ for all } 1 \leq j \neq k \leq {2n}.
\end{align}
Moreover, under the additional assumption that $\partial_j U_k = \partial_k U_j$ for all $j,k$, the commutation relations hold for $\LieL_j$ of the form \eqref{eq: Lj_vector_field} iff there exists smooth functions $h_1, \ldots, h_{2n} : \R \to \R$ such that
\begin{align}\label{eq: NVE}
\frac{1}{2} U_j^2 + \sum_{k \neq j} \frac{2}{x_k - x_j} U_k - \sum_{k \neq j} \frac{6}{(x_k - x_j)^2} = h_j(x_j), \quad j=1,\ldots,2n.
\end{align}
\end{thm}

The argument for the commutation relation \eqref{eq: Lie_CR} does not rely on the particular form \eqref{eq: Loewner1_driving} of the dynamics for the driving function, or equivalently on formula \eqref{eq: Lj_vector_field} for the differential operators. Instead it is based on properties of the Loewner equation and the half-plane capacity. Our proof is an adaptation of the one found in \cite{Dubedat07} (see also \cite{Graham:multiple_SLE}) that derives the same algebraic commutation relation for multiple SLE$(\kappa)$ curves. In Dub\'{e}dat's case the geometric commutation property of the curves only needs to hold in law and the operators $\LieL_j$ of \eqref{eq: Lj_vector_field} are replaced by the second order differential operators that are the generators of the random driving functions. Equation \eqref{eq: Lie_CR} is essentially a geometric condition, with the particular form of the $\LieL_j$ largely irrelevant. 

Our derivation of \eqref{eq: NVE} differs from Dub\'{e}dat's approach in that it is less computational and more conceptual. Most notably, we derive \eqref{eq: NVE} purely as a consequence of viewing the Loewner evolution of one curve in the coordinate chart induced by the growth of another curve. In contrast, the derivation that Dub\'{e}dat uses for the $\kappa > 0$ analogue of \eqref{eq: NVE} is more algebraic. Essentially, in the $\kappa = 0$ case his approach would be to insert \eqref{eq: Lj_vector_field} into \eqref{eq: Lie_CR} and reduce terms. Our derivation of \eqref{eq: NVE} is somewhat more geometric in spirit, and unifies nicely with our derivation of conformal invariance.

Note that the algebraic description of geometric commutation in \eqref{eq: NVE} only forces that the left hand side of \eqref{eq: NVE} is a function of the single variable $x_j$. From the multiple SLE$(0)$ point of view this is somewhat surprising. For that particular family of curves it is proved by Peltola-Wang \cite[Proposition 1.6]{PW20} that their versions of $U_j$ satisfy \eqref{eq: NV0}; that is they satisfy \eqref{eq: NVE} with $h_j \equiv 0$. This also follows from our Theorems \ref{thm: real_locus} and \ref{thm: Z}. Combining these results with our Theorem \ref{thm:commutation} gives another proof that the multiple SLE$(0)$ curves geometrically commute. However, since \eqref{eq: NVE} is an if and only if statement, Theorem \ref{thm:commutation} also implies that if one can find solutions $U_j$ to \eqref{eq: NVE} for some $h_j \neq 0$, then there are families of geometrically commuting curves that are \textit{distinct} from multiple SLE$(0)$. In this paper we do not make any considerations of this type, but it would be interesting to consider particular examples and their possible relation to other statistical mechanical systems. Instead we focus on the case $h_j \equiv 0$ (\textit{without} knowing beforehand that it corresponds to multiple SLE$(0)$ curves) and show that $h_j \equiv 0$ is forced by the M\"{o}bius invariance requirement on the curves that is laid out in assumption (iv) of p.~\pageref{item: Moebius invariance}.

\begin{thm}\label{thm: hj_is_zero}
If solutions $U_j, j=1,\ldots,2n$, to \eqref{eq: NVE} generate curves that are each invariant under M\"{o}bius transformations then $h_j \equiv 0$ for $j=1,\ldots,2n$.
\end{thm}

Theorem \ref{thm: hj_is_zero} is stated only as invariance under M\"{o}bius transformations, but by a standard construction this is sufficient to obtain full conformal invariance. Indeed, on other simply connected domains the ensemble of curves is defined by conformal mapping, but this is only possible so long as this definition is self-consistent with the group of conformal automorphisms of $\H$ to itself. These conformal automorphisms are precisely the subgroup PSL$(2, \R)$ of fractional linear transformations with real coefficients. Invariance of the curves with respect to PSL$(2, \R)$ is ultimately a statement about the behavior of the $\LieL_j$ under these transformations. Since $\mc{U}$ and $U_j = \partial_j \mc{U}$ are assumed to be defined globally on $\R^{2n}_d$ the $\LieL_j$ generate curves determined by both the boundary data $\bfs x = (x_1, \ldots, x_N)$ and the boundary data $\phi(\bfs x) = (\phi(x_1), \ldots, \phi(x_N))$, where $\phi \in \PSL(2, \R)$. M\"{o}bius invariance is the statement that if $\gamma_j$ is the curve determined by $\LieL_j$ and $\bfs x$ then $\phi(\gamma_j)$ is the curve determined by $\LieL_j$ and $\phi(\bfs x)$, at least up to a time change. To obtain Theorem \ref{thm: hj_is_zero} from this condition we first show it forces that $\mc{U}$ is a pre-pre-Schwarzian form of order $3$ in each variable, and hence each $U_j$ is a pre-Schwarzian form of order $3$ in the $x_j$ variable and a scalar in the remaining ones. These facts imply that the $U_j$ satisfy the system of \textbf{conformal Ward identities} \eqref{eq: CWI0}. The proof of Theorem \ref{thm: hj_is_zero} shows that $h_j \equiv 0$ is necessary for \eqref{eq: NVE} and \eqref{eq: CWI0} to hold simultaneously.

\subsection{Transformation of Loewner Flows under Coordinate Changes} \label{sec:Loewner_flows_charts}
We now prove Theorems \ref{thm:commutation} and \ref{thm: hj_is_zero}: that local commutation and conformal invariance imply that the $U_j$ of \eqref{eq: Loewner1_driving}-\eqref{eq:comm_Uj} satisfy the null vector equations \eqref{eq: NV0}. We first analyze the consequences of local commutation and conformal invariance separately and combine the arguments together at the end of the section. Both local commutation and conformal invariance are instances of essentially the same general property: that the Loewner chain of a curve, when viewed in a different coordinate chart, is a time change of the Loewner chain in the standard coordinate chart but with different initial conditions. For local commutation we study the Loewner chain for $\gamma_j$ in the coordinates induced by the conformal maps $g_{k,t}$ for $k \neq j$, while for conformal invariance we study the Loewner chain for each $\gamma_j$ in the coordinates corresponding to M\"{o}bius transformations of $\H$.

To this end, let us briefly review how Loewner chains tranform under coordinate changes. The following summary is modified from \cite[Section 4.6]{Lawler:book} and \cite[Section 3.2]{Lawler:trieste}. Let $\gamma = \gamma(t)$ be a continuous, non-crossing curve in $\overline{\H}$, and for simplicity we assume that $\gamma(0) = x \in \R$ and $\gamma(0,t] \subset \H$. Assume that $\gamma$ is generated by the Loewner chain
\begin{align}\label{eq:vf_LE}
\partial_t g_t(z; x) = \frac{2}{g_t(z; x) - W_t}, \, \dot{W}_t = v(W_t; g_t(z_1), \ldots, g_t(z_N)) \quad g_0(z) = z, \, W_0 = x.
\end{align}
We only consider the case where $\dot{W}_t = v(W_t; g_t(z_1), \ldots, g_t(z_N))$ for some $v : \R \times \C^N \to \R$. We allow $v$ to depend on the location of a collection of marked points in the flow $g_t$, but to keep the notation brief we only write $\dot{W}_t = v(W_t)$ from now on. Under a conformal transformation $\Phi : \mathcal{N} \to \H$, defined in a neighborhood $\mathcal{N}$ of $x$ such that $\gamma[0,T] \subset \mathcal{N}$ and such that $\Phi$ sends $\partial \mathcal{N} \cap \R$ to $\R$, the Loewner chain of the image curve $\tilde{\gamma}(t) = \Phi \circ \gamma(t)$ is as follows, at least for $0 \leq t \leq T$. Let $\tilde{g}_t$ denote the unique conformal transformation of $\H \backslash \tilde{\gamma}[0, t]$ onto $\H$ that satisfies the normalization $\tilde{g}_t(z) = z + o(1)$ as $z \to \infty$. Letting $h_t = \tilde{g}_t \circ \Phi \circ g_t^{-1}$, it can be computed that $\tilde{g}_t(z)$ evolves as
\[
\partial_t \tilde{g}_t(z; x) = \frac{2 h_t'(W_t)^2}{\tilde{g}_t(z; x) - \tilde{W}_t}, \quad \tilde{g}_0(z) = z, \, \tilde{W}_0 = \Phi(x),
\]
where $\tilde{W}_t = \tilde{g}_t \circ \Phi \circ \gamma(t) = \tilde{g}_t \circ \Phi \circ g_t^{-1}(W_t) = h_t(W_t)$ is the driving function for the new flow. Note that $\tilde{W}_0 = \Phi(W_0) = \Phi(x)$. The equation for $\partial_t \tilde{g}_t(z)$ shows that $\tilde{\gamma}$ is parameterized so that its half-plane capacity satisfies $\operatorname{hcap}(\tilde{\gamma}[0,t]) = 2 \sigma(t)$, where
\begin{align}\label{eq: hcap_reparameterization}
\sigma(t) := \int_0^t h_s'(W_s)^2 \, ds.
\end{align}
Furthermore, by $\tilde{W}_t = h_t(W_t)$ and $\dot{W}_t = v(W_t)$ we compute that when the time evolution of $W_t$ is sufficiently smooth (which is the only case that we consider) then the time evolution of $\tilde{W}_t$ satisfies
\begin{align}
\label{eq: tildeW_dot}
\dot{\tilde{W}}_t = (\partial_t h_t)(W_t) + h_t'(W_t) \dot{W}_t = -3 h_t''(W_t) + h_t'(W_t) v(W_t).
\end{align}
The first part of the final equality comes from \cite[eqn. (4.35)]{Lawler:book}. Thus the driving function for the image curve $\tilde{\gamma}$ follows a vector field that varies with time. Using that $h_0 = \Phi$, we see from \eqref{eq: tildeW_dot} that at $t = 0$ and at the point $\Phi(x)$ (since $\tilde{W}_0 = \Phi(x)$) this vector field is
\begin{align}\label{eq:vf_at_zero}
-3 \Phi''(x) + \Phi'(x) v(x). 
\end{align}
On the other hand, consider the Loewner chain $g_{\sigma(t)}(\cdot; \Phi(x)), t \geq 0$. These are the normalizing maps for the curve started from $\Phi(x)$, whose half-plane capacity is parameterized by \eqref{eq: hcap_reparameterization}, but whose driving function evolves according to $v$. This curve and the image curve are the same iff their Loewner chains are the same, i.e. if $\tilde{g}_t(\cdot; x) = g_{\sigma(t)}(\cdot; \Phi(x))$. Equality of these two chains requires that the driving functions both start at the same position (which is already true since they both start at $\Phi(x)$) and have the same first derivative at $t = 0$. For the image curve the first derivative is given above by \eqref{eq:vf_at_zero}, while for the curve corresponding to $g_{\sigma(t)}(\cdot; \Phi(x))$ the first derivative is $v(\Phi(x)) \sigma'(0) = v(\Phi(x)) \Phi'(x)^2$. Again this used \eqref{eq: hcap_reparameterization} and that $h_0 = \Phi$. Thus a necessary condition for the two curves to be equal is
\begin{align}\label{eq: v_transform}
v(\Phi(x)) \Phi'(x)^2 = -3 \Phi''(x) + \Phi'(x) v(x).
\end{align}
An immediate application of \eqref{eq: v_transform} is the following result on conformal invariance of curves generated by deterministic Loewner flow of the type \eqref{eq:vf_LE}. 

\begin{prop}\label{prop: Mobius_inv}
The curve generated by the Loewner flow \eqref{eq:vf_LE} is invariant under M\"{o}bius transformations iff $v$ is a pre-Schwarzian form of order $3$ in the driving function variable and a scalar field in all other variables.
\end{prop}

The proposition is stated for $v : \R \times \C^N \to \R$ of the form $v(x; z_1, \ldots, z_N)$, where the $z_i$ are other marked points on which $v$ may depend. For a given $\phi \in \PSL(2, \R)$, $v$ being a pre-Schwarzian form of order $3$ and a scalar in all the other variables means that $v$ transforms as
\begin{align}\label{eq: U_is_PS3}
v(\bfs x) = v(\phi(\bfs x)) \phi'(x) + 3 (\log \phi')'(x),
\end{align}
where, as usual, $\phi(\bfs x) = (\phi(x); \phi(z_1), \ldots, \phi(z_N))$. See any book on differential geometry for this definition, or \cite[Lecture 4]{KM13}. 

\begin{proof}
For conformal invariance we use \eqref{eq: v_transform} in which the conformal map $\Phi$ is replaced by a generic $\phi \in \PSL(2, \R)$. By equation \eqref{eq: v_transform} a necessary condition for the generated curve to be conformally invariant is that
\[
v(\phi(\bfs x)) \phi'(x)^2 = -3\phi''(x) + \phi'(x) v(\bfs x).
\]
Rearranging this gives \eqref{eq: U_is_PS3}.
\end{proof}

Finally, we end this section by using the argument of Proposition \ref{prop: Mobius_inv} to complete the proof of Proposition \ref{prop: SLE_zero_rho_curve}, specifically the conformal invariance of the SLE$(0, \bfs \rho)$ curve iff $\int \! \bfs \rho = -6$. A positive $\kappa$ version of this result can be found in several places, for example \cite[Theorem 3]{SW05}.

\begin{proof}[Proof of Proposition \ref{prop: SLE_zero_rho_curve} - conformal invariance]
By definition, specifically \eqref{eq: g}~--~\eqref{eq: zero_rho_driving}, SLE$(0, \bfs \rho)$ Loewner chains are of the form \eqref{eq:vf_LE}, with $v$ given by
\[
v(\bfs x) = v(x; z_1, \ldots, z_N) = \sum_{k=1}^N \frac{\rho_k}{x - z_k}, \quad \rho_k \in \R.
\]
By either \eqref{eq: v_transform} or the argument in the proof of Proposition \ref{prop: Mobius_inv}, the SLE$(0, x, \bfs \rho)$ Loewner flow is invariant under M\"{o}bius transformations of $\H$ to itself iff
\[
v(\phi(\bfs x)) \phi'(x)^2 = -3 \phi''(x) + \phi'(x) v(\bfs x)
\]
holds for all $\phi \in \PSL(2, \R)$ and for all choices of $\bfs x$. Since $\PSL(2,\R)$ is generated by translations, dilations, and inversions, the invariance holds iff it holds in each of these three cases. It is straightforward to check that translations and dilations impose no conditions on the values of $\rho_k$, since it is already true that
\[
v(x+h, z_1+h, \ldots, z_N + h) = v(x, z_1, \ldots, z_N), \quad v(r \bfs x)r = v(\bfs x)
\]
holds for all $h \in \R$, all $r > 0$, and all appropriate $\bfs x$, regardless of the value of $\rho_k$, by the definition of $v$. For the inversion map $\phi(x) = -1/x$, however, it is required that
\[
v(\phi(\bfs x)) x^{-4} = 6 x^{-3} + x^{-2} v(\bfs x).
\]
Straightforward algebra and the definition of $v$ allows us to rewrite this equation as
\[
x^{-3} \left( \sum_{k \neq j} \frac{\rho_k (z_k - x)}{x - z_k} - 6 \right) = 0.
\]
Clearly this holds for all choices of $x$ and $z_i$ iff $\int \! \bfs \rho = -6$.
\end{proof}

\subsection{Derivation of the Null Vector Equations\label{subsec:NVE_deriv}}

We now use the theory of Section \ref{sec:Loewner_flows_charts} to prove Theorems \ref{thm:commutation} and \ref{thm: hj_is_zero}. Throughout these proofs and the rest of the section we use the notation
\begin{equation} \label{eq: Moser's notation}
x_{jk} = x_{jk}(\bfs x) = \begin{cases}
(x_j - x_k)^{-1}, & k \neq j \\
0, & k = j
\end{cases}
\end{equation}
and that these quantities satisfy the identities 
\[
x_{jk} = - x_{kj}, \quad x_{jk}^{-1} + x_{kl}^{-1} = x_{jl}^{-1}, \quad x_{lj} + x_{kl} = x_{jk}^{-1} x_{lj} x_{lk} \textrm{ for } j \neq k.
\]

\begin{proof}[Proof of Theorem \ref{thm:commutation}]
The definition of commutation implies that for sufficiently small $s, t > 0$ the normalizing map for the hull $\gamma_j[0,t] \cup \gamma_k[0,s]$ is the composition of the Loewner maps for each individual hull $\gamma_j[0,t]$ or $\gamma_k[0,s]$, when applied in either order. The first such Loewner map to be applied distorts the half-plane capacity of the remaining hull, which is quantified in \eqref{eq: hcap_reparameterization}. For example, in removing $\gamma_j[0,t]$ we are considering the coordinate change $\Phi = g_{j,t}$, which leads to
\begin{align}\label{eq: sigma_jk}
\sigma_{j,k}^{t,s} = \hcap(\tilde{\gamma}_k[0,s]) = \int_0^s (h_{j,k}^{t,u})'(x_k(u))^2 \, du.
\end{align}
In this case $h_{j,k}^{t,s} = \tilde{g}_{k,s} \circ g_{j,t} \circ g_{k,s}^{-1}$. With this notation in hand, commutation implies that
\[
g_{k, \sigma_{j,k}^{t,s}} \circ g_{j,t} = g_{j, \sigma_{k,j}^{s,t}} \circ g_{k, s}.
\]
On the left hand side the driving function first evolves according to the dynamics of $\LieL_j$ for $t$ units of time and then $\LieL_k$ for $\sigma_{j,k}^{t,s}$ units of time. The right hand side is analogous. These Loewner maps can be the same only if the driving function move to the same position when the maps are applied in either order. In our setup the motion of the driving functions is fully determined by the motion of the points in $\bfs x$, so a necessary condition for these maps to be the same is that
\begin{align}\label{eqn:semigroup_commutation}
e^{\sigma_{j,k}^{t,s} \LieL_k} e^{t \LieL_j} = e^{\sigma_{k,j}^{s, t} \LieL_j} e^{s \LieL_k},
\end{align}
where $t \mapsto e^{t \LieL}$ denotes the flow map corresponding to the dynamics $\LieL$. The commutation relation \eqref{eq: Lie_CR}, as we now explain, is a straightforward consequence of this identity. From \eqref{eq: sigma_jk} we obtain
\begin{align*}
\sigma_{j,k}^{t,s} = s ((h_{j,k}^{t,0})'(x_k) + O(s)) = s (g_{j,t}'(x_j) + O(s)) &= s(1 - 2t x_{jk}^2 + o(t) + O(s)) \\
&= s - 2st x_{jk}^2 + o(st) + O(s^2),
\end{align*}
where the constants in the error terms may depend on $x_j$ and $x_k$. Now use the above to expand \eqref{eqn:semigroup_commutation} in powers of $s$ and $t$, and keeping only those coefficients of $st$ (the first place at which the expansion differs) we obtain the infinitesimal commutation relation \eqref{eq: Lie_CR}. 

Now we prove \eqref{eq: NVE}. One way to proceed is to insert the formula \eqref{eq: Lj_vector_field} for $\LieL_j$ into the commutation relation \eqref{eq: Lie_CR} and reduce terms. Instead we derive \eqref{eq: NVE} directly as an application of \eqref{eq: v_transform}. By \eqref{eq: v_transform} and the preceding discussion, a necessary condition for the curve $t \mapsto g_{k,s}(\gamma_j(t))$ to be the same as the curve determined by $\LieL_j$ and initial data $e^{s \LieL_k} \bfs x$ is that
\begin{align}\label{eq: v_transform_at_s}
g_{k,s}'((e^{s \LieL_k} \bfs x)_j)^2 U_j(e^{s \LieL_k} \bfs x) = -3 g_{k,s}''((e^{s \LieL_k} \bfs x)_j) + g_{k,s}'((e^{s \LieL_k} \bfs x)_j) U_j(\bfs x).
\end{align}
Commutation forces that this relation holds for sufficiently small $s > 0$, and we will now compute that equation \eqref{eq: NVE} is the derivative of this equation at $s = 0$. The computation is by Taylor expansion. On the one hand the flow map $e^{s \LieL_k}$ perturbs the points $\bfs x$ and the functions $U_j$ by
\[
(e^{s \LieL_k} \bfs x)_j = x_j + 2s x_{jk} + O(s^2), \quad U_j(e^{s \LieL_k} \bfs x) = U_j(\bfs x) + s U_k \partial_k U_j(\bfs x) + s \sum_l 2 x_{lk} \partial_l U_j(\bfs x) + O(s^2). 
\]
On the other hand the small $s$ behavior of the Loewner map $g_{k,s}$ is
\[
g_{k,s}(x_j) = x_j + 2s x_{jk} + O(s^2), \,\, g_{k,s}'((e^{s \LieL_k} \bfs x)_j) = 1 - 2s x_{jk}^2 + O(s^2), \,\, g_{k,s}''((e^{s \LieL_k} \bfs x)_j) = 4s x_{jk}^3 + O(s^2).
\]
Inserting the last two equations into \eqref{eq: v_transform_at_s} and collecting terms, there is a term $U_j(\bfs x)$ on each side that cancels, and then equality of the coefficients of $s$ leads to
\begin{align}\label{eq: CRjk_zero}
U_k \partial_k U_j - 2x_{jk}^2 U_j + \sum_l x_{lk} \partial_l U_j + 12 x_{jk}^3 = 0.
\end{align}
Use that $\partial_j U_k = \partial_k U_j$ (which follows from the assumption $U_j = \partial_j \mc{U}$) on the first term to replace it with $U_k \partial_j U_k$. Also define
\begin{align}\label{eq: NVj}
\NV_k = \frac{1}{2} U_k^2 + \sum_{l} 2 x_{lk} U_l - \sum_{l} 6 x_{lk}^2
\end{align}
and then notice that \eqref{eq: CRjk_zero} is equivalent to $\partial_j \NV_k \equiv 0$. This holds for each $j \neq k$, which is in turn equivalent to each function $\NV_j$ only depending on $x_j$.
\end{proof}

Now we turn to the matter of conformal invariance. Since the Loewner chains \eqref{eq: Loewner1}-\eqref{eq: Loewner1_driving} are of the type \eqref{eq:vf_LE}, Proposition \ref{prop: Mobius_inv} implies that the each curve in the ensemble is conformally invariant iff each $U_j$ is a pre-Schwarzian form of order $3$ at $x_j$ and a scalar in all other variables, i.e. that
\begin{align}\label{eq: U_is_PS3_2}
U_j(\bfs x) = U_j(\phi(\bfs x)) \phi'(x_j) + 3 (\log \phi')'(x_j).
\end{align}
Since we assume $U_j = \partial_j \mc{U}$ this implies that $\mc{U}$ is a pre-pre-Schwarzian (PPS) form of order $3$ in each variable. The latter means that $\mc{U}$ transforms as
\begin{align}\label{eq: U_is_PPS3}
\mc{U}(\bfs x) = \mc{U}(x_1, \ldots, \phi(x_j), \ldots, x_{2n}) + 3 \log \phi'(x_j)
\end{align}
in each variable $x_j$. In particular, applying the same relation in each coordinate gives
\begin{align}\label{eq: U_is_PPS_2}
\mc{U}(\bfs x) = \mc{U}(\phi(\bfs x)) + \sum_{k=1}^{2n} 3 \log \phi'(x_k),
\end{align}
and then differentiating both sides with respect to $x_j$ recovers \eqref{eq: U_is_PS3_2}. Equation \eqref{eq: U_is_PPS_2} also shows that $U_j$ satisfy the conformal Ward identities \eqref{eq: CWI0}. The calculation, which is standard, involves applying \eqref{eq: U_is_PPS_2} with respect to the infinitesimal generators $1, z$, and $z^2$ (regarded as vector fields) for translations, dilations, and inversions, respectively. We record this result as the following corollary.

\begin{cor}\label{prop: U_is_PPS}
Assume that $U_j = \partial_j \mc{U}$ and that all $2n$ curves generated by Loewner chains \eqref{eq: Loewner1}-\eqref{eq: Loewner1_driving} are M\"{o}bius invariant. Then $U_j$ solve the conformal Ward identities \eqref{eq: CWI0} and $\mc{U}$ is a pre-pre-Schwarzian form of order $3$ in each variable.
\end{cor}

Hence we have established that our $U_j$ of \eqref{eq: Lj_vector_field} must satisfy the conformal Ward identities \eqref{eq: CWI0} in order to satisfy the axiom of conformal invariance. This puts us in a position to prove Theorem \ref{thm: hj_is_zero}: that the only solutions $U_j$ to both \eqref{eq: NVE} and \eqref{eq: CWI0} are those with $h_j \equiv 0$.

\begin{proof}[Proof of Theorem \ref{thm: hj_is_zero}]
Consider the $1/4$-superposition of the $\LieL_j$ operators, written as a vector field $V : \R^{2n} \to \R^{2n}$ with components
\[
V_j(\bfs x) = \frac{1}{4} U_j(\bfs x) + \sum_k \frac{1}{2} x_{jk}(\bfs x).
\]
Since $\partial_j U_k = \partial_k U_j$ it follows that $\partial_j V_k = \partial_k V_j$, and since $U_j$ are assumed to satisfy the null vector equation \eqref{eq: NVE} a straightforward but tedious computation shows that $\NV_j(\bfs x)$ of \eqref{eq: NVj} becomes
\[
\NV_j(\bfs x) = 8 V_j^2 - 8 \sum_k x_{jk} (V_j + V_k) - 8 \sum_k x_{jk}^2 + 4 \sum_{k,l} (x_{jk} x_{jl} + x_{jl} x_{kl}).
\]
Now assume $\NV_j(\bfs x) = h_j(x_j)$ and sum the above over $j$. The second and fourth terms cancel by symmetry, leading to
\begin{align}\label{eq: Vnorm}
|V(\bfs x)|^2 = \sum_j V_j(\bfs x)^2 = \sum_{j} \sum_k x_{jk}(\bfs x)^2 + \frac{1}{8} \sum_j h_j(x_j).
\end{align}
Furthermore, it is straightforward to verify that the conformal Ward identities \eqref{eq: CWI0} for the $U_j$ become the three equations
\[
\sum_{j} V_j(\bfs x) = 0, \quad \sum_{j} x_j V_j(\bfs x) = n(n-2), \quad \sum_j x_j^2 V_j(\bfs x) = (n-2) \sum_j x_j.
\]
It is sufficient to show that these three equations and \eqref{eq: Vnorm} holding simultaneously forces that $h_j \equiv 0$ for each $j$. To do so we put the dynamics $\dot{\bfs x}(t) = V(\bfs x(t))$ on the space and then consider the evolution of the above conformal Ward identities with respect to time. Note that these dynamics are equivalent to $\dot{x}_j = V_j(\bfs x)$, and by differentiating again and using \eqref{eq: Vnorm} we obtain
\begin{align}\label{eq: Vnorm_ddot}
\ddot{x}_j = \sum_k \partial_k V_j(\bfs x) \dot{x}_k = \sum_k (\partial_j V_k(\bfs x)) V_k(\bfs x) = \frac{1}{2} \partial_j \sum_k V_k(\bfs x)^2 = -2 \sum_k x_{jk}^3 + \frac{1}{16} h_j'(x_j).
\end{align}
The second equality used that $\partial_j V_k = \partial_k V_j$. In Section \ref{sec: CM} we will see that these equations also come from a particular Hamiltonian system. Now from the second conformal Ward identity we have
\begin{align}\label{eq: CWI2_is_zero}
0 = \frac{d}{dt} \sum_j x_j V_j(\bfs x) = \frac{d}{dt} \sum_j x_j \dot{x}_j = \sum_j \dot{x}_j^2 + \sum_j x_j \ddot{x}_j &= \frac{1}{8} \sum_j (\tfrac{1}{2} x_j h_j'(x_j) + h_j(x_j)).
\end{align}
The last equality comes about from inserting \eqref{eq: Vnorm} and \eqref{eq: Vnorm_ddot}, and noting that
\[
0 = \sum_j \sum_k x_{jk}(\bfs x)^2 - 2  \sum_j \sum_k x_j x_{jk}(\bfs x)^3,
\]
which follows from finding a common denominator and then using symmetry. Now in \eqref{eq: CWI2_is_zero} the terms in the final summation are functionally independent (i.e. the $j$th term is a function of $x_j$ alone) so the sum is zero iff each term is a constant, leading to
\[
\frac{1}{2} x_j h_j'(x_j) + h_j(x_j) = C_j \quad j=1,\ldots,2n.
\]
The general solution to this equation is $h_j(x) = C_j + D_j x^{-2}$ for $D_j$ arbitrary. This is the consequence of the second conformal Ward identity. From the first and third conformal Ward identities we obtain
\[
0 = \sum_j \frac{d}{dt} \left(x_j^2 V_j(\bfs x) - (n-2) x_j\right) + (n-2) V_j(\bfs x) = \sum_j \frac{d}{dt} x_j^2 \dot{x}_j = \sum_j 2 x_j \dot{x}_j^2 + x_j^2 \ddot{x}_j.
\]
From \eqref{eq: Vnorm_ddot} the last part of the summation becomes
\[
\sum_j x_j^2 \ddot{x}_j = -2 \sum_{j,k} x_j^2 x_{jk}^3 + \frac{1}{16} \sum_j x_j^2 h_j'(x_j).
\]
For the first part of the summation, use that $\dot{x}_j = V_j$ and $\NV_j(\bfs x) = h_j(x_j)$ to rewrite it as
\[
\sum_j 2 x_j \dot{x}_j^2 = 2 \sum_{j,k} x_j x_{jk} (\dot{x}_j + \dot{x}_k) + 2 \sum_{j,k} x_j x_{jk}^2 - \sum_{j,k,l} x_j x_{jk} x_{jl} + \frac{1}{4} \sum_j x_j h_j(x_j).
\]
The first term on the right hand side is identically zero because
\[
\sum_{j,k} x_j x_{jk} (\dot{x}_j + \dot{x}_k) = \sum_{j < k} (x_j - x_k) x_{jk} (\dot{x}_j + \dot{x}_k) = \sum_{j < k} \dot{x}_j + \dot{x}_k = (2n-1) \sum_j \dot{x}_j = 0, 
\]
the last equality being the first conformal Ward identity. From the last four equations we obtain
\[
0 = 2 \sum_{j,k} (x_j x_{jk}^2 - x_j^2 x_{jk}^3) - \sum_{j,k,l} x_j x_{jk} x_{jl} + \frac{1}{4} \sum_j \left( x_j h_j(x_j) + \frac{1}{4} x_j^2 h_j'(x_j) \right).
\]
However it is easy to see that, due to some basic algebra and symmetry considerations, that
\[
\sum_{j,k} x_j x_{jk}^2 - x_j^2 x_{jk}^3 = - \sum_{j,k} x_j x_k x_{jk}^3 = 0, \quad \sum_{j,k,l} x_j x_{jk} x_{jl} = 0.
\]
This leads to the conclusion that 
\[
0 = \sum_j x_j h_j(x_j) + \frac{1}{4} x_j^2 h_j'(x_j),
\]
and again the terms are functionally independent and so each one must be equal to a constant. However the general solution to the equation
\[
x_j h_j(x_j) + \frac{x_j^2}{4} h_j'(x_j) = c_j
\]
is $h_j(x) = \frac{4}{3} c_j x^{-1} + d_j x^{-4}$ for $d_j$ arbitrary. In conclusion, if the $U_j$ satisfy both the generalized null vector equations \eqref{eq: NVE} and the conformal Ward identities \eqref{eq: CWI0}, then $h_j$ are given by both
\[
h_j(x) = C_j + \frac{D_j}{x^2}, \quad h_j(x) = \frac{4}{3} \frac{c_j}{x} + \frac{d_j}{x^4},
\]
for some constants $c_j, C_j, d_j, D_j$. The only common solution is $h_j \equiv 0$.
\end{proof}

\section{Calogero-Moser Dynamics \label{sec: CM}}

In this section we prove our results related to the Calogero-Moser system. We first show that the first order dynamics of the $1/4$-superposition of the SLE$(0, x_j, \bfs \rho_j)$ processes coming from Theorem \ref{thm: real_locus} can be converted into two autonomous second order Calogero-Moser systems, one for the poles and one for the critical points. This uses explicit computation and the stationary relation. Our arguments are simple but different from those typically encountered in the literature, owing to the non-traditional direction from which we approach the Calogero-Moser system. We then discuss connections with some of the known results about the Calogero-Moser system, in particular its well-known integrability properties via its explicit Lax pair.

\subsection{Calogero-Moser via Poles and Critical Points} \label{subsec:poles_to_CM}

Our first result expresses the first order dynamics of the poles and critical poles under the $1/4$-superposition in a manner that is helpful for later computations.

\begin{prop}\label{prop:CMfromSLE}
Let $\bfs x = \{ x_1, \ldots, x_{2n} \}$ be distinct boundary points and $\bfs \zeta = \{ \zeta_1, \ldots, \zeta_{n+1} \} \subset \widehat{\C}$ be a solution to the stationary relation \eqref{eq: stationary}. Let $\bfs x(t)$ and $\bfs \zeta(t)$ be the evolution of the $x_j$ and $\zeta_k$ under the $1/4$-superposition of the associated SLE$(0, x_j, \bfs \rho_j)$ processes. Then the pair $(\bfs x(t), \bfs \zeta(t))$ forms the closed dynamical system satisfying
\begin{equation} \label{eq: Calogero-Moser1}
\dot x_j = \phantom{-}\sum_{k\ne j} \frac1{x_j-x_k} - \sum_{k=1}^{n+1} \frac1{ x_j-\zeta_k},
\end{equation}
and 
\begin{equation} \label{eq: Calogero-Moser2}
\dot\zeta_k = -\sum_{l\ne k}\frac1{\zeta_k-\zeta_l}+\sum_{j=1}^{2n}\frac1{\zeta_k-x_j}. 
\end{equation}
\end{prop}

\begin{proof}[Proof of Proposition~\ref{prop:CMfromSLE}] 
By \eqref{eq: x2} and \eqref{eq: rhoj} the evolution of $x_j(t)$ under the $1/4$-superposition is 
\begin{equation}\label{eq: x dot}
\dot x_j = \frac14 \left( \sum_{k \neq j} \frac{2}{x_j-x_k} - \sum_{k=1}^{n+1} \frac{4}{x_j - \zeta_k} + \sum_{k\ne j} \frac2{x_j-x_k} \right) = \sum_{k\ne j} \frac1{x_j-x_k} - \sum_{k=1}^{n+1} \frac1{x_j - \zeta_k},
\end{equation}
which shows \eqref{eq: Calogero-Moser1}. On the other hand, since the poles follow the Loewner flow we have $\zeta_k(t) := g_t(\zeta_k(0))$, and therefore
\[
\dot\zeta_k =\dot g_t(\zeta_k(0))= \frac12\sum_{j=1}^{2n}\frac{1}{g_t(\zeta_k(0))-x_j}=\frac12\sum_{j=1}^{2n}\frac{1}{\zeta_k-x_j}.
\]
The stationary relation~\eqref{eq: stationary} implies that 
\begin{equation}\label{eq: zeta dot}
\dot\zeta_k =\sum_{l\ne k}\frac1{\zeta_k-\zeta_l}= -\sum_{l\ne k}\frac1{\zeta_k-\zeta_l}+\sum_{j=1}^{2n} \frac1{\zeta_k-x_j},
\end{equation}
which shows \eqref{eq: Calogero-Moser2}.
\end{proof}

The proof relied on the fact that the stationary relation \eqref{eq: stationary} holds at all times under the evolution of the Loewner chain (before collisions). Preservation of the stationary relation was shown in the proof of Theorem \ref{thm: real_locus}, but it of course relies on the assumption that the stationary relation holds at time zero. 

From Proposition \ref{prop:CMfromSLE} we now prove Theorem \ref{thm: CM-main} that the critical points follow the Calogero-Moser dynamics under the $1/4$-superposition. 

\begin{proof}[Proof of Theorem \ref{thm: CM-main}]
By differentiating \eqref{eq: x dot}, we have 
\[
\ddot x_j = -\sum_{k\ne j} \frac{\dot x_j-\dot x_k}{(x_j-x_k)^2} + \sum_l\frac{\dot x_j-\dot\zeta_l}{(x_j-\zeta_l)^2}.
\]
Using the formula \eqref{eq: x dot} for $\dot x_j, \dot x_k$ and the second equality of \eqref{eq: zeta dot} for $\dot \zeta_l$ we obtain
\begin{align*}
\ddot x_j = &-\sum_{k\ne j}\frac1{(x_j-x_k)^2}\left(\frac2{x_j-x_k} + \sum_{l\ne j,k}\Big(\frac1{x_j-x_l} - \frac1{x_k-x_l}\Big)+ \sum_l \Big(\frac1{\zeta_l-x_j}-\frac1{\zeta_l-x_k}\Big) \right)\\
&+\sum_l \frac1{(x_j-\zeta_l)^2}\left(\sum_{k\ne j} \frac1{x_j-x_k} + \sum_{m=1}^{n} \frac1{\zeta_m-x_j} +\sum_{m\ne l}\frac1{\zeta_l-\zeta_m}-\sum_{k=1}^{2n} \frac1{\zeta_l-x_k} \right).
\end{align*}
Rearranging terms gives 
\begin{align*}
\ddot x_j &+\sum_{k\ne j}\frac2{(x_j-x_k)^3} - \sum_{k\ne j}\sum_{l\ne j,k}\frac1{(x_j-x_k)(x_j-x_l)(x_k-x_l)}\\
&=-\sum_{k\ne j}\sum_{l}\frac1{(x_j-x_k)(x_j-\zeta_l)(x_k-\zeta_l)}+\sum_l \frac1{(x_j-\zeta_l)^2}\sum_{k\ne j} \frac1{x_j-x_k}\\
&+\sum_l \frac1{(x_j-\zeta_l)^2}\left(\sum_{m} \frac1{\zeta_m-x_j} - \sum_{m\ne l}\frac1{\zeta_l-\zeta_m} \right).
\end{align*}
The last term on the right hand side used the stationary relation \eqref{eq: stationary}. Now simplify the right hand side by grouping powers of $(x_j - \zeta_l)^{-2}$, and then use the stationary relation again to obtain
\begin{align*}
\sum_l &\frac1{(x_j-\zeta_l)^2}\left(-\sum_{k\ne j}\frac1{x_k-\zeta_l}+\sum_{m} \frac1{\zeta_m-x_j} - \sum_{m\ne l}\frac1{\zeta_l-\zeta_m} \right)\\
&=\sum_l \frac1{(x_j-\zeta_l)^2}\sum_{m\ne l}\Big(\frac1{\zeta_l-\zeta_m}+ \frac1{\zeta_m-x_j}\Big)= \sum_{l}\sum_{m\ne l}\frac1{(x_j-\zeta_l)(x_j-\zeta_m)(\zeta_l-\zeta_m)}.
\end{align*}
Combining all of the above, we obtain 
\begin{align*}
\ddot x_j +\sum_{k\ne j}\frac2{(x_j-x_k)^3} 
&= \sum_{k\ne j}\sum_{l\ne j,k}\frac1{(x_j-x_k)(x_j-x_l)(x_k-x_l)}\\
&+ \sum_{l}\sum_{m\ne l}\frac1{(x_j-\zeta_l)(x_j-\zeta_m)(\zeta_l-\zeta_m)}.
\end{align*}
The right-hand side cancels by symmetry. 
\end{proof}

Using Proposition \ref{prop:CMfromSLE} we can also show that the second order dynamics of the poles are an autonomous system under the $1/4$-superposition, and that they follow the same Calogero-Moser dynamics as the critical points. 

\begin{cor}\label{cor: CM}
Let $\bfs x = \{ x_1, \ldots, x_{2n} \}$ be real and distinct and $\bfs \zeta = \{ \zeta_1, \ldots, \zeta_{n+1} \} \subset \widehat{\C}$ be symmetric under conjugation and solve the stationary relation \eqref{eq: stationary}. Then under the $1/4$-superposition of the SLE$(0, x_j, \bfs \rho_j)$ processes the poles $\zeta_k$ follow the second order dynamics.
\[
\ddot \zeta_k = - \sum_{l \neq k} \frac{2}{(\zeta_k - \zeta_l)^3}.
\]
\end{cor}

\begin{proof}[Proof of Corollary \ref{cor: CM}]
Differentiating the first equality of \eqref{eq: zeta dot}, we have 
\[
\ddot \zeta_k = -\sum_{l\ne k} \frac{\dot \zeta_k-\dot \zeta_l}{(\zeta_k-\zeta_l)^2}.
\]
Now by using the first equality of \eqref{eq: zeta dot} again for $\dot \zeta_k, \dot \zeta_l$ we obtain
\[
\ddot \zeta_k = -\sum_{l\ne k}\frac1{(\zeta_k-\zeta_l)^2}\Big(\frac1{\zeta_k-\zeta_l} + \sum_{m\ne k,l}\frac1{\zeta_k-\zeta_m} - \frac1{\zeta_l-\zeta_k} - \sum_{m\ne k,l}\frac1{\zeta_l-\zeta_m} \Big)
\]
Rearranging terms gives
\[
\ddot \zeta_k = -\sum_{l\ne k}\frac2{(\zeta_k-\zeta_l)^3}+ \sum_{l\ne k}\sum_{m\ne k,l}\frac1{(\zeta_k-\zeta_l)(\zeta_k-\zeta_m)(\zeta_l-\zeta_m)}.
\]
The last term cancels by symmetry. 
\end{proof}

Now we use Theorem \ref{thm: CM-main} and Corollary \ref{cor: CM} to prove Theorem \ref{parabolic PDE} on the PDE $R_t = R \circ g_t^{-1}$. As a preliminary to Theorem \ref{parabolic PDE} we first prove that the polynomial $Q_t$ evolves according to the backward heat equation. Recall that
\[
Q_t(z) = \prod_k (z-\zeta_k(t)).
\]

\begin{proof}[Proof of Theorem \ref{parabolic PDE}]
Let $L_t(z) = \sum_k \log(z-\zeta_k(t))$. On the one hand, differentiating with respect to $t$ gives
\[
\dot L = -\sum_k \frac{\dot\zeta_k}{z-\zeta_k} = -\sum_k \sum_{l\ne k} \frac1{(z-\zeta_k)(\zeta_k-\zeta_l)}.
\]
On the other hand, differentiating with respect to $z$ gives
\begin{align*}
L''+ L'^2 &= \Big(\sum_k\frac1{z-\zeta_k}\Big)^2 - \sum_k\frac1{(z-\zeta_k)^2} = \sum_k \sum_{l\ne k} \frac1{z-\zeta_k}\frac1{z-\zeta_l}\\
&= \sum_k \sum_{l\ne k} \Big(\frac1{z-\zeta_k}-\frac1{z-\zeta_l}\Big)\frac1{\zeta_k-\zeta_l} \\
&= 2 \sum_k \sum_{l \ne k} \frac{1}{(z - \zeta_k)(\zeta_k - \zeta_l)} = -2 \dot{L}.
\end{align*}
Thus $L_t(z)$ satisfies
\[
\dot L = -\frac12 L''- \frac12 L'^2.
\]
Applying this equation to the exponential of $L_t(z)$ shows that $Q_t(z)$ solves the backward heat equation. Now for the evolution of $R_t$, recall that $R_t(z)=P_t(z)/Q_t(z)$ is a canonical rational function with critical points $x_j(t)$ and poles $\zeta_k(t)$, and therefore
\begin{equation}\label{eq: R_t'}
R_t'(z) = \frac{\prod_{j=1}^{2n}(z-x_j(t))}{\prod_{\zeta_k \neq \infty} (z-\zeta_k(t))^2}.
\end{equation}
Since $R_t \circ g_t = R \circ g_t^{-1} \circ g_t = R$ we have 
\[
\frac{d}{dt} R_t(g_t(z)) \equiv 0.
\]
Let $f_t = g_t^{-1}$. Applying the chain rule the last equation can be rewritten as
\[\frac{\dot R_t(z)}{R_t'(z)} = \frac{\dot f_t(z)}{f_t'(z)} = -\frac12\sum_{j=1}^{2n} \frac1{z-x_j(t)}.
\]
On the other hand by differentiating \eqref{eq: R_t'} with respect to $z$ we obtain
\begin{align*}
-\frac12\sum_{j=1}^{2n} \frac1{z-x_j(t)} = -\frac12\Big(\frac{R_t''(z)}{R_t'(z)} + 2\sum_{\zeta_k \neq \infty} \frac1{z-\zeta_k(t)} \Big) = -\frac12 \frac{R_t''(z)}{R_t'(z)} - \frac{Q_t'(z)}{Q_t(z)}.
\end{align*}
Equating the left hand side to $\dot{R_t}(z)/R_t'(z)$ completes the proof.
\end{proof}

\begin{rmk*}
A standard computation also shows that since $Q_t$ evolves according to the backward heat equation then the meromorphic function $u_t$ having poles at $\zeta_k(t)$, namely
\[
u_t(z) =\frac{Q_t'(z)}{Q_t(z)}=\sum_{k}\frac1{z-\zeta_k(t)},
\]
satisfies the Burgers-Hopf equation $\dot u = -\tfrac{1}{2} u''- uu'$. Regardless of whether one chooses to evolve $Q$ or $u$ the main point is that both evolve autonomously, and the solution can be inserted into the evolution equation for $R_t$.
\end{rmk*}

\subsection{The Hamiltonian Point of View\label{subsec:Hamiltonian}}

Recall that the Calogero-Moser system is also Hamiltonian, with the Hamiltonian $\Ham = \Ham(\bfs x, \bfs p)$ given by \eqref{eq: Ham}. The Hamiltonian is in \textit{normal form}, meaning that it splits into the sum of a kinetic energy (a function of $\bfs p$) and a potential energy (a function of $\bfs x$). The phase space for $(\bfs x, \bfs p)$ is typically taken to be $\{ \bfs x \in \R^{2n} : x_1 < \ldots < x_{2n} \} \times \R^{2n}$, with $\bfs p$ representing the conjugate momenta. Under the standard Hamiltonian equations of motion
\[
\dot{x}_j = \frac{\partial \Ham}{\partial p_j}, \quad \dot{p}_j = - \frac{\partial \Ham}{\partial x_j}
\]
we have $\bfs p = \dot{\bfs x}$. It is straightforward to directly verify that $\dot{\Ham} = 0$ along the Hamiltonian flow (of course this is always true in the Hamiltonian setting), but it is also interesting to note that the Hamiltonian itself is zero along orbits determined by solutions to the null vector equations.

\begin{prop}\label{prop:Hamiltonian_is_zero}
Let the critical points $\bfs x$ evolve via \eqref{eq: x dot} with initial velocities
\[
\dot{x}_j = \frac{1}{4} \left( U_j(\bfs x) + \sum_{k \neq j} \frac{2}{x_j - x_k} \right)
\]
where $U_j$ is a solution to the null vector equation \eqref{eq: NV0} for $j=1,\ldots,2n$. Then $\Ham(\bfs x, \dot{\bfs x}) \equiv 0$.
\end{prop}

Note that the result applies when the $U_j$ are the vector fields for the driving functions corresponding to the $1/4$-superposition of the SLE$(0, x_j, \bfs \rho_j)$ processes and the $\bfs x$ and $\bfs \zeta$ obey the stationary relation. This follows from Theorem \ref{thm: Z}, equation \eqref{eq: U}, and equation \eqref{eq: rhoj} for the $\bfs \rho_j$.

\begin{proof}
Square the equation for $\dot{x}_j$ and use that $U_j$ solves the null vector equations \eqref{eq: NV0} to obtain
\begin{align}\label{eq:Ham_NVE}
\dot{x}_j^2 - \sum_{k \neq j} \frac{\dot{x}_j + \dot{x}_k}{x_j - x_k} + \frac{1}{2} \sum_{k \neq j} \sum_{l \neq k} \frac{1}{(x_j - x_k) (x_j - x_l)} - \sum_{k \neq j} \frac{1}{(x_j - x_k)^2} = 0, \quad j=1,\ldots,2n.
\end{align}
Summing the above over $j$ gives $2 \Ham(\bfs x, \dot{\bfs x}) = 0$, since the second and third terms cancel by symmetry.
\end{proof}

In fact the relation \eqref{eq:Ham_NVE} already provides a proof of Theorem \ref{thm: CM2SLEzero}.

\begin{proof}[Proof of Theorem \ref{thm: CM2SLEzero}]
By invoking any one of \cite{EG02, MTV:Shapiro, EG11, PW20} we know that for each choice of $(\bfs x; \alpha)$ there exists an $R \in \CRR_{n+1}(\bfs x; \alpha)$. Then $U_{\alpha,j}(\bfs x)$ of \eqref{eq: U} exist and, by Theorem \ref{thm: Z}, solve the system of null vector equations \eqref{eq: NV0}. Each $U_{\alpha,j}$, $j=1,\ldots,2n$, is also the formula for the dynamical evolution of the SLE$(0, x_j, \bfs \rho_j)$ process. By Theorem \ref{thm: real_locus} the $1/4$-superposition of these processes generates the real locus of $R$, while by Theorem \ref{thm: CM-main} the second order evolution of the points $\bfs x(t)$ follows the Calogero-Moser dynamics \eqref{eq: CM-main} under this superposition. However, under this superposition the initial velocities $\dot{x}_j(0)$ are specified by \eqref{eq: x dot} which, as in Proposition \ref{prop:Hamiltonian_is_zero}, is the same as
\[
\dot{x}_j = \frac{1}{4} \left( U_{\alpha,j}(\bfs x) + \sum_{k \neq j} \frac{2}{x_j - x_k} \right).
\]
The algebra in equation \eqref{eq:Ham_NVE} completes the argument.
\end{proof}

We have now twice used the algebra in \eqref{eq:Ham_NVE}, but we note that it does not rely on our specific solutions \eqref{eq: U} to the null vector equations. Instead it only assumes that $U_j$ satisfies the null vector equations \eqref{eq: NV0}. We study this idea further in the rest of this section, culminating in the proof of Theorem \ref{thm: CM2SLEzero}. Assume that $\{ (x_j, U_j), j=1,\ldots,2n \}$ are related to $\{(x_j, p_j), j=1,\ldots,2n \}$ via
\begin{align}\label{eq: xdot_u_relation}
p_j = U_j + \sum_{k \neq j} \frac{2}{x_j - x_k}.
\end{align}
Note that we have dropped the factors of $1/4$ appearing in the $1/4$-superposition; this makes the coefficients that appear below somewhat simpler. Then solving for $U_j$ and inserting the result into the left hand side of the null vector equation \eqref{eq: NV0} leads to the identity (already used in \eqref{eq:Ham_NVE})
\begin{align}\label{eq:NVE_for_p}
\frac{1}{2}U_j^2 + \sum_k 2 x_{kj} U_k - \sum_k 6 x_{kj}^2 = \frac{1}{2} p_j^2 - \frac{1}{2} \sum_k (p_j + p_k)x_{jk} + \frac{1}{4} \sum_k \sum_{l \neq k} x_{jk} x_{jl} - \frac{1}{2} \sum_k x_{jk}^2. 
\end{align}
Here we use the notation ~\eqref{eq: Moser's notation} employed by Moser \cite{Moser:CM_integrable} in his derivation of the Lax pair:
\[
x_{jk} = x_{jk}(\bfs x) = \begin{cases}
0, & j = k, \\
(x_j - x_k)^{-1}, & j \neq k.
\end{cases}
\]
To study the null vector equations from the Hamiltonian point of view we make the definition 
\begin{align}\label{eq:Hj}
2 \Ham_j = 2\Ham_j(\bfs x, \bfs p) = p_j^2 - \sum_k (p_j + p_k)x_{jk} + \frac{1}{2} \sum_k \sum_{l \neq k} x_{jk} x_{jl} - \sum_k x_{jk}^2.
\end{align}
We call the collection of $\Ham_j$ the \textbf{null vector Hamiltonians}, and we note that
\[
\sum_{j} \Ham_j = \Ham.
\]
Our next result shows that $\Ham_j$ has a nice interpretation in terms of the Lax pair for the Calogero-Moser system. Recall that the Lax pair is two square matrices $L = L(\bfs x, \bfs p)$ and $M = M(\bfs x, \bfs p)$ each of size $2n \times 2n$, and by \cite{Moser:CM_integrable} the entries are given by
\begin{align}\label{eq:Lax_pair}
L_{jk} = \begin{cases} 
p_j, & j = k, \\
-x_{jk}, & j \neq k, 
\end{cases}
\quad \textrm{ and } \quad
M_{jk} = \begin{cases}
-\sum_{l} x_{jl}^2, & j = k ,\\
x_{jk}^2, & j \neq k.
\end{cases}
\end{align}
This leads to the following representation of $\Ham_j$ in terms of $L^2$.

\begin{lem}\label{lem:NV_Ham_Lax}
Each $\Ham_j$ is the one-half the $j$th row sum of $L^2$, i.e.
\[
\Ham_j = \tfrac{1}{2} \mb{e}_j' L^2 \mb{1},
\]
where $\mb{e}_j'$ is the transpose of the $j$th standard basis vector and $\mb{1}$ is the vector of all ones. Consequently, the $U_j$, $j=1,\ldots,2n$, defined by \eqref{eq: xdot_u_relation} solve the null vector equations \eqref{eq: NV0} for a given $\bfs x$ iff the $\bfs p$ variables satisfy $L^2(\bfs x, \bfs p) \mb{1} = \mb{0}$. 
\end{lem}

\begin{proof}
Write $L = P - X_1$, where $P = P(\bfs p) = \diag(\bfs p)$ is the square matrix with entries of $\bfs p$ along its diagonal, and $X_1 = X_1(\bfs x)$ is the square matrix with entries $(X_1)_{jk} = x_{jk}$. Note that $P$ is symmetric and $X_1$ is anti-symmetric. Then
\[
L^2 = P^2 - P X_1 - X_1 P + X_1^2.
\]
It is straightforward to compute the entries of $P^2 - P X_1 - X_1 P$ and see that they give the first two terms on the right hand side of \eqref{eq:Hj}. For $X_1^2$ we have
\begin{align*}
\mb{e}_j' X_1^2 \mb{1} = \sum_{k} (X_1^2)_{jk} = - \sum_{k} \sum_l x_{jl} x_{kl} &= -\sum_{l} x_{jl}^2 - \sum_{k \neq j} \sum_{l \neq j} x_{jl} x_{kl} \\
&= -\sum_{l} x_{jl}^2 - \frac{1}{2} \sum_{k \neq j} \sum_{l \neq k} (x_{jl} x_{kl} + x_{jk} x_{lk} ) \\
&= - \sum_{l} x_{jl}^2 + \frac{1}{2} \sum_{k \neq j} \sum_{l \neq k} x_{jk} x_{jl}.
\end{align*}
The second-to-last equality simply symmetrized the double sum in $k$ and $l$, while the last equality used the easily verified identity $x_{jk} x_{jl} = -x_{jl} x_{kl} - x_{kj} x_{kl}$ (which holds for $j \neq k$ and $l \neq k$). The concluding statement, that one has solutions $U_j$ to the null vector equations iff one has solutions $\bfs p$ to $L^2(\bfs x, \bfs p) \mb{1} = \mb{0}$, follows exactly from the first result, the algebraic identity \eqref{eq:NVE_for_p}, and the bijection \eqref{eq: xdot_u_relation} between $p_j$ and $U_j$.
\end{proof}

The latter result allows for a study of the null vector equations from a geometric point of view, by considering the submanifold of the phase space given by
\[
\left \{ (\bfs x, \bfs p) \in \R^{2n} \times \R^{2n} : x_1 < \ldots < x_{2n}, \, L^2(\bfs x, \bfs p) \mb{1} = \mb{0} \right \}.
\]
Using the time evolution of the Lax pair we now show that this submanifold is invariant under the Calogero-Moser dynamics.

\begin{prop}\label{prop:Hj_preservation}
For each $c \in \R$, the submanifolds
\[
\left \{ (\bfs x, \bfs p) : \Ham_j(\bfs x, \bfs p) = c \textrm{ for all } j \right \}
\]
are individually preserved under the Hamiltonian system corresponding to the Hamiltonian \eqref{eq: Ham}.
\end{prop}

\begin{proof}
By Lemma \ref{lem:NV_Ham_Lax}, $(\bfs x, \bfs p)$ being on the submanifold is equivalent to $L^2 \mb{1} = c \mb{1}$. We compute the time evolution of this vector from the time derivative of $L^2$, using Moser's result \cite{Moser:CM_integrable} that the Lax pair satisfies the evolution equation
\begin{align}\label{eq:Lax_matrix_evolution}
\dot{L} = [L,M].
\end{align}
Here $[L,M] = LM - ML$ is the standard matrix commutator. The time derivative on the left hand side is with respect to the standard equations of motion determined by the Hamiltonian $\Ham$. Thus
\[
\dot{(L^2)} = L \dot{L} + \dot{L} L = L [L,M] + [L,M] L = [L^2,M],
\]
from which it follows that
\[
\frac{d}{dt} L^2 \mb{1} = L^2 M \mb{1} - M L^2 \mb{1} = - M c \mb{1} = \mb{0}. 
\]
The last line twice used that $M \mb{1} = \mb{0}$, which is apparent from the definition of $M$.
\end{proof}

We do not make any statement about the non-emptiness or structure of these submanifolds. Assuming they have a nice topological structure, their invariance under Hamiltonian flow for \eqref{eq: Ham} is another type of commutation property, which we now show. Recall that to every smooth function $F = F(\bfs x, \bfs p)$ of phase space there is an associated vector field defined by
\[
X_F = \sum_{j=1}^{2n} \frac{\partial F}{\partial p_j} \partial_{x_j} - \sum_{j=1}^{2n} \frac{\partial F}{\partial x_j} \partial_{p_j}.
\]
Given two smooth functions $F = F(\bfs x, \bfs p)$ and $G = G(\bfs x, \bfs p)$ the commutator of their vector fields is 
\[
[X_F, X_G] = X_{\{ F,G \}}
\]
where $\{F,G\}$ is the Poisson bracket defined by
\[
\{F, G \} = \sum_{j=1}^{2n} \frac{\partial F}{\partial p_j} \frac{\partial G}{\partial x_j} - \frac{\partial F}{\partial x_j} \frac{\partial G}{\partial p_j}.
\]
Lengthy and somewhat tedious calculations show that the Hamiltonians $\Ham_j$ of \eqref{eq:Hj} satisfy the Poisson bracket identity
\[
\{ \Ham_j, \Ham_k \} = x_{jk}^{-2} (\Ham_k - \Ham_j)
\]
for all $j,k$. Consequently the vector fields $\Ham_j$ commute along the submanifolds of Proposition \ref{prop:Hj_preservation}.

\section{Examples \label{sec: examples}}

In this section we study several examples of our results, focusing on the cases $n=1, n=2,$ and $n=3$ for which explicit computation is possible. Most of the phenomena present in the higher $n$ case are already present in $n=2$ case. We find explicit formulas for the poles $\zeta_1, \ldots, \zeta_n$ in terms of the critical points $\bfs x$, which we then convert into formulas for the $U_j$ that are purely in terms of the $\bfs x$. This is in contrast to \eqref{eq: U} in Theorem \ref{thm: Z} that represents the solutions in terms of the critical points and the poles, but is only possible because we consider small $n$. For larger $n$ this type of analysis will be much more complicated. We also consider the evolution of $R \circ g_t^{-1}$ of Theorem \ref{parabolic PDE}.

We summarize the results of our examples in Table \ref{tab:examples-table}. To keep the expressions simple we present the results for specific values of the critical points, but explicit formulas for general critical points are worked out in the text. For Examples \ref{eg: circle}, \ref{eg: neighbor}, \ref{eg: rainbow}, and \ref{eg: n=3} we use the canonical rational element to carry out the analysis, which we recall is the unique element of $\CRR_{n+1}(\bfs x; \alpha)$ satisfying the hydrodynamic normalization at infinity. In Example \ref{eg: hyperbola} we consider the $n = 1$ case but for a rational function without a finite pole. Our main results do not describe this case but the analysis can still be carried out. In general, for $R = P/Q$ with $\deg P = n+1$ and $\deg Q = d \leq n+1$ most of our theory still holds, modulo some minor changes to the statements. The primary difference between the cases $d < n$ and $d \in \{ n, n+1 \}$ is that for $d < n$ there are several curves that simultaneously go off to infinity, see Example~\ref{eg: hyperbola}. Example \ref{eg: rainbow} is interesting because $0$ is both a critical point \textit{and} a double pole of the canonical rational function. This type of behavior is not ruled out by any of our theorems, but the example shows that the analysis can be carried out regardless. This emphasizes that our typical assumption of simplicity of the poles is convenient rather than crucial. 

\begin{table}[ht]

\caption{Summary of our examples.}

\begin{tabular}{c|cccc}
& \cellcolor[HTML]{EFEFEF} Critical Points & \cellcolor[HTML]{EFEFEF} Finite Poles & \cellcolor[HTML]{EFEFEF} Rational Function & \cellcolor[HTML]{EFEFEF} Real Locus in $\mathbb{\H}$ \\
\hline
\cellcolor[HTML]{EFEFEF} Example \ref{eg: circle} & $-1,1$ & $0$ & $z + 1/z$ & $x^2 + y^2 = 1$ \\
\cellcolor[HTML]{EFEFEF} Example \ref{eg: hyperbola} & $-1,1$ & none & $z^3 - 3z$ & $y^2 = 3(x^2 - 1)$ \\
\cellcolor[HTML]{EFEFEF} Example \ref{eg: neighbor} & $-3,0,1,2$ & $\pm \sqrt{7/3}$ & $(z^3-3)/(3z^2-7)$ & Fig. \ref{figure: quadruple1} and \eqref{eq: neighbor_locus} \\
\cellcolor[HTML]{EFEFEF} Example \ref{eg: rainbow} & $-3,0,1,2$ & $0,0$ & $(z^3 + 7z - 3)/z^2$ & Fig. \ref{figure: quadruple2} and \eqref{eq: rainbow_locus} \\
\cellcolor[HTML]{EFEFEF} Example \ref{eg: n=3} & $\pm a, \pm 1, \pm 1/a$ & $0,\pm i$ & $z+1/z+(a+1/a)^2\dfrac{z}{1+z^2}$ & Fig. \ref{figure: n=3} and \eqref{eq: n=3} \\
\end{tabular}
\label{tab:examples-table}
\end{table}

\subsection{\texorpdfstring{The $n=1$ case}{The n=1 case}}
We consider the $n=1$ case with critical points $x,y \in \R$.

\begin{eg} \label{eg: circle}
We describe a single curve connecting $x,y$, which turns out to be a semi-circle connecting the two in $\H$. Of course this is already clear from the hyperbolic geodesic property but we show the explicit calculations. From the Loewner equation point of view the semi-circle is really two curves that ``meet in the middle'', one growing from $x$ and the other from $y$. We assume there is a single finite pole $\zeta$, and according to the centroid relation~\eqref{eq: means} we have $\zeta = \frac12(x+y)$. Note this also agrees with the stationary relation \eqref{eq: stationary}. The canonical rational function is
\[
R(z) = z-\frac{(\zeta-x)(\zeta-y)}{z-\zeta}.
\]
By Theorem \ref{thm: Z} the solutions to the null vector equations are
\[
U_1=-6/(x-y), \quad U_2=-6/(y-x).
\]
For simplicity, we set $\mu(t) = \nu(t) \equiv 1/4$ and assume $x \le y$. Then the Loewner dynamics are 
\[
\dot g_t = \frac12 \Big( \frac{1}{g_t-x_t} + \frac{1}{g_t-y_t}\Big), \qquad \dot x_t = -\frac{1}{x_t-y_t}, \quad \dot y_t = \frac{1}{x_t-y_t}.
\]
One can easily check that the Hamiltonian is 
\[
\mathcal{H} = \frac{1}{2}(\dot x_t^2 + \dot y_t^2) - \frac{1}{(x_t-y_t)^2} \equiv 0.
\]
With the specific initial data $x_0=-1,\,y_0=1$, we find 
$x_t, y_t$, and $\zeta_t$ as 
\[
x_t = -\sqrt {1-t},\qquad y_t = \sqrt{1-t}, \qquad \zeta_t \equiv 0, \qquad(0\le t \le 1).
\]
By Lemma \ref{lem: gt_algebraic} we obtain
\[
g_t(z)-\frac{ t-1}{g_t(z)}=z+\frac1z, \qquad \gamma_1(t) = -\sqrt{1-t} + i\sqrt t, \qquad
\gamma_2(t) = \sqrt{1-t} + i\sqrt t, \qquad (0\le t \le 1).
\]
Thus the curves lie on the unit circle and $g_t$ is explicitly given by 
\[
g_t(z) = \frac{1 + z^2 - \sqrt{(1 - z^2)^2 + 4 t z^2 }}{2 z}.
\]
\begin{figure}[ht]
\begin{center}
\includegraphics[width=.33\textwidth]{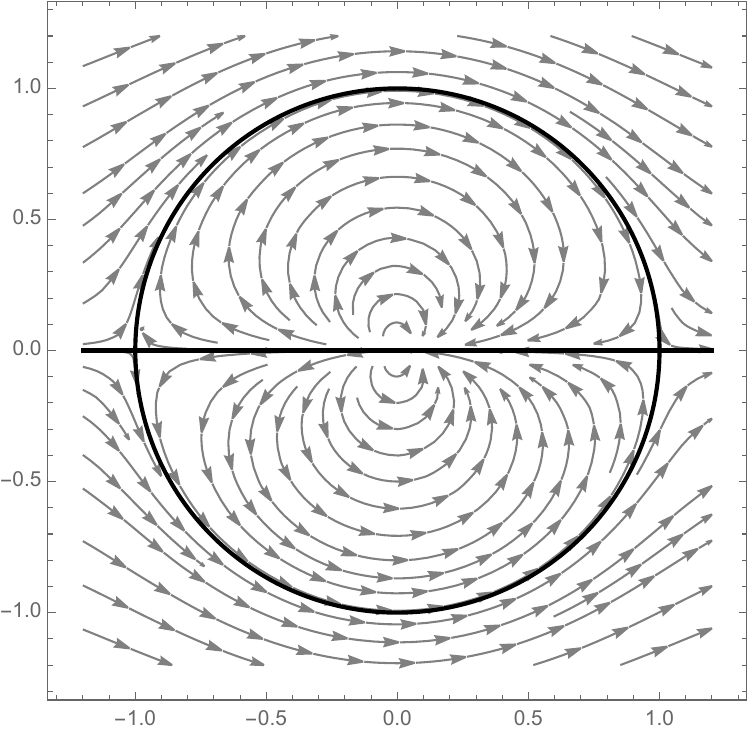}
\hspace{.1\textwidth}
\includegraphics[width=.33\textwidth]{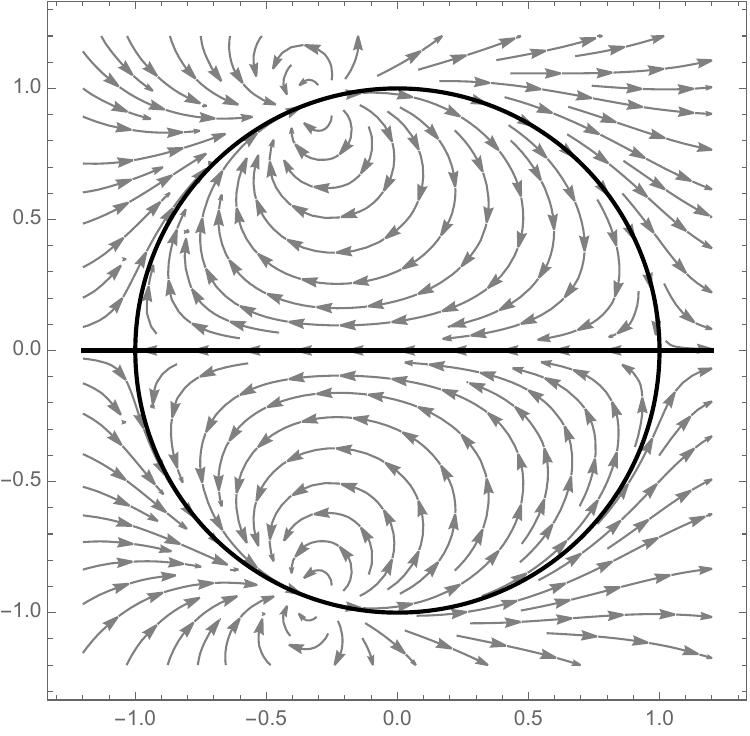} 
\end{center}
\caption{Flow lines of $1/R'$ for Example~\ref{eg: circle}: the left picture is the flow lines for canonical $R$ and the right picture is the flow line for a non-canonical $R$. Arcs on the real locus $\Gamma(R)$, highlighted by the thick dark line, are also flow lines. }
\label{figure: circle}
\end{figure}
The polynomial $Q_t(z) = z$ satisfies the backward heat equation \eqref{eq: backward heat} and the rational function 
\[
R_t(z) = z - \frac{t-1}z
\]
satisfies the following parabolic PDE:
\[
\dot R_t(z) = -\frac12 R_t''(z)- \frac1{z} R_t'(z),
\]
see Theorem~\ref{parabolic PDE}.
\end{eg}

\begin{eg} \label{eg: hyperbola}
We consider the $n=1$ case again with critical points $x,y$, but this time we assume there is no finite pole. Note this does \textit{not} fall into the framework of the results of this paper but we include it anyway. Since there is no pole the corresponding rational function is a polynomial, and we will see that its real locus consists of two curves that start from $x$ and $y$ and merge at infinity. According to Theorem \ref{thm: Z} we set 
\[
U_1=2/(x-y), \quad U_2=2/(y-x).
\]
For simplicity, set $\mu(t) = \nu(t) \equiv 1/4$ and assume $x \le y$. Then the Loewner dynamics are 
\[
\dot g_t = \frac12 \Big( \frac{1}{g_t-x_t} + \frac{1}{g_t-y_t}\Big), \qquad \dot x_t = \frac{1}{x_t-y_t}, \quad \dot y_t = -\frac{1}{x_t-y_t}.
\]
Solving the system of ODEs, we find the solutions $x_t$ and $y_t$ are 
\[
x_t = \frac{x_0+y_0 - \sqrt{4t + (x_0-y_0)^2}}2,\qquad
y_t = \frac{x_0+y_0 + \sqrt{4t + (x_0-y_0)^2}}2.
\]

Let us consider the merging case that $x \equiv x_0 = 0 = y_0 \equiv y$ first. In this case, we have $x_t = -\sqrt{t}$, $y_t = \sqrt{t}$. Applying Lemma \ref{lem: gt_algebraic} we see that $g_t = g_t(z)$ satisfies $g_t^3-3tg_t = z^3$. Therefore the paths $\gamma_1,\gamma_2$ are described by 
\[
\gamma_1(t)= \sqrt[3]{2}\, e^{\frac{2\pi}3i}\,\sqrt{t}, \qquad \gamma_2(t) = \sqrt[3]{2}\,e^{\frac{\pi}3i}\sqrt{t}.
\]
The curves are a pair of symmetric half-lines emanating from the real axis at angles $\pi/3$ and $2 \pi/3$.

\begin{figure}[ht]
\begin{center}
\includegraphics[width=.33\textwidth]{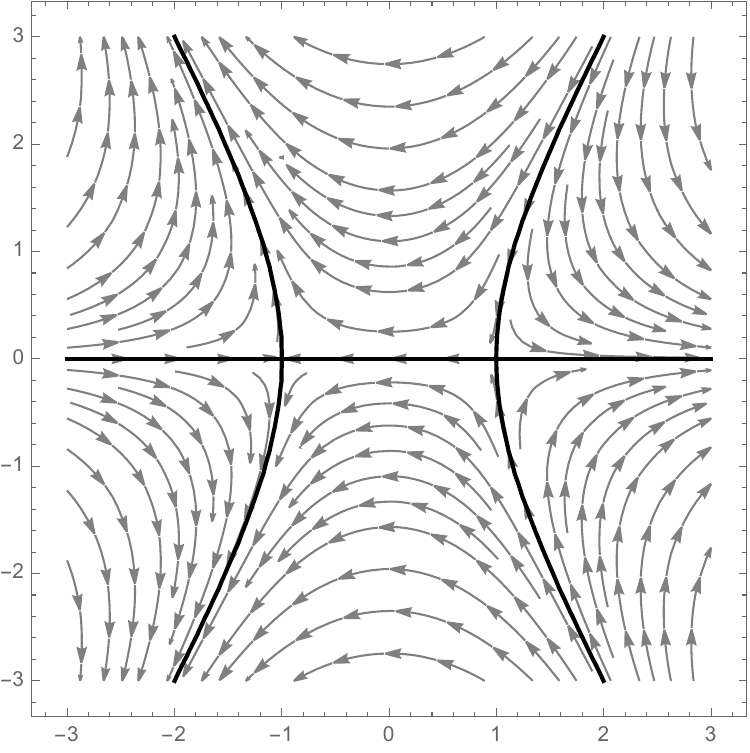}
\hspace{.1\textwidth}
\includegraphics[width=.33\textwidth]{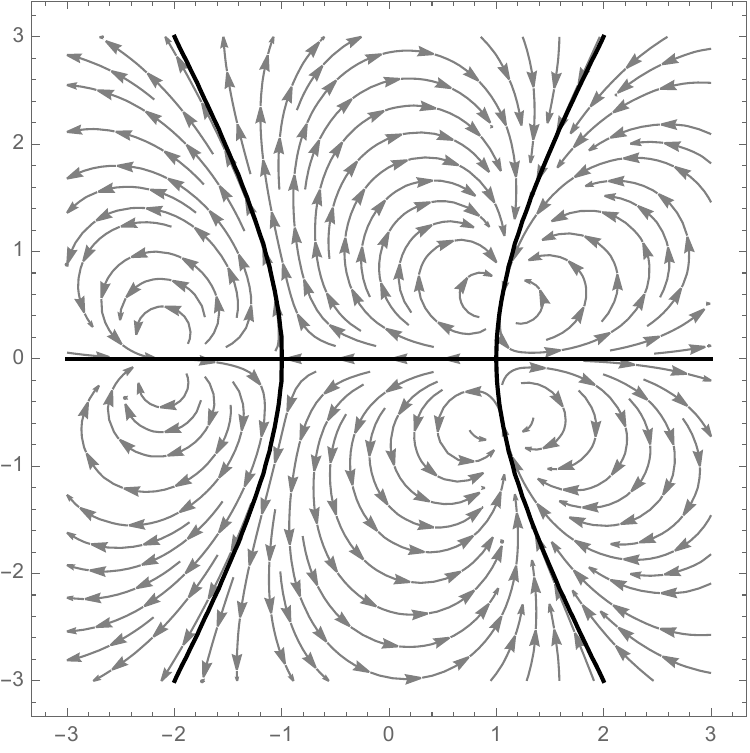} 
\end{center}
\caption{Flow lines of $1/R'$ for Example~\ref{eg: hyperbola}: the left picture is the flow lines for canonical $R$ and the right picture is the flow line for a non-canonical $R$. Arcs on the real locus $\Gamma(R)$ are also flow lines and are highlighted by the thick dark line.}
\label{figure: hyperbola}
\end{figure}

Next, we consider the case that curves start from different points, say $x \equiv x_0=-1, y \equiv y_0=1$. Solving the system of ODEs, we find the solutions $x_t$ and $y_t$ as 
$x_t = -\sqrt{1+t}, y_t = \sqrt{1+t}$. In this case, $z^3-3z$ is an integral of motion and $g_t^3-3(t+1)g_t=z^3-3z$. Furthermore, the SLE(0) paths $\gamma_1,\gamma_2$ are described by 
\[
\gamma_1(t) = -\frac12s+i\frac{\sqrt3}2\sqrt{s^2-4}, \qquad \gamma_2(t) = \frac12s+i\frac{\sqrt3}2\sqrt{s^2-4},
\]
where $s$ is real and $s^3-3s = 2(1+t)\sqrt{1+t}\, (s\ge2)$. The curves lie on the hyperbola
$y^2=3(x^2-1)$.
Both rational functions 
\[
R_t(z) = z^3 - 3tz, \qquad R_t(z) = z^3 - 3(t+1)z
\]
satisfy the backward heat equation:
\[
\dot R_t(z) = -\frac12 R_t''(z),
\]
see Theorem~\ref{parabolic PDE}.
\end{eg}

\begin{rmk*}
The Lax pair associated to the critical points in Examples~ \ref{eg: circle} and \ref{eg: hyperbola} is given by 
\[
L = \begin{pmatrix} \dot x & -\dfrac1{x-y}\\ \dfrac1{x-y} & \dot y \end{pmatrix}, \qquad M = \begin{pmatrix} - \dfrac1{(x-y)^2} & \phantom{-}\dfrac1{(x-y)^2} \\ \phantom{-}\dfrac1{(x-y)^2} & - \dfrac1{(x-y)^2} \end{pmatrix}.
\]
One can verify the equation $\dot L = [L,M]$ directly: 
\[
\dot L = \begin{pmatrix} \ddot x & \dfrac{\dot x-\dot y}{(x-y)^2} \\ -\dfrac{\dot x-\dot y}{(x-y)^2} & \ddot y \end{pmatrix}, \qquad [L,M] = \begin{pmatrix} - \dfrac2{(x-y)^3} & \dfrac{\dot x-\dot y}{(x-y)^2} \\ -\dfrac{\dot x-\dot y}{(x-y)^2} & \dfrac2{(x-y)^3} \end{pmatrix}.
\]
The characteristic polynomial of $L$ is 
$$p_L(\lambda) = \lambda^2 - (\dot x + \dot y)\lambda + \dot x \dot y + \frac1{(x-y)^2}.$$
Thus both $\dot x + \dot y$ and $\dot x \dot y + 1/(x-y)^2$ are integrals of motion. 
\end{rmk*}

\subsection{\texorpdfstring{The $n=2$ case}{The n=2 case}}
We consider the $n=2$ situation with $2$ finite poles. Thus we assume $-\infty < x_1 < x_2 < x_3 < x_4 < \infty$ are the critical points. We will describe the rational function and the real locus for both the \textit{rainbow} pattern ($x_1$ connected to $x_4$ and $x_2$ connected to $x_3$) and the \textit{neighbor} pattern ($x_1$ connected to $x_2$ and $x_3$ connected to $x_4$). As expected, the stationary relation 
\[
\sum_{j=1}^4 \frac1{\zeta_1-x_j} = \frac2{\zeta_1-\zeta_2}, \quad \sum_{j=1}^4 \frac1{\zeta_2-x_j} = -\frac2{\zeta_1-\zeta_2
}
\]
has two pairs of solutions $\{\zeta_1^+,\zeta_2^+\}, \{\zeta_1^-,\zeta_2^-\}$ for the finite poles:
\begin{align}\label{eq: n_two_poles}
\zeta_1^\pm = \frac14\Big(s_1 -\sqrt{s_1^2 - \frac83s_2 \pm \frac83 \sqrt{S} }\Big),\qquad \zeta_2^\pm = \frac14\Big(s_1 + \sqrt{s_1^2 - \frac83s_2 \pm \frac83 \sqrt{S} }\Big),
\end{align}
where 
\begin{align}\label{eq:n_two_S}
S =\Big(\sum_{j<k}\frac2{(x_j-x_k)^2}\Big)\Big/\Big(\sum_j\prod_{k\ne j} \frac1{(x_j-x_k)^2} \Big) = s_2^2-3s_1s_3+12s_4
\end{align}
and 
$s_k$'s are the elementary symmetric functions in the four variables $x_1, x_2, x_3, x_4:$
\[
s_1 = x_1 + x_2 + x_3 + x_4, \quad s_2 = \!\! \sum_{1 \leq j < k \leq 4} \!\! x_j x_k, \quad s_3 = \sum_{j< k < l} x_j x_k x_l, \quad s_4 = x_1 x_2 x_3 x_4.
\]
The plus case corresponds to the neighbor pattern while the minus case corresponds to the rainbow pattern, as we will see shortly. Note that we our solutions only involve two $\zeta$ because we are considering the canonical rational function and so the third pole is at infinity. Inserting $\zeta_1^{\pm}, \zeta_2^{\pm}$ into the system~\eqref{eq: U} to obtain $U_j^{\pm}$, $j=1,2,3,4$, we first claim that the solutions $U_j^{\pm}$ can be entirely expressed in terms of $x_j$'s as 
\begin{equation} \label{eq: Upm}
U_j^\pm = -\sum_{k\ne j}\frac2{x_j-x_k} \mp \frac{4\sqrt{S}}{\prod_{k\ne j}(x_j-x_k)}.
\end{equation}
To see this, let
\[
V_j = \frac{1}{4}\Big(U_j + \sum_{k\ne j}\frac2{x_j-x_k}\Big).
\]
Then, as in \eqref{eq:Ham_NVE}, the null vector equations \eqref{eq: NV0} can be rewritten in terms of $V_j$ as 
\begin{equation} \label{eq: NV04V}
V_j^2-\sum_{k\ne j} \frac{V_j+V_k}{x_j-x_k} + \frac{1}{2} \sum_{k \neq j} \sum_{l \neq k} \frac{1}{(x_j-x_k)(x_j-x_l)} -\sum_{k\ne j}\frac{1}{(x_j-x_k)^2}=0,
\end{equation}
Summing the above over $j=1,\ldots,4$ gives 
\begin{equation} \label{eq: sphere}
\sum_j V_j^2=2\sum_{j<k}\frac1{(x_j-x_k)^2},
\end{equation}
i.e. that $(V_1,\ldots,V_4)$ lies on a sphere. On the other hand, in the case $n=2$, the conformal Ward identities \eqref{eq: CWI0} can be rewritten in terms of $V_j$ as 
\begin{equation}\label{eq: tilde_CWI}
\sum_j V_j = \sum_j x_j V_j = \sum_j x_j^2 V_j = 0. 
\end{equation}
Due to the Vandermonde determinant formula, the intersection of these three hyperplanes is the straight line passing through the origin 
with direction vector
\[
\left( \frac1{\prod_{k\ne1}(x_1-x_k)}, \frac1{\prod_{k\ne2}(x_2-x_k)},\frac1{\prod_{k\ne3}(x_3-x_k)},\frac1{\prod_{k\ne4}(x_4-x_k)} \right).
\]
This line intersects the sphere \eqref{eq: sphere} at two points, so there are at most two solutions $U_j$ to both the null vector equations \eqref{eq: NV0} and the conformal Ward identities \eqref{eq: CWI0}. On the other hand, by Theorems \ref{thm: Z} and \ref{CWI} we already know that \eqref{eq: U} provides at least two solutions to the null vector equations and conformal Ward identities. Thus we have exactly two solutions to \eqref{eq: NV0} and \eqref{eq: CWI0}. It is straightforward to verify that for the proposed solutions \eqref{eq: Upm} the corresponding $V_j^{\pm}$ satisfy the version \eqref{eq: NV04V} of the null vector equations and the conformal Ward identities \eqref{eq: tilde_CWI}, showing that the $U_j^{\pm}$ are in fact equal to the solutions given by Theorem \ref{thm: Z}, up to a determination of the sign which we leave as an exercise for the reader.

Now we consider the dynamics for $x_j$ driven by $U_j^{\pm}$. For simplicity, we set $\nu_j \equiv 1/4$. Then using \eqref{eq: Upm} for $U_j^{\pm}$ we find the dynamics of $x_j$'s to be
\begin{align}\label{eq: n_two_xdot}
\dot x_j = \mp\frac{\sqrt{S}}{\prod_{k\ne j}(x_j-x_k)}.
\end{align}
The following symmetric terms, as the coefficients of the characteristic polynomial of the Lax matrix $L$, are integrals of motion:
\[
\dot x_1 + \dot x_2 + \dot x_3 + \dot x_4, \qquad \sum_{j<k}\Big(\dot x_j \dot x_k + \frac1{(x_j-x_k)^2}\Big),
\]
\[\dot x_1\dot x_2 \dot x_3 + \dot x_1 \Big(\frac1{(x_2-x_3)^2} + \frac1{(x_2-x_4)^2}+\frac1{(x_3-x_4)^2}\Big) + \textrm{ their cyclic terms},
\]
and 
\[
\dot x_1\dot x_2 \dot x_3\dot x_4 + \dot x_1\dot x_2 \frac1{(x_3-x_4)^2} + \frac1{(x_1-x_2)^2(x_3-x_4)^2}+ \textrm{ their cyclic terms}.
\]
Now recall from \eqref{eq:n_two_S} that $S$ is expressible in terms of the elementary symmetric functions $s_1, s_2, s_3, s_4$ in the four variables $x_1, x_2, x_3, x_4$.
From these formulas and \eqref{eq: n_two_xdot} we obtain, by symmetry, 
\begin{align*}
\dot s_1 &= \mp\sqrt{S} \sum_j \frac1{\prod_{k\ne j}(x_j-x_k)}=0, \\
\dot s_2 &= \mp\sqrt{S} \sum_j \frac{\sum_{k\ne j} x_k}{\prod_{k\ne j}(x_j-x_k)}=0, \\
\dot s_3 &= \mp\sqrt{S} \sum_j \frac{(\prod_{k\ne j} x_k)\sum_{k\ne j} 1/x_k}{\prod_{k\ne j}(x_j-x_k)}=0,
\end{align*} 
and
\[
\dot s_4 =\mp\sqrt{S} \sum_j \frac{\prod_{k\ne j} x_k}{\prod_{k\ne j}(x_j-x_k)}= \pm\sqrt{S}.
\]
Now let $R$ generically denote the canonical element (generically meaning that at this point use it for both the canonical element of the rainbow pattern and the neighbor pattern). Now we can describe $\dot g_t$ and $(R \circ g_t^{-1})'$ in terms of $s_j(t)$, obtaining
\begin{equation}\label{eq: g4quadruple}
\dot g_t = \frac12\sum\frac1{g_t-x_j(t)} = \frac12\frac{4g_t^3-3s_1(t)g_t^2+2s_t(t)g_t-s_3(t)}{g_t^4-s_1(t)g_t^3+s_2(t)g_t^2-s_3(t)g_t + s_4(t)}
\end{equation}
and 
\[
(R \circ g_t^{-1})'(z) = \frac{\prod_j(z-x_j(t))}{(z-\zeta_1^\pm(t))^2(z-\zeta_2^\pm(t))^2 } = \frac{z^4-s_1(t)z^3+s_2(t)z^2-s_3(t)z + s_4(t)}{\big(z^2-s_1(t)z/2+(s_2(t)\mp\sqrt{S(t)})/6)\big)^2}.
\]
Due to the translation invariance we may assume that $s_1(0) = 0.$ Then $s_1(t) \equiv 0$ and 
\begin{equation}\label{eq: R4quadruple}
(R \circ g_t^{-1})(z) = \frac{z^3 + z s_4(t)/\sigma(t) + s_3(t)/2}{z^2+\sigma(t)}, \qquad \sigma(t):=\frac{s_2(t)\mp\sqrt{S(t)}}6,
\end{equation}
where $S(t) = s_2(t)^2+12s_4(t)$.

Up to this point the calculations hold for generic critical points $x_1, x_2, x_3, x_4$. To make the remaining calculations simpler we now consider the dynamics of $x_j$, $S$, and $s_j$ with the specific initial data: 
\[
x_1(0) = -3, \quad x_2(0)=0,\quad x_3(0)=1, \quad x_4(0)=2.
\]
Then we have $s_1\equiv0$, $s_2 \equiv -7$, $s_3 \equiv -6$, and
\[
\dot s_4= \pm\sqrt{S}=\pm\sqrt{49+12s_4}.
\]
Solving this ODE, we find $s_4 = 3t^2\pm7t$ and $S(t) = (6t\pm7)^2$. We take up the plus case in the next example and the minus case in Example \ref{eg: rainbow}.

\begin{eg} \label{eg: neighbor}
We first consider the case $s_4 = 3t^2+7t$.
By \eqref{eq: g4quadruple} and \eqref{eq: R4quadruple}, $g_t$ is described by the following ODE: 
\[
\dot g_t = \frac{2g_t^3- 7g_t+3}{g_t^4 - 7g_t^2 + 6g_t + 3 t^2 +7t}
\]
and $R^+(z) = (z^3-3)/(z^2-7/3)$ is its integral of motion:
\[
\frac{g_t^3-3tg_t-3}{g_t^2-t-7/3}=\frac{z^3-3}{z^2-7/3}.
\]
The nontrivial solutions to $R^+(z) = R^+(u)$ are found as 
\[
z =\frac{-9+7 u\pm\sqrt{3} \sqrt{27+42 u-49 u^2-36 u^3+28 u^4}}{2 (-7+3 u^2)}.
\]
Eliminating the parameter $u$, the real locus of $R^+$ in $\mathbb{H}$ is given by 
\begin{align}\label{eq: neighbor_locus}
y=\sqrt{\frac{1}{6} (-7-6 x^2+\sqrt{49+24 x (-9+14 x)})},\qquad(-3\le x \le 0, 1\le x\le 2).
\end{align}
See Figure~\ref{figure: quadruple1} for its graph. We see that the graph is the neighbor pattern $\{ \{x_1, x_2\}, \{x_3, x_4 \} \}$. Furthermore the polynomial $Q_t(z) = 3z^2 -3t-7$ satisfies the backward heat equation \eqref{eq: backward heat} and the rational function 
\[
R_t(z) = \frac{z^3-3tz-3}{3z^2-3t-7}
\]
satisfies the parabolic PDE
\[
\dot R_t(z) = -\frac12 R_t''(z)- \frac{6z}{3z^2-3t-7} R_t'(z).
\]

\begin{figure}[ht]
\begin{center}
\includegraphics[width=.33\textwidth]{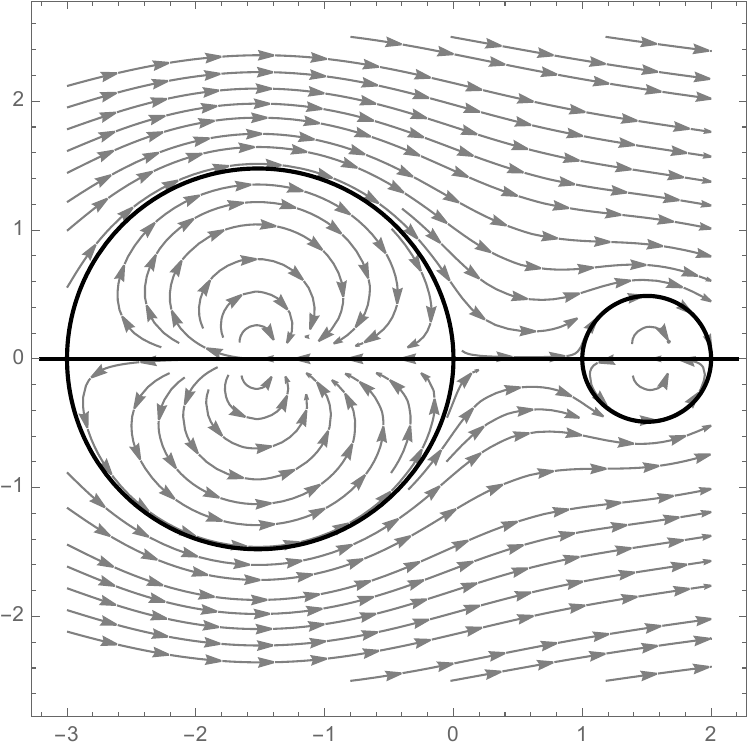}
\hspace{.1\textwidth}
\includegraphics[width=.33\textwidth]{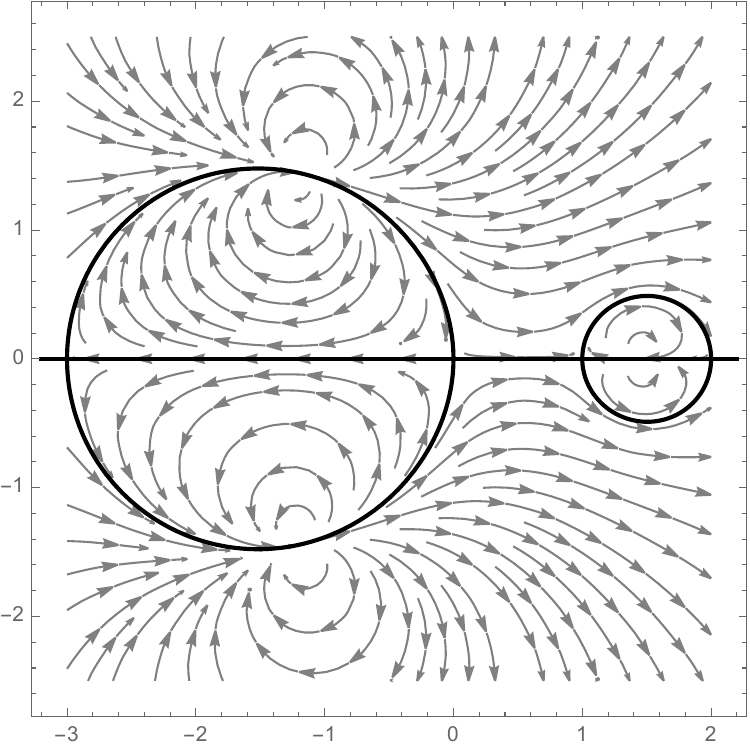} 
\end{center}
\caption{The real locus of $R=R^+$ in Example~\ref{eg: neighbor} and the flows of $1/R'$ for canonical $R$ (on the left) and non-canonical $R$ (on the right)}
\label{figure: quadruple1}
\end{figure}

\end{eg}

\begin{eg} \label{eg: rainbow}
We next consider the case $s_4 = 3t^2-7t$.
By \eqref{eq: g4quadruple} and \eqref{eq: R4quadruple}, $g_t$ is described by the following ODE: 
\[
\dot g_t = \frac{2g_t^3- 7g_t+3}{g_t^4 - 7g_t^2 + 6g_t + 3 t^2 -7t}
\]
and $R^-(z) = (z^3+7z-3)/z^2$ is its integral of motion:
\[
\frac{g_t^3+(7-3t)g_t-3}{g_t^2-t}=\frac{z^3+7z-3}{z^2},\qquad (0\le t \le \frac{7}{6}).
\]
The nontrivial solutions to $R^-(z) = R^-(u)$ are found as 
\[
z=\frac{-3+7 u\pm\sqrt{9-42 u+49 u^2-12 u^3}}{2 u^2}.
\]
Eliminating the parameter $u$, the real locus of $R^-$ in $\mathbb{H}$ is given by 
\begin{align}\label{eq: rainbow_locus}
y = \begin{cases} \sqrt{-x^2+\frac12(7+\sqrt{49-24x}\,)},\qquad (-3\le x \le 2),\\
\sqrt{-x^2+\frac12(7-\sqrt{49-24x}\,)},\qquad (0\le x \le 1).\end{cases}
\end{align}
See Figure~\ref{figure: quadruple2} for its graph. We see that the graph is the rainbow pattern $\{ \{x_1, x_4\}, \{x_2,x_3\} \}$. The polynomial $Q_t(z) = z^2-t$ satisfies the backward heat equation \eqref{eq: backward heat} and the rational function 
\[
R_t(z) = \frac{z^3+(7-3t)z-3}{z^2-t}
\]
satisfies the following parabolic PDE:
\[
\dot R_t(z) = -\frac12 R_t''(z)- \frac{2z}{z^2-t} R_t'(z),
\]

\begin{figure}[ht]
\begin{center}
\includegraphics[width=.33\textwidth]{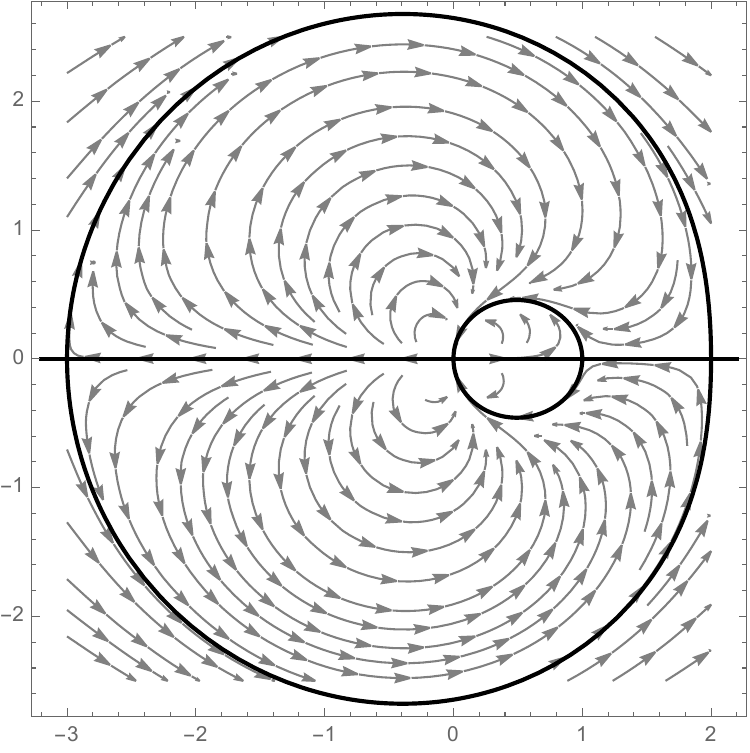}
\hspace{.1\textwidth}
\includegraphics[width=.33\textwidth]{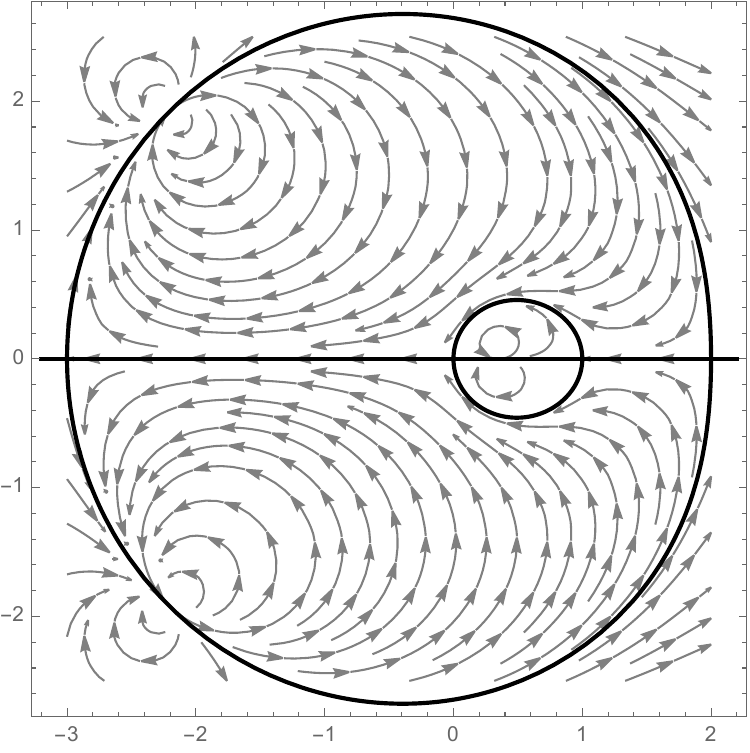} 
\end{center}
\caption{The real locus of $R=R^-$ in Example~\ref{eg: rainbow} and the flows of $1/R'$ for canonical $R$ (on the left) and non-canonical $R$ (on the right).}
\label{figure: quadruple2}
\end{figure}
\end{eg}

\subsection{\texorpdfstring{Verifying $\kappa \to 0$ Limits for $n=2$}{Verifying kappa to 0 Limits for n=2}}
In the $n=2$ case we have the explicit formulas \eqref{eq: n_two_poles} $\zeta_1^{\pm}$ and $\zeta_2^{\pm}$ for the poles. We use them to analyze the functions $Z^{\pm} = Z^{\pm}(x_1, x_2, x_3, x_4)$ of Theorem \ref{thm: Z}. Definition \eqref{eq: Z} of $Z^{\pm}$ tells us
\begin{equation} \label{eq: Zpm}
Z^\pm = \frac{(\zeta_1^\pm-\zeta_2^\pm)^8\prod_{j<k}(x_j-x_k)^2}{\prod_{j} (x_j-\zeta_1^\pm)^4(x_j-\zeta_2^\pm)^4}.
\end{equation}
We now use \eqref{eq: n_two_poles} to simplify $Z^\pm$ (up to a multiplicative constant) as 
\begin{align*}
Z^+ &= \frac{(x_3-x_1)^2(x_4-x_1)^2(x_3-x_2)^2(x_4-x_2)^2 }{(x_2-x_1)^6(x_4-x_3)^6} \Bigg(\frac{H^+_{\{\{1,2\},\{3,4\}\}} }{H^-_{\{\{1,3\},\{2,4\}\}}H^-_{\{\{1,4\},\{2,3\}\}}} \Bigg)^4,
\\
Z^- &= \frac{(x_2-x_1)^2(x_3-x_1)^2(x_4-x_2)^2(x_4-x_3)^2 }{(x_4-x_1)^6(x_3-x_2)^6} \Bigg(\frac{H^-_{\{\{1,4\},\{2,3\}\}} }{H^+_{\{\{1,3\},\{2,4\}\}}H^+_{\{\{1,2\},\{3,4\}\}}} \Bigg)^4,
\end{align*}
where $H_{\{\{j,k\},\{l,m\}\}}^\pm = H_{\{\{j,k\},\{l,m\}\}}^\pm(x_1, x_2, x_3, x_4)$ is a homogeneous function of degree 2 given by 
\[
H_{\{\{j,k\},\{l,m\}\}}^\pm = \pm2\sqrt{S} - (x_j+x_k)(x_l+x_m)+2(x_j x_k + x_l x_m).
\]
For $x_1<x_2<x_3<x_4$, $Z^\pm$ can be rewritten as 
\[
Z^+ = \frac{(1-z)^2}{(x_1-x_2)^6(x_3-x_4)^6}F_0(1-z), \qquad
Z^- = \frac{z^2}{(x_1-x_4)^6(x_2-x_3)^6}F_0(z), 
\]
where $z$ is a cross-ratio 
\[
z = \frac{(x_2-x_1)(x_4-x_3)}{(x_4-x_2)(x_3-x_1)}
\]
and $F_0$ is given by 
\[
F_0(z) = \frac{(1-z)^8}{\big((z+1)(z-1/2)(z-2)+(1-z+z^2)\sqrt{1-z+z^2}\big)^4}
\]
up to a multiplicative constant. 

Now we compare $Z^{\pm}$ with their analogues in the $\kappa > 0$ (but small) case. In this case the analogues of $Z^{\pm}$ are the so-called \textbf{pure partition functions}, which we denote by $\mathcal{Z}^{\pm}_{\kappa}$. When the logarithmic derivative of these partition functions is used as an interaction term in the driving functions of a particular stochastic Loewner system, it is known that it produces multiple SLE$(\kappa)$ curves that connect $x_1, x_2, x_3, x_4$ almost surely according to the neighbor pattern in the $\mathcal{Z}^+_{\kappa}$ case and the rainbow pattern in the $\mathcal{Z}^{-}_{\kappa}$ case. Explicit expressions for $\mathcal{Z}_\kappa^{\pm}$ are known: 
\begin{align*}
\mathcal{Z}_\kappa^+ &= (x_2-x_1)^{-2h}(x_4-x_3)^{-2h} (1-z)^{2/\kappa}F(1-z)/F(1) \\
\mathcal{Z}_\kappa^- &= (x_4-x_1)^{-2h}(x_3-x_2)^{-2h} z^{2/\kappa}F(z)/F(1), 
\end{align*}
where $h=(6-\kappa)/(2\kappa)$ and $F$ is a hypergeometric function:
\[
F(z) = \,_2F_1 \Big(\frac4\kappa,1-\frac4\kappa,\frac8\kappa;z\Big),
\]
see e.g., \cite{PelWu:multiple_SLEs,LK:configurational_measure, Lawler:part_functions_SLE, Lawler:PC}. Peltola and Wang \cite[Corollary 5.12]{PW20} showed that $-\frac1{12}\kappa \log \mathcal{Z}_{\kappa}^\pm$ converges as $\kappa \to 0$ to what they call the \textit{minimal potential}, the minimum value of a certain functional on ensembles of curves. From the relation between the minimal potential and $U_j$'s (see \cite[(1.8)-(1.10)]{PW20}), we have 
\[
\lim_{\kappa\to0}F(z)^\kappa = F_0(z),\qquad \lim_{\kappa\to0}(\mathcal{Z}_\kappa^\pm)^\kappa = Z^\pm
\]
up to a common multiplicative constant. This can be verified directly using the expressions above together with properties of hypergeometric functions.

\subsection{\texorpdfstring{The $n=3$ case}{The n=3 case}} 
We consider the $n=3$ case with $3$ finite poles and a single pole at infinity. Since $C_3 = 5$ there are five distinct solutions to the stationary relation for any $\bfs x$. We consider the particular set of critical points $\bfs x$ given by 
\[x_1 = -1/a, \quad x_2 = -1,\quad  x_3 = -a, \quad x_4 = a, \quad  x_5 =1, \quad x_6=1/a, \quad (0<a<1).
\] 
By restricting to sets $\bfs \zeta = \{ \zeta_1, \zeta_2, \zeta_3 \}$ that are symmetric under both horizontal and vertical reflections we find three distinct solutions to the stationary relation:
\begin{equation} \label{eq: 3R}
\{0,\pm i\},
\end{equation}
\begin{equation} \label{eq: 3N}
\Big\{0, \pm\frac1{\sqrt 6}\sqrt{\big(a+\frac1a\big)^2+2 + \sqrt{\Big(\big(a+\frac1a\big)^2+2\Big)^2-36}}\Big\},
\end{equation}
\begin{equation} \label{eq: 3RN}
\Big\{0, \pm\frac1{\sqrt 6}\sqrt{\big(a+\frac1a\big)^2+2 - \sqrt{\Big(\big(a+\frac1a\big)^2+2\Big)^2-36}}\Big\}. 
\end{equation}
This leaves two of the $C_3 = 5$ link patterns for which we do not provide explicit solutions (the asymmetric patterns). For each of the three pole sets above the canonical rational function is: 
\[
R(z) = z + \frac1{\zeta^4z} - \frac{z}{z^2-\zeta^2} \frac{(\zeta^2-a^2)(\zeta^2-1)(\zeta^2-1/a^2)}{2\zeta^4},
\]
where $\zeta$ is one of the non-zero elements of the set. In the next example we consider the canonical $R$ with poles \eqref{eq: 3R}. 
\begin{figure}[ht]
\begin{center}
\includegraphics[width=.3\textwidth]{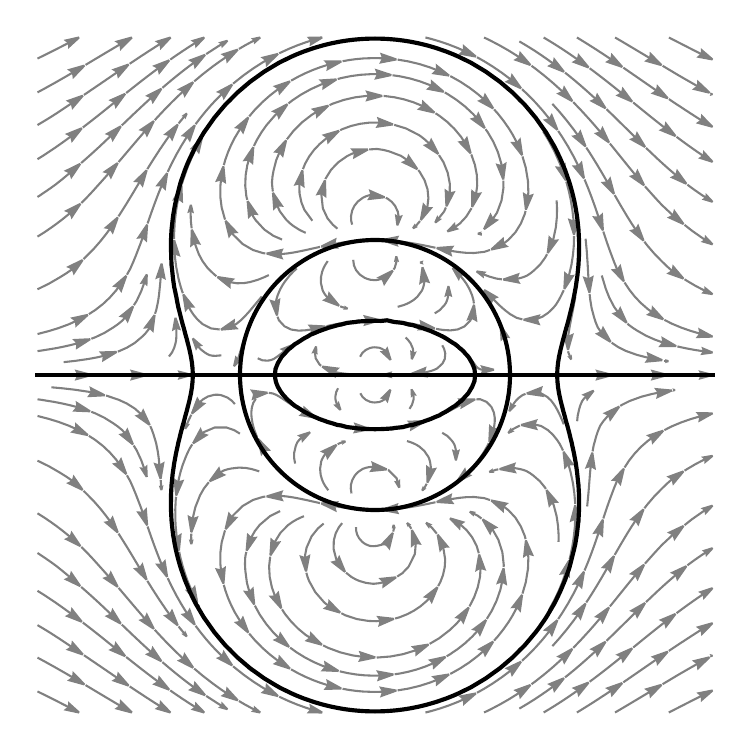}
\hspace{.03\textwidth}
\includegraphics[width=.3\textwidth]{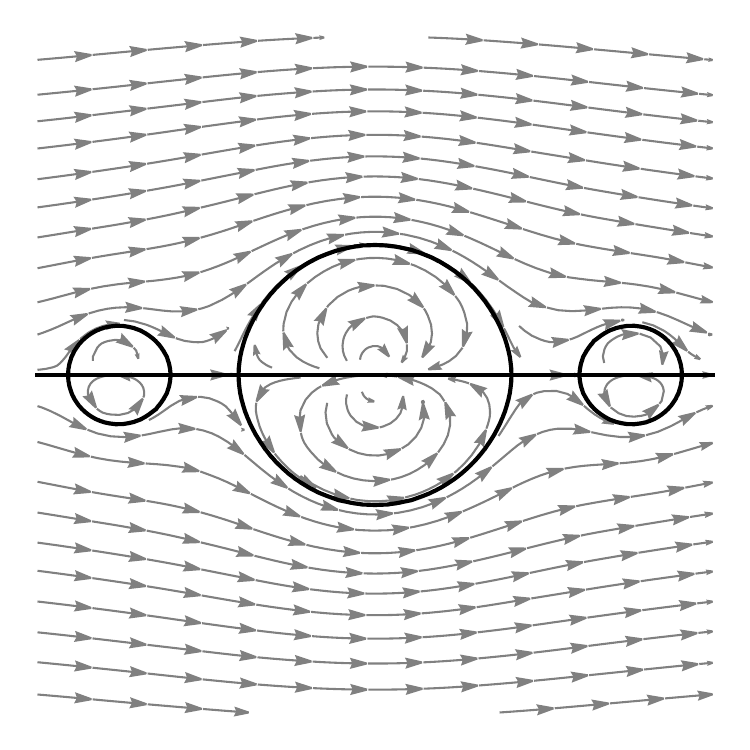}
\hspace{.03\textwidth}
\includegraphics[width=.3\textwidth]{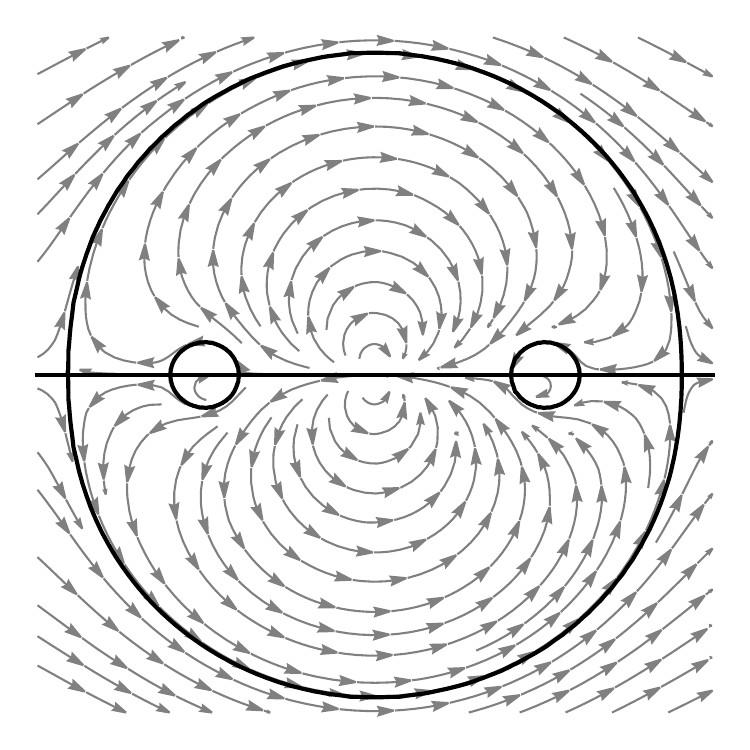} 
\end{center}
\caption{The real locus of $R$ and the flows of $1/R'$ for
canonical $R$ with poles \eqref{eq: 3R} (on the left), 
canonical $R$ with poles \eqref{eq: 3N}  (in the middle), and canonical $R$ with poles \eqref{eq: 3RN} (on the right).}
\label{figure: n=3}
\end{figure}

\begin{eg} \label{eg: n=3}
We consider an example with real locus $\Gamma$ invariant under $z\mapsto z^* = 1/\bar z$, i.e. $\Gamma = \Gamma^*$ or \eqref{eq: 3R}.
The canonical rational function 
\[
R(z) = z + \frac1z + \Big(a + \frac1a\Big)^2 \frac{z}{1+z^2}
\]
has the above prescribed critical points. A simple calculation shows that 
\[
\Im R(z) = (1-|z|^2)\Big(-\frac1{|z|^2}+\Big(a + \frac1a\Big)^2\frac1{|1+z^2|^2}\Big) \Im z .
\]
Thus its real locus in $\H$ is given by 
\begin{equation}\label{eq: n=3}
|z| = 1, \qquad \Big|z + \frac1z\Big| = a + \frac1a.
\end{equation}
For the evolution of $Q_t$ and $R_t$ of Theorem \ref{parabolic PDE} we have that $Q_t(z)=z(z^2+1-3t)$ satisfies the backward heat equation \eqref{eq: backward heat} and the rational function 
\[
R_t(z) = \frac{z^4-2(3t-1)z^2+3t^2-2t+1}{z(z^2+1-3t)}
+\Big(a+\frac{1}{a}\Big)^2 \frac{z^2-t}{z(z^2+1-3t)}  
\]
satisfies the parabolic PDE
\[
\dot R_t(z) = -\frac12 R_t''(z)- \frac{3z^2+1-3t}{z(z^2+1-3t)} R_t'(z).
\]

\end{eg}

\section{Concluding Remarks \label{sec: conclusion}}

The approach of this paper was motivated by our work in the forthcoming \cite{AKM}, in which we make a detailed study of the multiple SLE$(\kappa)$ process in the framework of the conformal field theory of \cite{KM13}. In \cite{AKM} the idea is to represent multi-point correlation functions as expected values of the Gaussian free field (GFF) and other (formal) random fields derived from the GFF by various operations, most notably \textbf{Wick products} and the so-called \textbf{operator product expansion}. An important part of the approach in \cite{AKM, KM13, KM} is that these random fields (even though they are only formally defined in a pointwise manner) come equipped with a rule for how they transform under change of coordinate charts. The transformation rule for each random field is inherited by its correlation function. This is the reason that we emphasize our solutions $Z_{\alpha}$ of \eqref{eq: Z} are differential forms and the $U_{\alpha}$ of \eqref{eq: U} are pre-pre-Schwarzians.

One of the main results of \cite{AKM} is that these correlation functions are martingale observables under the random Loewner evolution that generates the curves. The integral of motion in Theorem \ref{primitive} is a heuristic $\kappa \to 0$ limit of a \textbf{screened bosonic observable}. Briefly, this observable is created by inserting charges into the GFF at specific locations, taking an expectation to obtain a multi-point correlation function, and then integrating out some of the variables from this correlation function. We refer to this operation as the \textbf{method of screening}. In the CFT literature it usually goes by the name of the \textbf{Coulomb gas integrals} or \textbf{Dotsenko-Fateev integrals}, while in the SLE literature it was introduced by Dub\'{e}dat \cite{Dubedat06} under the name \textbf{Euler integrals}. In the language of CFT this method introduces auxiliary ``screening charges'' into the Coulomb gas. In the framework of \cite{AKM} this is accomplished by introducing the following charge distribution on the Riemann sphere:
\begin{align}\label{eq:charge}
\sum_{j=1}^{2n} a \delta_{x_j} - \sum_{k=1}^{n} 2a \delta_{\zeta_k} + 2b \delta_{\infty}, \quad a = \sqrt{2/\kappa}, \, b = \sqrt{\kappa/8} - \sqrt{2/\kappa}. 
\end{align}
Here $x_1, \ldots, x_{2n}$ are real and distinct, and $\zeta_1, \ldots, \zeta_n \in \C$ are distinct. At this point the $\zeta_i$ are arbitrary, although it is convenient to assume that they always appear in complex conjugate pairs. The special charge of $2b$ at infinity maintains the correct M\"{o}bius invariance of the system, which is a consequence of a version of the Gauss-Bonnet theorem (see \cite[Corollary 6.3]{KM:Riemann} for further explanation). Conformal field theory then associates a correlation function to this charge distribution:
\[
\mathcal{Z}_{\kappa}^{\pre}(\bfs x, \bfs \zeta) = \!\!\! \prod_{1 \leq i < j \leq 2n} \!\!\! (x_i - x_j)^{a^2} \!\!\! \prod_{1 \leq i < j \leq n} (\zeta_i - \zeta_j)^{4a^2} \prod_{i=1}^{2n} \prod_{j=1}^n (x_i - \zeta_j)^{-2a^2}.
\]
Importantly, $\mathcal{Z}_{\kappa}^{\pre}(\bfs x, \bfs \zeta)$ is also regarded as a \textit{differential} at the various coordinates rather than as a function. For our purposes this simply means that $Z_{\kappa}^{\pre}$ transforms under the conformal automorphisms of $\H$ as
\[
\mathcal{Z}_{\bfs \kappa}^{\pre}(\bfs x, \bfs \zeta) = \mathcal{Z}_{\bfs \kappa}^{\pre}(\phi(\bfs x), \phi(\bfs \zeta)) \prod_{x_i \in \bfs x} \phi'(x_i)^{\lambda_b(a)} \prod_{\zeta_k \in \bfs \zeta} \phi'(\zeta_k)^{\lambda_b(-2a)}, 
\]
where $\lambda_b : \R \to \R$ is given by $\lambda_b(x) = x^2/2 - xb$. The exponents $\lambda_b$ associated to each of the $\bfs x$ and $\bfs \zeta$ points are called \textit{conformal dimensions}. Moreover, from the explicit form of $\mathcal{Z}_{\kappa}^{\pre}(\bfs x, \bfs \zeta)$ one can derive a system of linear partial differential equations satisfied by $\mathcal{Z}_{\kappa}^{\pre}$ in the $\bfs x$ and $\bfs \zeta$ variables. At the $x_i$ points the PDEs are of second order while at the $\bfs \zeta$ points they are of first order. The conformal dimensions figure prominently in the coefficients of the various terms of the PDE. From a purely algebraic standpoint, the derivation of the system of PDEs is only based on the fact that $\mathcal{Z}_{\kappa}^{\pre}(\bfs x, \bfs \zeta)$ is of ``rational type''. By this we mean that it involves products of differences of the variables but the powers they are raised to are not necessarily integers. At most values of $\kappa$ this rational type function is multi-valued and infinitely ramified, making its analysis somewhat delicate. Ignoring these complications however, the next step in the method of screening is to choose closed contours $\mathcal{C}_1, \ldots, \mathcal{C}_n$ along which we may integrate out the $\bfs \zeta$ variables. The integration procedure is well-defined since the conformal dimension of $\mathcal{Z}_{\kappa}^{\pre}(\bfs x, \bfs \zeta)$ is $1$ at the $\zeta$ points, i.e. since $\lambda_b(-2a) = 1$. This leads to a new function of $\bfs x$ defined by
\begin{align}\label{eq:Zk_x_q}
\mathcal{Z}_{\kappa}(\bfs x) := \oint_{\mathcal{C}_1} \ldots \oint_{\mathcal{C}_n} \mathcal{Z}_{\kappa}^{\pre}(\bfs x, \bfs \zeta) \, d \zeta_n \ldots d \zeta_1.
\end{align}
The remarkable feature of this procedure is that, after the screening, the new functions $\mathcal{Z}_{\kappa}(\bfs x)$ solve the BPZ system \eqref{eq: BPZ} of partial differential equations. The solution implicitly depends on the choice of contours, and by varying this choice one can produce distinct solutions to \eqref{eq: BPZ}. We call solutions generated in this way \textit{screened solutions}. A rigorous analysis of their properties was carried out in the series of papers \cite{FK1, FK2, FK3, FK4} by Flores and Kleban. In \cite{FK3} Flores and Kleban showed that for each fixed $\bfs x$ there are $C_n$ distinct choices of the contours that lead to $C_n$ linearly independent solutions to the system of BPZ equations \eqref{eq: BPZ}, and that these solutions also solve $\kappa > 0$ versions of the conformal Ward identities \eqref{eq: CWI0}. In \cite{FK1, FK2} Flores and Kleban, using completely independent arguments from \cite{FK3}, show that there are at most $C_n$ linearly independent solutions to the system of BPZ equations plus the positive $\kappa$ conformal Ward identities. Hence the combination of their two main results proves that the dimension of the solution space to the system of $2n+3$ linear partial differential equations ($2n$ BPZ equations and $3$ conformal Ward equations) is exactly $C_n$. Other authors have obtained similar results using ideas from quantum groups \cite{KyPel:pure_partitions} or Brownian loop measure \cite{Lawler:part_functions_SLE, Lawler:PC, JL:smoothness}, but the Flores and Kleban approach is most similar to the CFT way of thinking that underlies the present paper. There is one small distinction that we make note of: the $\mathcal{Z}_{\kappa}^{\pre}$  introduced above is the focus of the forthcoming \cite{AKM}, which is a simplified and more symmetrical version of the analogous version that Flores and Kleban start from. 

Our underlying idea is to formally apply Laplace's method/the method of stationary phase to the integrals in these screened solutions. At least heuristically one would expect that most of the contribution to the integrals should come from neighborhoods of the critical points along each contour. The critical points are of the map $\bfs \zeta \mapsto \mathcal{Z}_{\kappa}^{\pre}(\bfs x, \bfs \zeta)$. Since $\mathcal{Z}_{\kappa}^{\pre}(\bfs x, \bfs \zeta)$ is of ``rational type'' and might be expected to degenerate to a rational function in the $\kappa \to 0$ limit, the obvious concentration points for the integrals are the poles of the rational function. Justifying the stationary phase argument seems to be complicated but the heuristic remains correct. This use of stationary phase is another reason why we refer to \eqref{eq: stationary} as a stationary relation. One can also view the functions $Z$ of Theorem \ref{thm: Z} as the appropriate $\kappa \to 0$ limit of the screened solutions $\mathcal{Z}_{\kappa}$ of the BPZ equations.

This same approach also leads to a heuristic explanation of how the flow line description of the real locus, as summarized in Section \ref{sec:locus_geodesic}, also naturally appears as a $\kappa \to 0$ limit of the Miller-Sheffield \textit{imaginary geometry} \cite{MS16:imaginary1,MS16:imaginary2,MS16:imaginary3, MS16:imaginary4}. In conformal field theory random vector fields (although formal and non-rigorous) naturally arise from the so-called \textit{vertex fields}. The vertex field that we consider takes the form
\[
\mathcal{V}^{(\sigma)}(z) = \frac{\mathcal{Z}_{\kappa}(\bfs x, z)}{\mathcal{Z}_{\kappa}(\bfs x)} \, e^{\odot i\sigma \Phi(z)},
\]
where $\mathcal{Z}_{\kappa}(\bfs x)$ is defined by \eqref{eq:Zk_x_q}, $\mathcal{Z}_{\kappa}(\bfs x, z)$ is defined below, $\sigma = 1/2b$ where $b$ is defined as in \eqref{eq:charge} (we explain this choice below), $\Phi$ is the Gaussian free field, and $e^{\odot i\sigma \Phi}$ is the Wick's exponential of ${i\sigma \Phi}$. Neither $\Phi(z)$ nor its Wick exponential make sense pointwise, but as this discussion is only a heuristic we do not concern ourselves with that fact. The Wick exponential term is the only random part of $\mathcal{V}^{(\sigma)}(z)$, while the prefactor is entirely deterministic. The Wick exponential transforms as a $(0,0)$-differential (i.e. as a function) while the choice of $\sigma = 1/2b$ is made so that the prefactor transforms as a vector field at $z$. More precisely, we build the $\mathcal{Z}_{\kappa}(\bfs x, z)$ as a screened correlation function of the charge distributions 
\[
\sum_{j=1}^{2n} a \delta_{x_j} - \sum_{k=1}^{n} 2a \delta_{\zeta_k} + 2b \delta_{\infty} + \sigma \delta_z - \sigma \delta_{\overline{z}}, \quad \sigma = \frac{1}{2b}.
\]
The pre-screened correlation function that conformal field theory associates to this charge distribution is
\begin{align*}
\frac{\mathcal{Z}_{\kappa}^{\pre}(\bfs x, \bfs \zeta, z)}{\mathcal{Z}_{\kappa}^{\pre}(\bfs x, \bfs \zeta)} &= (z - \overline{z})^{-\sigma^2} \prod_{j=1}^{2n} (z - x_j)^{a \sigma} (\overline{z} - x_j)^{-a \sigma} \prod_{k=1}^n (z - \zeta_k)^{-2a \sigma} (\overline{z} - \zeta_k)^{2a \sigma} \\
&= (z - \overline{z})^{-\sigma^2} \frac{\prod_j (z-x_j)^{2 a \sigma}}{\prod_{k} (z-\zeta_k)^{2a\sigma}(z-\overline{\zeta_k})^{2a\sigma}} \frac{\prod_k |z-\overline{\zeta_k}|^{4a \sigma}}{\prod_j |z-x_j|^{2a \sigma}}, 
\end{align*}
where the last equality uses the identities $w^{-1} \overline{w} = w^{-2} |w|^{2}$ and $2 b \sigma = 1$. Note that $\mathcal{Z}_{\kappa}^{\pre}(\bfs x, \bfs \zeta, z)$ is complex-valued, and at $z$ it has conformal dimensions $(\lambda_b(\sigma), \lambda_b^*(\sigma))$, where $\lambda_b^*(x) = x^2/2 + xb$. The important property of these dimensions is that $\lambda_b(\sigma) - \lambda_b^*(\sigma) = -2b\sigma = -1$, which means that $\mathcal{Z}_{\kappa}^{\pre}(\bfs x, \bfs \zeta, z)$ transforms as a vector field at $z$. Then we define $\mathcal{Z}_{\kappa}(\bfs x, z)$ to be the result of screening out the $\bfs \zeta$ variables, i.e. as in \eqref{eq:Zk_x_q} we define
\[
\mathcal{Z}_{\kappa}(\bfs x, z) := \oint_{\mathcal{C}_1} \ldots \oint_{\mathcal{C}_n} \mathcal{Z}_{\kappa}^{\pre}(\bfs x, \bfs \zeta, z) \, d \zeta_n \ldots d \zeta_1.
\]
Then $\mathcal{Z}_{\kappa}(\bfs x, z)$ remains a vector field in the $z$ variable, and as $\kappa \to 0$ we expect that steepest descent takes over and the integrals concentrate along the poles of a rational function. In other words, we expect that 
\[
\frac{\mathcal{Z}_{\kappa}(\bfs x, z)^\kappa}{\mathcal{Z}_{\kappa}(\bfs x)^\kappa} \sim  \frac{\mathcal{Z}_{\kappa}^{\pre}(\bfs x, \bfs \zeta_R(\bfs x), z)}{\mathcal{Z}_{\kappa}^{\pre}(\bfs x, \bfs \zeta_R(\bfs x))} \quad \textrm{as } \kappa \to 0
\]
where $\bfs \zeta_R(\bfs x)$ is the pole set of some $R \in \CRR_{n+1}(\bfs x)$. Since $\sigma = 1/2b \to 0$ and $a\sigma \sim -b\sigma$ as $\kappa \to 0$ this suggests that
\[
\frac{\mathcal{Z}_{\kappa}^{\pre}(\bfs x, \bfs \zeta_R(\bfs x),  z)}{\mathcal{Z}_{\kappa}^{\pre}(\bfs x, \bfs \zeta_R(\bfs x))} \sim  \frac {\prod_k (z-\zeta_k)(z-\overline{\zeta_k})}{\prod_j (z-x_j)} \frac{\prod_j |z - x_j|}{\prod_k |z - \overline{\zeta_k}|^2} = \frac{|R'(z)|}{R'(z)} \quad \textrm{as } \kappa \to 0.
\]
The limit on the right hand side is a locally scaled version of the vector field $v_R = 1/R'(z)$, and therefore has the same flow lines as $v_R$ itself. Thus we are led to the conclusion that the flow lines of $\mathcal{V}^{(\sigma)}(z)$, which are the key object of study in the Miller-Sheffield imaginary geometry, become the flow lines of $1/R'(z)$ in the $\kappa \to 0$ limit.

Given Flores and Kleban's results on the dimension of the solution space for the BPZ equations \eqref{eq: BPZ}, it is interesting to ask about the structure of the solution set to the null vector equations \eqref{eq: NV0}. More precisely we consider both systems together with the conformal Ward identities. For BPZ this means a system of $2n+3$ linear PDEs, while for the null vector equations it means the system of $2n$ quadratic equations \eqref{eq: NV0} plus the $3$ linear equations \eqref{eq: CWI0}. Here ``quadratic'' and ``linear'' are with respect to the $U_j$ variables, since we prefer to think of the $U_j$ as unknown variables with the $\bfs x = \{ x_1, \ldots, x_{2n} \}$ regarded as known parameters. For each fixed choice of $\bfs x$ we expect that there are exactly $C_n$ different solutions $(U_j : j=1,\ldots,2n)$ to the null vector + conformal Ward system. However this solution set should be discrete, in contrast to the BPZ + conformal Ward system where the solution space is linear of dimension $C_n$. That the solution space of the null vector + conformal Ward system has at least $C_n$ points is already proved in \cite{PW20}, who construct solutions for each of the $C_n$ possible link patterns. It also follows from our Theorem \ref{thm: Z} that the $U_{\alpha,j}$ of \eqref{eq: U} satisfy the system, \textit{if} we use the existence of $C_n$ distinct equivalence class in $\CRR_{n+1}(\bfs x)$ coming from any one of \cite{EG02, MTV:Shapiro, EG11, PW20}. However it does not seem to follow directly from these results that the null vector + conformal Ward system cannot contain \textit{more} than $C_n$ solutions. Likely it is true that from any solution $(U_j : j=1,\ldots,2n)$ one could construct a real rational function of degree $n+1$ and with critical points precisely at $\bfs x$, which would in turn imply that it must be one of the $C_n$ solution points already given. We have not been able to obtain this construction though.

A related problem is the enumeration of the non-equivalent solutions to the stationary relation \eqref{eq: stationary}, for each fixed collection of $\bfs x$ points. By Theorem \ref{thm: stationary} each solution set $\bfs \zeta$ of the stationary relation (that additionally satisfies $\bfs x \cap \bfs \zeta = \emptyset$) is the pole set of some $R \in \CRR_{n+1}(\bfs x)$. Two sets of solutions $(\zeta_1, \ldots, \zeta_{n+1})$ and $(\zeta_1^o, \ldots, \zeta_{n+1}^o)$ are equivalent if one is the pole set of a post-composition of the rational function associated to the other solution set. The existence of $C_n$ distinct equivalence classes in $\CRR_{n+1}(\bfs x)$ proves that there are at least $C_n$ non-equivalent solutions to the stationary relation. The matching upper bound also follows from Theorem \ref{thm: stationary} and Goldberg's upper bound \cite{Goldberg91}. However it seems that the stationary relation plays at most an indirect role in the statements or results of \cite{Goldberg91, EG02, MTV:Shapiro, EG11, PW20}. See \cite[eqn.~(2.2)]{MTV:Shapiro} for one appearance, although the distinction between poles and critical points is not so apparent there. It would be interesting to have an enumeration method that starts directly from \eqref{eq: stationary}.  

We expect that there are many interesting connections between Calogero-Moser and SLE type processes still to be discovered. In the physics literature Cardy and Doyon \cite{Cardy04, CD07} suggested a connection between the \textit{quantum} Calogero-Moser-Sutherland system and multiple \textit{radial} SLE$(\kappa)$ curves. Interestingly, in one of his last works Feynman \cite{POLY2019} also studied the Calogero-Moser system by taking classical limits of constructions originating in conformal field theory.

\providecommand{\bysame}{\leavevmode\hbox to3em{\hrulefill}\thinspace}
\providecommand{\MR}{\relax\ifhmode\unskip\space\fi MR }
\providecommand{\MRhref}[2]{%
  \href{http://www.ams.org/mathscinet-getitem?mr=#1}{#2}
}
\providecommand{\href}[2]{#2}


\end{document}